\documentclass[12pt]{amsart}
\usepackage{amscd,amssymb,amsthm,amsmath,amssymb,textcomp,supertabular,longtable,enumerate,rotating,mathrsfs}
\usepackage[matrix,arrow,curve]{xy}
\usepackage[usenames,dvipsnames]{color}
\usepackage[table]{xcolor}
\usepackage{stmaryrd}
\usepackage{hyperref}
\usepackage{mathrsfs}

\usepackage{tikz}

\newtheorem{theorem}[equation]{Theorem}

\newtheorem{proposition}[equation]{Proposition}
\newtheorem{lemma}[equation]{Lemma}
\newtheorem{corollary}[equation]{Corollary}
\newtheorem{conjecture}[equation]{Conjecture}

\theoremstyle{definition}
\newtheorem{example}[equation]{Example}

\theoremstyle{remark}
\newtheorem{remark}[equation]{Remark}
\newtheorem{notation}[equation]{Notation}

\makeatletter\@addtoreset{equation}{section} \makeatother

\newcommand{\CC}{\mathbb{C}}
\renewcommand{\P}{\mathbb{P}}
\newcommand{\Q}{\mathbb{Q}}
\newcommand{\WW}{\mathrm{W}}
\newcommand{\ZZ}{\mathbb{Z}}

\newcommand{\fS}{\mathfrak{S}}
\newcommand{\fA}{\mathfrak{A}}
\newcommand{\PSp}{\mathrm{PSp}_4(\mathbf{F}_3)}

\newcommand{\mumu}{\boldsymbol{\mu}}
\newcommand{\cE}{\mathscr{E}}

\newcommand{\cO}{\mathscr{O}}
\newcommand{\cU}{\mathscr{U}}
\newcommand{\cX}{\mathscr{X}}
\newcommand{\cY}{\mathscr{Y}}
\newcommand{\cN}{\mathcal{N}}
\newcommand{\cT}{\mathscr{T}}

\newcommand{\cH}{\mathscr{H}}

\newcommand{\bcM}{\overline{\mathcal{M}}}

\newcommand{\bP}{\mathbf{P}}

\newcommand{\bs}{\mathbf{s}}
\newcommand{\bp}{\mathbf{p}}

\newcommand{\tS}{\tilde{S}}

\newcommand{\CR}{\mathrm{CR}}
\newcommand{\PGL}{\mathrm{PGL}}
\newcommand{\rC}{\mathrm{C}}
\newcommand{\rR}{\mathrm{R}}

\newcommand{\oP}{\overline{\Phi}}

\def\mumu{\boldsymbol{\mu}}
\def\SSZ{{\fS_{3,3}\rtimes\mumu_2}}
\def\CCZ{{\left(\mumu_3\times\mumu_3\right)\rtimes\mumu_2}}
\def\CCZz{{\left(\mumu_3\times\mumu_3\right)\rtimes\mumu_4}}
\def\DD{\mathrm{D}}
\def\VV{\mathrm{V}}
\def\Karz{\mumu_5\rtimes\mumu_4}

\def\le{\leqslant}
\def\vstrut{\vphantom{$A_{j_j}$}}

\DeclareMathOperator{\Sing}{{\mathrm{Sing}}}

\DeclareMathOperator{\rk}{{\mathrm{rk}}}
\DeclareMathOperator{\Pic}{{\mathrm{Pic}}}
\DeclareMathOperator{\VCl}{{\mathrm{ExCl}}}
\DeclareMathOperator{\Cl}{{\mathrm{Cl}}}
\DeclareMathOperator{\rkPic}{{\mathrm{\rk\,Pic}}}
\DeclareMathOperator{\rkCl}{{\mathrm{\rk\,Cl}}}
\DeclareMathOperator{\Gr}{{\mathrm{Gr}}}
\DeclareMathOperator{\Fl}{{\mathrm{Fl}}}
\DeclareMathOperator{\Aut}{{\mathrm{Aut}}}
\DeclareMathOperator{\Out}{{\mathrm{Out}}}
\DeclareMathOperator{\Bl}{\mathrm{Bl}}
\DeclareMathOperator{\Ker}{{\mathrm{Ker}}}

\DeclareMathOperator{\Iso}{{\mathrm{Iso}}}
\DeclareMathOperator{\Ind}{{\mathrm{Ind}}}

\newcommand{\bDD}{\mathfrak{D}}
\newcommand{\bDDtau}{\widehat{\mathfrak{D}}}

\newenvironment{xsmallmatrix}
  {\renewcommand\thickspace{\kern1em}\smallmatrix}
  {\endsmallmatrix}

\author{Ivan Cheltsov, Alexander Kuznetsov, and Constantin Shramov}

\title[Coble fourfold, invariant quartics, and Wiman--Edge sextics]{Coble fourfold, $\fS_6$-invariant quartic threefolds,\\[1ex]and Wiman--Edge sextics}

\dedicatory{To Arnaud Beauville, on the occasion of his 70th birthday}

\address{\emph{Ivan Cheltsov}
\newline
\textnormal{School of Mathematics, The University of Edinburgh,  Edinburgh EH9 3JZ, UK.}
\newline
\textnormal{Laboratory of Algebraic Geometry, National Research University Higher School of Economics, Russia.}
\newline
\textnormal{\texttt{I.Cheltsov@ed.ac.uk}}}

\address{\emph{Alexander Kuznetsov}
\newline
\textnormal{Algebraic Geometry Section, Steklov Mathematical Institute of Russian Academy of Sciences,}
\newline
\hspace{.5\textwidth}\textnormal{8 Gubkina street, Moscow 119991, Russia.}
\newline
\textnormal{Laboratory of Algebraic Geometry, National Research University Higher School of Economics, Russia.}
\newline
\textnormal{Interdisciplinary Scientific Center J.-V.~Poncelet, Moscow, Russia.}
\newline
\textnormal{\texttt{akuznet@mi-ras.ru}}}

\address{\emph{Constantin Shramov}
\newline
\textnormal{Algebraic Geometry Section, Steklov Mathematical Institute of Russian Academy of Sciences,}
\newline
\hspace{.5\textwidth}\textnormal{8 Gubkina street, Moscow 119991, Russia.}
\newline
\textnormal{Laboratory of Algebraic Geometry, National Research University Higher School of Economics, Russia.}
\newline
\textnormal{\texttt{costya.shramov@gmail.com}}}

\begin{document}

\begin{abstract}
We construct two small resolutions of singularities of the Coble fourfold
(the double cover of the four-dimensional projective space branched over the Igusa quartic).
We use them to show that all $\fS_6$-invariant three-dimensional quartics are birational
to conic bundles over the quintic del Pezzo surface with the discriminant curves from the Wiman--Edge pencil.
As an application, we check that $\fS_6$-invariant three-dimensional quartics are unirational,
obtain new proofs of rationality of four special quartics among them and irrationality of the others,
and describe their Weil divisor class groups as $\fS_6$-representations.
\end{abstract}

\maketitle

\tableofcontents

\section{Introduction}
\label{section:intro}

Consider the projectivization $\P^5$ of the standard permutation representation of the symmetric group~$\fS_6$
over an algebraically closed field $\Bbbk$ of characteristic zero,
and the invariant hyperplane~$\P^4$ given by the equation
\begin{equation}\label{eq:hyperplane}
x_1 + x_2 + x_3 + x_4 + x_5 + x_6 = 0
\end{equation}
therein, where $x_1,\ldots,x_6$ are homogeneous coordinates in $\P^5$.
Consider the classical family of $\fS_6$-invariant quartics $X_t$, $t\in\Bbbk\cup\{\infty\}$,
in this hyperplane defined by the equations
\begin{equation}\label{eq:beauville}
(x_1^4 + x_2^4 + x_3^4 + x_4^4 + x_5^4 + x_6^4) - t(x_1^2 + x_2^2 + x_3^2 + x_4^2 + x_5^2 + x_6^2)^2 = 0,
\end{equation}
cf.~\cite{Beauville}.
Every $\fS_6$-invariant quartic in $\P^4$ is one of the quartics~$X_t$; moreover, most of these quartics have automorphism groups isomorphic to $\fS_6$,
and every quartic threefold with a faithful $\fS_6$-action is isomorphic to some~$X_t$ (see Lemma~\ref{lemma:quartics-Aut}).
We refer to these quartics as {\sf $\fS_6$-invariant quartics}.

Every quartic $X_t$ is singular along a certain $30$-point orbit $\Sigma_{30} \subset \P^4$ of the group $\fS_6$ (see~\S\ref{subsection:xt}),
and $\Sigma_{30}$ coincides with $\Sing(X_t)$ unless $t = \infty$ or $t$ is in the finite {\sf discriminant set}
\begin{equation}
\label{def:dd}
\bDD := \left\{ \frac14, \frac12, \frac16, \frac7{10} \right\}.
\end{equation}
For these special values of $t$ the singular locus of $X_t$ is even larger (see Theorem~\ref{theorem:Geer} for its detailed description).

The quartic $X_{1/4}$ that corresponds to the parameter $t = 1/4$ is particularly interesting.
Its equation can be written as
\begin{equation}\label{eq:igusa}
\Big(x_1^4 + x_2^4 + x_3^4 + x_4^4 + x_5^4 + x_6^4\Big) - \frac14\Big(x_1^2 + x_2^2 + x_3^2 + x_4^2 + x_5^2 + x_6^2\Big)^2 = 0
\end{equation}
inside the hyperplane~\eqref{eq:hyperplane}.
It is called the {\sf Igusa quartic}.
The Igusa quartic is singular along a union of 15 lines (that itself forms an intersting configuration~$\CR$, called the {\sf Cremona--Richmond configuration}).
In this sense, $X_{1/4}$ is the most singular of all~$\fS_6$-invariant quartics, except for $X_\infty$
(which is a double quadric, i.e., a quadric with an everywhere non-reduced scheme structure).

The quartic $X_{1/2}$ is known as the {\sf Burkhardt quartic}.
It has the largest symmetry group among the other quartics in this family
(with the exception of~$X_{\infty}$), see~\cite{Coble-PSp} and Lemma~\ref{lemma:quartics-Aut}.
It also has many other interesting properties, see for instance~\cite{Todd36}, \cite{JSBV}, and~\cite[\S5]{Hunt}.

The quartics $X_{1/6}$ and $X_{7/10}$ have been studied in~\cite{CheltsovShramov}, cf.~\cite{Todd33,Todd35,CheltsovShramov-A6}.

The double cover of $\P^4$ branched over the Igusa quartic is called the {\sf Coble fourfold}.
We denote it by $\cY$ and write
\begin{equation*}
\pi \colon \cY \to \P^4
\end{equation*}
for the double covering morphism.
The Coble fourfold can be written as
a complete intersection in the weighted projective space $\P(2,1^6)$ of the hyperplane~\eqref{eq:hyperplane} with the hypersurface
\begin{equation}\label{eq:coble}
x_0^2 = \Big(x_1^4 + x_2^4 + x_3^4 + x_4^4 + x_5^4 + x_6^4\Big) - \frac{1}{4}\Big(x_1^2 + x_2^2 + x_3^2 + x_4^2 + x_5^2 + x_6^2\Big)^2,
\end{equation}
where $x_0$ is the coordinate of weight $2$.
The Coble fourfold $\cY$ is singular along the Cremona--Richmond configuration $\CR$, because so is the Igusa quartic.
Moreover, it has a big group of symmetries:
it carries an action of the symmetric group $\fS_6$ by permutation of coordinates
\begin{equation}\label{eq:action-standard}
g \cdot (x_0 : x_1 : x_2 : x_3 : x_4 : x_5 : x_6) := (x_0 : x_{g(1)} : x_{g(2)} : x_{g(3)} : x_{g(4)} : x_{g(5)} : x_{g(6)}),
\end{equation}
and also the Galois involution $\sigma \colon \cY \to \cY$ of the double cover
\begin{equation}
\label{eq:sigma}
\sigma(x_0 : x_1 : x_2 : x_3 : x_4 : x_5 : x_6) := (-x_0 : x_1 : x_2 : x_3 : x_4 : x_5 : x_6),
\end{equation}
commuting with the symmetric group action.
One can check (see Corollary~\ref{corollary:aut-y}) that they generate the whole automorphism group
\begin{equation*}
\Aut(\cY) \cong \fS_6 \times \mumu_2,
\end{equation*}
where $\mumu_2$ denotes the group of order $2$.
Sometimes it is convenient to twist the action of the symmetric group by the Galois involution.
The obtained action
\begin{equation}\label{eq:action-twisted}
g \diamond (x_0 : x_1 : x_2 : x_3 : x_4 : x_5 : x_6) := (\epsilon(g)x_0 : x_{g(1)} : x_{g(2)} : x_{g(3)} : x_{g(4)} : x_{g(5)} : x_{g(6)}),
\end{equation}
where $g \in \fS_6$
and $\epsilon(g)$ is the sign of the permutation $g$,
is called {\sf the twisted action}.
In contrast, the action~\eqref{eq:action-standard} is called {\sf the natural action}.
It is important not to confuse between these two actions,
so we strongly recommend the reader to keep an eye on them.
Note however, that the actions agree on the alternating group $\fA_6 \subset \fS_6$.
Similarly, if~$G$ is a subgroup of~$\fS_6$,
by the natural and the twisted action of~$G$
on~$\cY$ we mean the restrictions to~$G$ of the natural and the twisted actions of~$\fS_6$, respectively.

Recall that the group $\fS_6$ has outer automorphisms (in fact, the group $\Out(\fS_6)$ is of order~2,
see for instance~\cite{HowardMillsonSnowdenVakil})
characterized by the property that they take a transposition in~$\fS_6$ to a permutation of cycle type~\mbox{$[2,2,2]$};
see Lemma~\ref{lemma:restriction-s6-s5} for other information about outer automorphisms.
If the image of a subgroup~\mbox{$G \subset \fS_6$} under an outer automorphism is not conjugate to~$G$, we call this image a {\sf non-standard} embedding of~$G$.
For instance, we have non-standard embeddings of $\fS_5$, $\fA_5$, $\fS_4 \times \fS_2$, etc.

The first main result of this paper is a construction of two small resolutions of singularities of the Coble fourfold
that are equivariant with respect to maximal proper subgroups of~$\fS_6$;
note that the rank of the $\fS_6$-invariant Weil divisor class group of $\cY$
(with respect both to the natural and the twisted action of~$\fS_6$) equals~$1$, see Corollary~\ref{corollary:rk-1},
hence there are no small resolutions of singularities of~$\cY$ equivariant with respect to the entire group~$\fS_6$.
The varieties~$\cY_{4,2}$ and $\cY_{5,1}$ discussed below already appeared in \cite{FarkasVerra} in a slightly different context.
A smooth quintic del Pezzo surface $S$ is unique up to isomorphism,
and $\Aut(S) \cong \fS_5$, see for instance~\mbox{\cite[\S8.5]{Dolgachev-CAG}}; we fix such an isomorphism.

\begin{theorem}\label{theorem:coble}
Consider the twisted $\fS_6$-action~\eqref{eq:action-twisted} on the Coble fourfold $\cY$.
\begin{itemize}
\item[(i)]
For every non-standard embedding $\fS_4 \times \fS_2 \hookrightarrow \fS_6$
there is an $\fS_4 \times \fS_2$-equivariant small resolution of singularities
\begin{equation*}
\rho_{4,2} \colon \cY_{4,2}  = \Bl_{\bP_0,\bP_1,\bP_2,\bP_3}(\P^2 \times \P^2) \to \cY,
\end{equation*}
where $\Bl_{\bP_0,\bP_1,\bP_2,\bP_3}(\P^2 \times \P^2)$ is the blow up of $\P^2 \times \P^2$ at a general quadruple of points $\bP_0,\bP_1,\bP_2,\bP_3 \in \P^2 \times \P^2$.

\item[(ii)]
For every non-standard embedding $\fS_5 \hookrightarrow \fS_6$
there is an $\fS_5$-equivariant small resolution of singularities
\begin{equation*}
\rho_{5,1} \colon \cY_{5,1}  = \P_S(\cU_3) \to \cY,
\end{equation*}
where $S$ is the quintic del Pezzo surface and $\cU_3$ is a vector bundle of rank $3$ on $S$.

\item[(iii)]
The maps $\rho_{4,2}$ and $\rho_{5,1}$ are isomorphisms over the complement of the Cremona--Richmond configuration $\CR \subset \cY$
and are uniquely defined up to the Galois involution $\sigma$ of $\cY$ over $\P^4$ by the above properties.

\item[(iv)]
For every non-standard embedding $\fS_5 \hookrightarrow \fS_6$ and every subgroup $\fS_4 \subset \fS_5$
there is a unique $\fS_4$-equivariant small birational map
$\theta_1 \colon \cY_{5,1} \dashrightarrow \cY_{4,2}$
such that the diagram
\begin{equation}
\label{diagram:varphi}
\vcenter{\xymatrix{
\cY_{5,1} \ar[dd]_{p} \ar@{-->}[rr]^-{\theta_1} \ar[dr]_{\rho_{5,1}} &&
\cY_{4,2} \ar[dd]^{p_1} \ar[dl]^{\rho_{4,2}} \\
& \cY \\
S \ar[rr]^\varphi &&
\P^2
}}
\end{equation}
commutes, where $p \colon \cY_{5,1} = \P_S(\cU_3) \to S$ is the natural projection,
$p_1$ is the composition $\cY_{4,2} \to \P^2 \times \P^2 \to \P^2$ of the blow up with the first projection,
and $\varphi$ is the unique $\fS_4$-equivariant birational contraction $S \to \P^2$.
\end{itemize}
\end{theorem}

The vector bundle $\cU_3$ is described explicitly in~\S\ref{subsection:y51}.

The Coble fourfold is constructed from the Igusa quartic $X_{1/4}$,
but it turns out that it has a very interesting property with respect to all $\fS_6$-invariant quartics.
Since the pencil~$\{X_t\}$ is generated by $X_{1/4}$ and the double quadric $X_\infty$, we have
\begin{equation*}
X_{\frac14} \cap X_t = X_\infty \cap X_t
\qquad
\text{for any $t \not\in \left\{\frac14,\infty\right\}$.}
\end{equation*}
Hence the restriction of $X_{1/4}$ to $X_t$ has multiplicity 2, so that the double cover $\pi \colon \cY \to \P^4$ splits over $X_t$.
In other words, $\pi^{-1}(X_t)$ is the union of two irreducible components that are isomorphic to~$X_t$ and are swapped by the Galois involution~\eqref{eq:sigma}.
It is natural here to replace the parameter~$t$ in the pencil with the new parameter $\tau$ defined by
\begin{equation}\label{eq:t-tau}
t = \frac{\tau^2 + 1}{4},
\end{equation}
and define the subvarieties $\cX_\tau \subset \cY \subset \P(2,1^6)$ by~\eqref{eq:hyperplane}, \eqref{eq:coble}, and the formula
\begin{equation}\label{eq:cxtau}
x_0 + \frac{\tau}2\Big(x_1^2 + x_2^2 + x_3^2 + x_4^2 + x_5^2 + x_6^2\Big) = 0.
\end{equation}
Note that $\cX_\tau \subset \cY$ is fixed by the natural action of $\fS_6$, but \emph{is not fixed} by the twisted action.
This trivial observation leads to various reductions of groups of symmetries.

With this definition of $\cX_\tau$ we have an equality (see Lemma~\ref{lemma:cxtau-xt})
\begin{equation*}
\pi^{-1}\left(X_{\frac{\tau^2 + 1}{4}}\right) = \cX_\tau \cup \cX_{-\tau}.
\end{equation*}
The map $\sigma \colon \cX_\tau \to \cX_{-\tau}$ is an isomorphism,
and the map $\pi \colon \cX_\tau \to X_{(\tau^2+1)/4}$ is an isomorphism for all $\tau \ne \infty$.
The map~\mbox{$\pi \colon \cX_\infty \to (X_\infty)_{\mathrm{red}}$} is the double covering
branched over $(X_\infty)_{\mathrm{red}} \cap X_{1/4}$.
Thus, the threefolds~$\cX_\tau$ have the same singularities as the quartics~$X_t$
(except for $\cX_\infty$ which becomes smooth away from the $\fS_6$-orbit $\Sigma_{30}$, see Remark~\ref{remark:double-quadric-Sing}).

We consider the preimages of the divisors $\cX_\tau$ in the small resolutions $\cY_{5,1}$ and $\cY_{4,2}$:
\begin{equation}\label{eq:beauville-pullback-51-42}
\cX_\tau^{5,1} := \rho_{5,1}^{-1}(\cX_\tau),
\qquad
\cX_\tau^{4,2} := \rho_{4,2}^{-1}(\cX_\tau).
\end{equation}
Because of the mixture of the natural and the twisted action, the natural groups of symmetries of the maps
$\rho_{5,1} \colon \cX^{5,1}_\tau \to \cX_\tau$ and $\rho_{4,2} \colon \cX^{4,2}_\tau \to \cX_\tau$
(that is, the groups with respect to which these maps are equivariant)
get smaller.
In particular, for $\tau \ne 0,\infty$ the first of them reduces to $\fA_5$ and the other to
\begin{equation*}
\fA_{4,2} := (\fS_4 \times \fS_2) \cap \fA_6 \cong \fS_4.
\end{equation*}

Our second main result is the following.
Recall the discriminant set $\bDD$ defined in~\eqref{def:dd}.

\begin{theorem}
\label{theorem:three-pencils}
The maps
\begin{equation*}
\rho_{5,1} \colon \cX^{5,1}_\tau \to \cX_\tau
\qquad\text{and}\qquad
\rho_{4,2} \colon \cX^{4,2}_\tau \to \cX_\tau
\end{equation*}
are birational contractions for all $\tau$, and are small for $\tau \ne 0$.
Similarly, the maps
\begin{equation*}
\pi\circ\rho_{5,1} \colon \cX^{5,1}_\tau \to X_{\frac{\tau^2+1}{4}}
\qquad\text{and}\qquad
\pi\circ\rho_{4,2} \colon \cX^{4,2}_\tau \to X_{\frac{\tau^2+1}{4}}
\end{equation*}
are birational contractions for all $\tau\neq\infty$,
and are small for $\tau \ne 0,\infty$.
Moreover, $\cX^{5,1}_\tau$ is smooth \textup{(}and thus $\rho_{5,1}$ is a small resolution of singularities of $\cX_\tau$\textup{)} unless
\begin{equation*}
t=\frac{\tau^2 + 1}{4} \in \bDD.
\end{equation*}
The above maps are equivariant with respect to the following group actions:
\begin{equation*}
\begin{array}{|c|c|c|}
\hline
& \rho_{5,1} \text{ or } \pi \circ \rho_{5,1} & \rho_{4,2} \text{ or } \pi \circ \rho_{4,2} \\
\hline
\tau \ne 0, \infty & \fA_5 & \fA_{4,2} \\
\hline
\tau = 0 \text{ or } \tau = \infty & \fS_5 & \fS_{4,2} \\
\hline
\end{array}
\end{equation*}
where all subgroups of $\fS_6$ are non-standard and the action is twisted.
\end{theorem}

We use the above results to construct an interesting (birational) conic bundle structure on the quartics $X_t$ as follows.
The fourfold $\cY_{5,1} = \P_S(\cU_3)$ by definition comes with a~$\P^2$-fibration $p \colon \cY_{5,1} \to S$ over the quintic del Pezzo surface $S$.
We consider its restriction to the threefolds~$\cX^{5,1}_\tau \subset \cY_{5,1}$.
We show that the maps
\begin{equation*}
p \colon \cX^{5,1}_\tau \to S
\end{equation*}
are $\fA_5$-equivariant conic bundles (and for $\tau = 0, \infty$ they are $\fS_5$-equivariant).
We also discuss their properties, and identify their discriminant curves in $S$ with the Wiman--Edge pencil
(see~\S\ref{subsection:WE} for its definition and the choice of parameterization) of $\fA_5$-invariant divisors from the linear system~\mbox{$|-2K_S|$}.

All this is combined in our third main result.
Recall that a flat conic bundle $\cX \to S$ is called {\sf standard} if both $\cX$ and $S$ are smooth and the relative Picard rank $\rho(\cX/S)$ equals~$1$.

\begin{theorem}\label{theorem:conic-bundles}
The map $p \colon \cX^{5,1}_\tau \to S$ is a flat conic bundle, equivariant with respect to the group~$\fA_5$ \textup{(}for $\tau = 0,\infty$ it is $\fS_5$-equivariant\textup{)}.
It is a standard conic bundle unless
\begin{equation*}
t = \frac{\tau^2 + 1}4 \in \bDD.
\end{equation*}
Its discriminant locus is the curve $\Delta_{s(\tau)} \subset S$ from the Wiman--Edge pencil, where
\begin{equation}\label{eq:s-via-tau}
s(\tau)  = \frac{\tau^3 - \tau}{5\tau^2 + 3}
\end{equation}
for an appropriate choice of the resolution $\rho_{5,1}$.
\end{theorem}

We apply the above results in several ways.
First, we prove unirationality of $\fS_6$-invariant quartics~$X_t$ (see Corollary~\ref{corollary:xt-unirational}).
Further, we give a new and uniform proof of rationality and irrationality of the quartics~$X_t$.
For $t \not\in \bDD$ irrationality follows from the description of the intermediate Jacobian
of a resolution of singularities of $X_t$ via the Prym variety arising from the conic bundle, see Theorem~\ref{theorem:Beauville}.
For $t \in \bDD$ we show that the conic bundle can be transformed birationally into the product $S \times \P^1$,
hence $X_t$ is rational, see Theorem~\ref{theorem:Cheltsov-Shramov}.
Finally, we describe the class groups~$\Cl(X_t)$ of Weil divisors of the quartics~$X_t$ as $\fS_6$-representations (see Theorem~\ref{theorem:Cl}),
and discuss $G$-Sarkisov links centered at these quartics for some subgroups $G \subset \fS_6$.
We also prove unirationality and irrationality of the threefold~$\cX_\infty$, and describe
its class group as an $\fS_6\times\mumu_2$-representation.

\smallskip
The plan of our paper is the following.
In~\S\ref{section:resolutions} we construct the resolutions of the Coble fourfold~$\cY$ and prove Theorem~\ref{theorem:coble}.
In~\S\ref{section:conic-bundles} we discuss the conic bundle structures on the $\fS_6$-invariant quartics induced by the resolutions of the Coble fourfold,
and prove Theorems~\ref{theorem:three-pencils} and~\ref{theorem:conic-bundles}.
In~\S\ref{section:applications} we prove rationality and irrationality of the quartics $X_t$,
and in \S\ref{section:representation-class-groups} we describe the $\fS_6$-action on their class groups.
In Appendix~\ref{section:CR} we discuss the Cremona--Richmond configuration $\CR = \Sing(X_{1/4})$ of~15 lines in $\P^4$
and show that such configuration is unique up to a projective transformation of $\P^4$.

\smallskip
Throughout the paper $\Bbbk$ denotes an algebraically closed field of characteristic zero; however, many constructions do not use the assumption that the field is algebraically closed.
By $\mumu_n$ we denote the cyclic group of order $n$.
Furthermore, we denote by
\begin{equation}
\label{eq:sn1n2}
\fS_{n_1,n_2} \cong \fS_{n_1} \times \fS_{n_2} \subset \fS_{n_1 + n_2}
\qquad\text{and}\qquad
\fA_{n_1,n_2} = \fA_{n_1 + n_2} \cap \fS_{n_1,n_2} \subset \fA_{n_1 + n_2}
\end{equation}
the subgroup of $\fS_{n_1 + n_2}$ that consists of permutations preserving the subsets of the first~$n_1$ and the last~$n_2$ indices,
and its intersection with the alternating group~$\fA_{n_1+n_2} \subset \fS_{n_1 + n_2}$.
Note that~$\fA_{n-2,2} \cong \fS_{n-2}$.

\smallskip
We are grateful to A.\,Beauville, S.\,Bloch, I.\,Dolgachev, G.\,Kapustka, D.\,Markushevich,  Yu.\,Prokhorov, and E.\,Tevelev,
for useful discussions.
We also thank the referee for valuable remarks.
This paper was written during the first author's stay at the Max Planck Institute for Mathematics in~2017.
He would like to thank the institute for the excellent working conditions.
All authors were partially supported by the HSE University Basic Research Program,
Russian Academic Excellence Project~\mbox{``5--100''}.
The second and the third authors were also supported
by RFBR grants~15-01-02164 and 15-01-02158.
The third author was also supported by Young Russian Mathematics award.

\section{Small resolutions of the Coble fourfold}
\label{section:resolutions}

Recall that the fourfold $\cY$ is defined by~\eqref{eq:coble} as the double cover of $\P^4$
(considered as the hyperplane~\eqref{eq:hyperplane} in $\P^5$) branched over the Igusa quartic~\eqref{eq:igusa}.
It comes with the natural and the twisted actions of the symmetric group $\fS_6$, see~\eqref{eq:action-standard} and~\eqref{eq:action-twisted},
the double covering~\mbox{$\pi \colon \cY \to \P^4$} and its Galois involution $\sigma \colon \cY \to \cY$, see~\eqref{eq:sigma},
commuting with both actions of $\fS_6$.

The fourfold $\cY$ has been studied by Coble in \cite{Coble1,Coble2,Coble3}.
He showed that~$\cY$ is a compactification of the moduli space
of ordered sets of $6$ points in the projective plane.
A modern treatment of $\cY$ has been given in \cite{DolgachevOrtland,MatsumotoSasakiYoshida,Hunt,HowardMillsonSnowdenVakil},
see also~\cite{BauerVerra}.
In particular, Dolgachev and Ortland proved in \cite{DolgachevOrtland} that $\cY$ can be obtained as
the GIT-quotient $(\mathbb{P}^2)^6 \sslash \mathrm{SL}_{3}(\Bbbk)$ with respect to the diagonal action of~$\mathrm{SL}_3(\Bbbk)$.
In~\cite{CMS18} the variety $\cY$ came up in the study of moduli spaces of K3 surfaces.
In~\cite{Hunt}, Hunt called it the Coble variety (he also denoted it by $\cY$).
In the current paper we prefer to call~$\cY$ the {\sf Coble fourfold}.

Since the Coble fourfold $\cY$ is singular, it is interesting to construct its
resolution of singularities that would be natural from the geometric point of view.
One interesting resolution was provided by Naruki~\cite{Naruki}, see also~\cite{HKT} and~\cite[\S2]{DGK}.
It has plenty of important properties due to its interpretation as a moduli space of cubic surfaces.
However, it is quite big (it has a horde of exceptional divisors).
On the other hand, one can observe that the variety~$\cY$ has non-$\mathbb{Q}$-factorial singularities, so we can hope to have
a nice \emph{small} resolution (i.e., with exceptional locus of codimension~2).

In this section we construct two small resolutions of singularities of $\cY$;
one is equivariant with respect to the subgroup $\fS_{4,2} \subset \fS_6$,
and another is equivariant with respect to the subgroup $\fS_5 \subset \fS_6$.
Note that in both cases a \emph{non-standard} embedding of the subgroup is used
(equivalently, a standard embedding is composed with an outer automorphism of~$\fS_6$),
and in both cases we consider the \emph{twisted} action of $\fS_6$ on $\cY$.

\subsection{Blow up of $\P^2 \times \P^2$}\label{subsection:y42}

Let $\WW_3$ be the irreducible three-dimensional representation of the symmetric group~$\fS_4$
with the non-trivial determinant, i.e., a summand of the four-dimensional permutation representation.
Explicitly, $\WW_3 \cong \rR(3,1)$ in the notation of~\cite[\S4.1]{FultonHarris}.
Choose a $\fS_4$-orbit of length 4
\begin{equation*}
\{P_0,P_1,P_2,P_3\} \subset \P(\WW_3) \cong \P^2.
\end{equation*}
In appropriate coordinates such quadruple can be written as
\begin{equation}\label{eq:points}
P_0 = (1:1:1),\ P_1 = (1:0:0),\ P_2 = (0:1:0),\ P_3 = (0:0:1).
\end{equation}
Denote by
\begin{equation*}
\overline{P_iP_j} \subset \P(\WW_3),
\qquad 0 \leqslant i < j \leqslant 3,
\end{equation*}
the line passing through the points $P_i$ and $P_j$.

Consider the diagonal action of $\fS_4$ on $\P(\WW_3) \times \P(\WW_3)$ and the diagonal quadruple
\begin{equation*}
\bP = \{\bP_0,\bP_1,\bP_2,\bP_3\}  \subset \P(\WW_3) \times \P(\WW_3),
\qquad
\bP_i = (P_i,P_i).
\end{equation*}
Note that $\bP$ is an $\fS_4$-orbit.
The vector space $\WW_3 \otimes \WW_3$ can be regarded as a representation of the group $\fS_{4,2}$, see~\eqref{eq:sn1n2},
where $\fS_4$ acts diagonally and the non-trivial element of $\fS_2$ interchanges the factors.
The linear span of the points~$\bP_i$ in $\P(\WW_3 \otimes \WW_3)$ induces an embedding of the permutation representation $\Bbbk^4$ of~$\fS_4$
(with the trivial action of $\fS_2$) into~\mbox{$\WW_3 \otimes \WW_3$}.
We denote by
\begin{equation}\label{eq:w-w3w3}
\WW_5 {}:={} (\WW_3 \otimes \WW_3)/\Bbbk^4
\end{equation}
the quotient five-dimensional representation of $\fS_{4,2}$.
Note that as a representation of $\fS_4$ it is the direct sum $\WW_5\vert_{\fS_4} \cong \rR(2,2) \oplus \rR(2,1,1)$;
here we again use the (standard) notation of~\mbox{\cite[\S4.1]{FultonHarris}}.

The linear projection $\WW_3 \otimes \WW_3 \to \WW_5$ induces a rational map
\begin{equation*}
\bar\pi_{4,2} \colon \P(\WW_3) \times \P(\WW_3)
\hookrightarrow \P(\WW_3 \otimes \WW_3)
\dashrightarrow \P(\WW_5).
\end{equation*}
Note that the center of this projection is the linear span of $\bP$ in $\P(\WW_3\otimes\WW_3)$, which
intersects~\mbox{$\P(\WW_3) \times \P(\WW_3)$} exactly by~$\bP$.
Therefore, to regularize the map $\bar\pi_{4,2}$ we should consider
the blow up~$\cY_{4,2}$ of~\mbox{$\P(\WW_3) \times \P(\WW_3)$} in the quadruple $\bP$:
\begin{equation}
\label{eq:y42}
\cY_{4,2} {}:={} \Bl_{\bP_0,\bP_1,\bP_2,\bP_3}(\P(\WW_3) \times \P(\WW_3)) \xrightarrow{\ \beta\ } \P(\WW_3) \times \P(\WW_3)
\end{equation}
with $\beta$ being the blow up morphism.
This induces a commutative diagram
\begin{equation}
\label{diagram:y42}
\vcenter{\xymatrix{
& \cY_{4,2} \ar[dl]_\beta \ar[dr]^{\pi_{4,2}}
\\
\P(\WW_3) \times \P(\WW_3) \ar@{-->}[rr]^-{\bar\pi_{4,2}}
&&
\P(\WW_5)
}}
\end{equation}
By construction
the fourfold $\cY_{4,2}$ is smooth and carries a faithful action of~$\fS_{4,2}$.
The above diagram is $\fS_{4,2}$-equivariant.

We are going to show that the map $\pi_{4,2} \colon \cY_{4,2} \to \P(\WW_5)$ defined by the diagram~\eqref{diagram:y42} factors through the Coble fourfold;
more precisely, $\pi_{4,2}$ factors as a composition
\begin{equation*}
\cY_{4,2} \xrightarrow{\ \rho_{4,2}\ } \cY \xrightarrow{\ \pi\ } \P(\WW_5),
\end{equation*}
with $\rho_{4,2}$ being a small $\fS_{4,2}$-equivariant resolution of singularities.
We accomplish this in two steps.

First, consider the linear projection
\begin{equation*}
\P(\WW_3) \times \P(\WW_3)\hookrightarrow \P(\WW_3 \otimes \WW_3) \dashrightarrow \P^5
\end{equation*}
from the linear span of the points~$\bP_1$, $\bP_2$, and~$\bP_3$; as before,
the latter linear span intersects $\P(\WW_3) \times \P(\WW_3)$ exactly by the triple~$\bP_1$, $\bP_2$, $\bP_3$.
If $(u_1:u_2:u_3)$ and $(v_1:v_2:v_3)$ are homogeneous coordinates on the first and the second factors of $\P(\WW_3) \times \P(\WW_3)$ such that~\eqref{eq:points} holds,
this map is given by
\begin{equation}\label{eq:uv-yz}
((u_1:u_2:u_3),(v_1:v_2:v_3)) \mapsto (u_2v_3:u_3v_1:u_1v_2:u_3v_2:u_1v_3:u_2v_1),
\end{equation}
and it is easy to describe its structure.
We denote by $y_1$, $y_2$, $y_3$, $z_1$, $z_2$, and~$z_3$
the homogeneous coordinates on~$\P^5$, so that the right hand side of~\eqref{eq:uv-yz}
is the point~\mbox{ $(y_1:y_2:y_3:z_1:z_2:z_3)$}.

\begin{lemma}\label{lemma:projection-1}
The linear projection
$\P(\WW_3) \times \P(\WW_3) \dashrightarrow \P^5$
with center in the
span of the points $\bP_1,\bP_2,\bP_3$ induces an $\fS_{3,2}$-equivariant commutative diagram
\begin{equation*}
\xymatrix{
& \Bl_{\bP_1,\bP_2,\bP_3}(\P(\WW_3) \times \P(\WW_3)) \ar[dr]^{\rho'_{4,2}} \ar[dl]_{\beta'}
\\
\P(\WW_3) \times \P(\WW_3) \ar@{-->}[rr] &&
\cY_{4,2}' \ar@{^{(}->}[r] & \P^5
}
\end{equation*}
where $\beta'$ is the blow up, $\cY'_{4,2} \subset \P^5$ is a singular cubic hypersurface
given by the equation
\begin{equation}\label{eq:cubic}
y_1y_2y_3 = z_1z_2z_3,
\end{equation}
and $\rho'_{4,2}$ is a small birational contraction.
The map $\rho'_{4,2}$ contracts
\begin{itemize}
\item
the proper transforms of the six planes $\P(\WW_3) \times P_i$ and $P_i \times \P(\WW_3)$, $1 \leqslant i \leqslant 3$,~and
\item
the proper transforms of the three quadrics $\overline{P_iP_j} \times \overline{P_iP_j}$, $1 \leqslant i < j \leqslant 3$,
\end{itemize}
onto nine lines $L_{ij}$, $1 \leqslant i,j \leqslant 3$, given in $\P^5$ by the equations
\begin{equation*}
y_k = z_l = 0, \quad k \ne i,\ l \ne j.
\end{equation*}
Moreover, $\rho'_{4,2}$ is an isomorphism over the complement of the lines $L_{ij}$.
Finally, the map~$\rho'_{4,2} \circ (\beta')^{-1}$ takes the point $\bP_0$ to the point
$\bP'_0 = (1:1:1:1:1:1) \in \cY'_{4,2}$.
\end{lemma}
\begin{proof}
The map is toric, so everything is easy to describe.
We skip the actual computation which is straightforward but tedious.
\end{proof}

The cubic fourfold~\eqref{eq:cubic} is known as {\sf Perazzo primal}, \cite[Exercise~9.16]{Dolgachev-CAG}, \cite[\S6]{Loo09}.

Using equation~\eqref{eq:cubic} one can easily check that the
union of the nine lines $L_{ij}$ is the singular locus of the cubic $\cY'_{4,2}$.

The second step is to project the cubic $\cY'_{4,2}$ from the point $\bP'_0$.

\begin{lemma}\label{lemma:projection-2}
The linear projection $\bar\pi'_{4,2} \colon \cY'_{4,2} \dashrightarrow \P(\WW_5)$ from the point $\bP'_0$ defines a regular map
$\pi''_{4,2} \colon \Bl_{\bP'_0}(\cY'_{4,2}) \xrightarrow{\ \ \ } \P(\WW_5)$ that fits into a commutative diagram
\begin{equation}\label{eq:pervyj-romb}
\vcenter{\xymatrix{
& \Bl_{\bP'_0}(\cY'_{4,2}) \ar[rr]^-{\rho''_{4,2}} \ar[dr]^{\pi''_{4,2}} \ar[dl]_{\beta''} && \cY \ar[dl]_{\pi}
\\
\cY'_{4,2} \ar@{-->}[rr]^{\bar\pi'_{4,2}} && \P(\WW_5)
}}
\end{equation}
where $\cY$ is the Coble fourfold, $\pi \colon \cY \to \P(\WW_5)$ is the double covering, and~$\rho''_{4,2}$ is a small birational morphism.
Furthermore, the exceptional locus of~$\rho''_{4,2}$ is the union of proper transforms of the six planes $\Pi_w\subset \cY'_{4,2}$ given by the equations
\begin{equation*}
z_i = y_{w(i)}, \qquad 1 \leqslant i \leqslant 3,
\end{equation*}
indexed by all bijections $w \colon \{1,2,3\} \to \{1,2,3\}$;
the map $\rho''_{4,2}$ contracts them onto six lines in $\cY$ \textup(i.e., rational curves that are isomorphically projected to lines in~$\P(\WW_5)$\textup),
and is an isomorphism over the complement of those.
\end{lemma}
\begin{proof}
Note that the point $\bP'_0$ is a smooth point of the cubic $\cY'_{4,2}$, so the projection from it factors through a double covering of $\P(\WW_5)$;
in fact, this is the Stein factorization for the morphism~$\pi''_{4,2}$.
We have to identify its branch divisor with the Igusa quartic.

Take a point
\begin{equation*}
(y_{i} : z_{i}) = (y_1:y_2:y_3:z_1:z_2:z_3)
\end{equation*}
in $\P^5$ which is different from $\bP'_0$.
The line $M_{(y_{i} : z_{i})}$ in~$\P(\WW_5)$ passing through the point~\mbox{$(y_{i} : z_{i})$}  and the point $\bP'_0$ can be parameterized as
\begin{equation}\label{eq:line-parameterization}
M_{(y_{i} : z_{i})} = \{ (\lambda + \mu y_1: \lambda + \mu y_2: \lambda + \mu y_3: \lambda + \mu z_1: \lambda + \mu z_2: \lambda + \mu z_3) \},
\end{equation}
where $\lambda$ and $\mu$ are considered as homogeneous coordinates on this line.
Substituting this parameterization into~\eqref{eq:cubic}, we see that
the intersection of $M_{(y_{i} : z_{i})}$ with the cubic $\cY'_{4,2}$ is given by the equation
\begin{equation*}
(\lambda + \mu y_1)(\lambda + \mu y_2)(\lambda + \mu y_3) =
(\lambda + \mu z_1)(\lambda + \mu z_2)(\lambda + \mu z_3).
\end{equation*}
Expanding both sides and canceling the factor $\mu$
that corresponds to the intersection point~$\bP'_0$, we can rewrite the above equation as
\begin{equation}\label{eq:yp42}
(\bs_1(y) - \bs_1(z))\lambda^2 + (\bs_2(y) - \bs_2(z))\lambda\mu + (\bs_3(y) - \bs_3(z))\mu^2 = 0,
\end{equation}
where $\bs_d$ denotes the elementary symmetric polynomial of degree $d$.
Restricting equation~\eqref{eq:yp42} to the hyperplane
\begin{equation}\label{eq:hyperplane-yz}
y_1 + y_2 + y_3 + z_1 + z_2 + z_3 = 0,
\end{equation}
which is identified by the linear projection~$\bar\pi'_{4,2}$ from the point~$\bP'_0$ with the space~$\P(\WW_5)$,
we obtain the equation of the double cover over~$\P(\WW_5)$ we are interested in
(embedded into the projectivization of the vector bundle \mbox{$\cO_{\P(\WW_5)} \oplus \cO_{\P(\WW_5)}(-1)$} over~$\P(\WW_5)$).
The branch divisor of~$\bar\pi'_{4,2}$
is given in the hyperplane~\eqref{eq:hyperplane-yz} by the discriminant of the quadratic equation~\eqref{eq:yp42}:
\begin{equation}\label{eq:igusa-yz}
(\bs_2(y) - \bs_2(z))^2 - 4(\bs_1(y) - \bs_1(z))(\bs_3(y) - \bs_3(z)) = 0.
\end{equation}

Let us show that the quartic $X'' \subset \P^4$ defined by equations~\eqref{eq:hyperplane-yz} and~\eqref{eq:igusa-yz}
is isomorphic to the Igusa quartic;
this will identify the double covering with the Coble fourfold
in a way respecting
the projection to $\P^4$, that is, ensuring that the upper right triangle in diagram~\eqref{eq:pervyj-romb} is commutative.

To do this we use the following substitutions:
\begin{equation}\label{eq:xyz}
\begin{array}{ll}
x_1 = y_1 - \frac23\bs_1(y) + \frac13\bs_1(z), \qquad\qquad\qquad&
x_4 = z_1 + \frac13\bs_1(y) - \frac23\bs_1(z), \\[.5ex]
x_2 = y_2 - \frac23\bs_1(y) + \frac13\bs_1(z), &
x_5 = z_2 + \frac13\bs_1(y) - \frac23\bs_1(z), \\[.5ex]
x_3 = y_3 - \frac23\bs_1(y) + \frac13\bs_1(z), &
x_6 = z_3 + \frac13\bs_1(y) - \frac23\bs_1(z).
\end{array}
\end{equation}
They express the composition of the projection $\bar\pi'_{4,2}$ with a particular identification of its target space $\P(\WW_5)$
with the hyperplane~\eqref{eq:hyperplane} in $\P^5$.
A direct verification shows that substituting these expressions into equation~\eqref{eq:igusa} of the Igusa quartic we get~\eqref{eq:igusa-yz}.
This proves that~\eqref{eq:igusa-yz} is isomorphic to the cone over the Igusa quartic with the vertex at the
point $\bP'_0$, hence its intersection with~\eqref{eq:hyperplane-yz} is isomorphic to the Igusa quartic.

Finally, we describe the exceptional locus of the projection~$\pi''_{4,2}$.
Clearly, it is the union of those lines $M_{(y_{i} : z_{i})}$ that are contained in the cubic $\cY'_{4,2}$,
i.e., the subvariety of those points $(y_i : z_i)$ for which~\eqref{eq:yp42} is identically zero.
This condition can be rewritten as:
\begin{equation*}
\bs_1(y) - \bs_1(z) = \bs_2(y) - \bs_2(z) = \bs_3(y) - \bs_3(z) = 0
\end{equation*}
Of course, this is equivalent to $(y_{i} : z_{i}) \in \Pi_w$ for some permutation~$w$.
Thus the exceptional locus is the union of the proper transforms of the planes $\Pi_w$.
Each of these planes passes through $\bP'_0$, hence is contracted onto a line in $\P^4 \cong \P(\WW_5)$.
\end{proof}

\begin{remark}
There is also a computation-free way to identify the branch divisor $X''$ of the map $\pi''_{4,2}$ with the Igusa quartic.
Indeed, note that the singular locus of $X''$ contains 15 lines
(the images of the 9 singular lines $L_{ij}$ of $\cY'_{4,2}$ and the images of the~6 planes~$\Pi_w$),
then check that they form a Cremona--Richmond configuration (e.g., by using Theorem~\ref{theorem:CR-unique}),
and then apply Corollary~\ref{corollary:Igusa-unique}.
\end{remark}

\begin{remark}\label{remark:sigma42}
Using equation~\eqref{eq:yp42} it is easy to write the (birational) involution of the double covering $\cY'_{4,2} \dashrightarrow \P^4$ explicitly.
Indeed, choose a point~\mbox{$(y_{i} : z_{i}) = (y_1:y_2:y_3:z_1:z_2:z_3)$} on the cubic~\mbox{$\cY'_{4,2}\subset\P^5$} different from~$\bP'_0$.
Using the parameterization~\eqref{eq:line-parameterization}, we see that the point $(y_{i} : z_{i})$ corresponds to~\mbox{$\lambda=0$}.
Keeping in mind that $\bs_3(y)=\bs_3(z)$ at our point $(y_{i} : z_{i})$, and finding the second root of the equation~\eqref{eq:yp42}
in $\lambda/\mu$, we conclude that the involution of the double covering~\mbox{$\cY'_{4,2} \dashrightarrow \P^4$}
is given by
\begin{equation}\label{eq:crocodile-1}
(y_i : z_i) \mapsto \Big((\bs_1(y) - \bs_1(z))y_i - (\bs_2(y) - \bs_2(z)) : (\bs_1(y) - \bs_1(z))z_i - (\bs_2(y) - \bs_2(z))\Big).
\end{equation}
Furthermore, the induced birational involution of $\P(\WW_3) \times \P(\WW_3)$ can be written as
\begin{multline}\label{eq:crocodile-2}
\bar\sigma_{4,2} \colon ((u_1:u_2:u_3),(v_1:v_2:v_3)) \mapsto \\
\mapsto
\left(
\left(
\frac{v_2 - v_3}{\det\left(\begin{smallmatrix}u_2 & u_3\\v_2 & v_3\end{smallmatrix}\right)} :
\frac{v_3 - v_1}{\det\left(\begin{smallmatrix}u_3 & u_1\\v_3 & v_1\end{smallmatrix}\right)} :
\frac{v_1 - v_2}{\det\left(\begin{smallmatrix}u_1 & u_2\\v_1 & v_2\end{smallmatrix}\right)}
\right),
\left(
\frac{u_2 - u_3}{\det\left(\begin{smallmatrix}u_2 & u_3\\v_2 & v_3\end{smallmatrix}\right)} :
\frac{u_3 - u_1}{\det\left(\begin{smallmatrix}u_3 & u_1\\v_3 & v_1\end{smallmatrix}\right)} :
\frac{u_1 - u_2}{\det\left(\begin{smallmatrix}u_1 & u_2\\v_1 & v_2\end{smallmatrix}\right)}
\right)
\right);
\end{multline}
to see this one can just compose~\eqref{eq:uv-yz} with~\eqref{eq:crocodile-1}
and observe that it gives the same result as a composition of~\eqref{eq:crocodile-2} with~\eqref{eq:uv-yz}.
Similarly, we deduce from~\eqref{eq:igusa-yz} that the ramification divisor of the map~$\bar\pi_{4,2}$
is given by the equation
\begin{equation*}
\bs_2(u_2v_3,u_3v_1,u_1v_2) = \bs_2(u_3v_2,u_1v_3,u_2v_1),
\end{equation*}
that can be compactly rewritten as
\begin{equation}\label{eq:igusa-uv}
\det \left(
\begin{smallmatrix}
u_1v_1 & u_2v_2 & u_3v_3 \\
u_1 & u_2 & u_3 \\
v_1 & v_2 & v_3
\end{smallmatrix}
\right) = 0.
\end{equation}
This gives a determinantal representation of a threefold birational to the Igusa quartic.
\end{remark}

Combining the results of Lemmas~\ref{lemma:projection-1} and~\ref{lemma:projection-2} we obtain a commutative diagram
\begin{equation}
\label{diagram:big-4-2}
\vcenter{\xymatrix@C=4em@R=3ex{
\P(\WW_3) \times \P(\WW_3) \ar@{-->}[dd]_{\bar\pi_{4,2}} &
\Bl_{\bP_1,\bP_2,\bP_3}(\P(\WW_3) \times \P(\WW_3)) \ar[l]_-{\beta'} \ar[d]_{\rho'_{4,2}} &
\cY_{4,2} \ar[l] \ar[d]
\\
&
\cY'_{4,2} \ar@{-->}[dl]_{\bar\pi'_{4,2}} &
\Bl_{\bP_0}(\cY'_{4,2}) \ar[l]_-{\beta''} \ar[d]^{\rho''_{4,2}}
\\
\P(\WW_5) &&
\cY \ar[ll]_-\pi
}}
\end{equation}
where the upper right square is Cartesian and the composition $\cY_{4,2} \to \P(\WW_3) \times \P(\WW_3)$ of the upper horizontal arrows is the blow up map $\beta$.

\begin{proposition}
\label{proposition:first-Q-factorialization}
The linear projection $\bar\pi_{4,2} \colon \P(\WW_3) \times \P(\WW_3) \dashrightarrow \P(\WW_5)$
with center in the
span of the points~$\bP_0$, $\bP_1$, $\bP_2$, and~$\bP_3$
gives rise to a
commutative diagram
\begin{equation}
\label{eq:vtoroj-romb}
\vcenter{\xymatrix{
& \cY_{4,2} \ar[rr]^{\rho_{4,2}} \ar[dr]^{\pi_{4,2}} \ar[dl]_\beta && \cY \ar[dl]_\pi
\\
\P(\WW_3) \times \P(\WW_3) \ar@{-->}[rr]^-{\bar\pi_{4,2}} && \P(\WW_5)
}}
\end{equation}
where $\rho_{4,2}$ is a small resolution of singularities defined uniquely up to a composition with the Galois involution $\sigma \colon \cY \to \cY$.
The map $\rho_{4,2}$
contracts:
\begin{itemize}
\item
the proper transforms of the eight planes $\P(\WW_3) \times P_i$ and $P_i \times \P(\WW_3)$, $0 \leqslant i \leqslant 3$,
\item
the proper transforms of the six quadrics $\overline{P_iP_j} \times \overline{P_iP_j}$, $0 \leqslant i < j \leqslant 3$, and
\item
the proper transform of the diagonal $\P(\WW_3) \hookrightarrow \P(\WW_3) \times \P(\WW_3)$,
\end{itemize}
and is an isomorphisms on the complement of those.
Moreover, the morphism $\pi_{4,2}$ induces a non-standard embedding $\fS_{4,2} \to \fS_6$
such that $\rho_{4,2}$ is $\fS_{4,2}$-equivariant with respect to the twisted action of $\fS_{4,2}$  on~$\cY$.
\end{proposition}

\begin{proof}
We define the map $\rho_{4,2}$ as the composition of the right vertical arrows in~\eqref{diagram:big-4-2}.
Its uniqueness up to $\sigma$ is evident.
We note that the composition
\begin{equation*}
\rho_{4,2} \circ \beta^{-1} \colon \P(\WW_3) \times \P(\WW_3) \dashrightarrow \cY \subset \P(2,1^6)
\end{equation*}
can be defined by explicit formulas:
\begin{equation}
\label{eq:choice-rho-pp}
\hspace{-2em}\left\{
\arraycolsep=.15em
\begin{array}{llllllllllll}
{}x_0  =
\hbox to 0pt{$- u_1u_3v_1v_2 - u_1u_2v_2v_3 - u_2u_3v_1v_3 + u_1u_2v_1v_3 + u_2u_3v_1v_2 + u_1u_3v_2v_3$,\hss}\\
\textstyle
{}x_1 =
 \frac13 (	& u_2v_3 & -\,2	& u_3v_1 & -\,2 & u_1v_2 & + 	& u_3v_2 & + 	& u_1v_3 & + 	& u_2v_1), \\[.5ex]
{}x_2 =
 \frac13 ( -\,2	& u_2v_3 & + 	& u_3v_1 & -\,2 & u_1v_2 & + 	& u_3v_2 & + 	& u_1v_3 & + 	& u_2v_1), \\[.5ex]
{}x_3 =
 \frac13 ( -\,2	& u_2v_3 & -\,2	& u_3v_1 & + 	& u_1v_2 & + 	& u_3v_2 & + 	& u_1v_3 & + 	& u_2v_1), \\[.5ex]
{}x_4 =
 \frac13 ( 	& u_2v_3 & + 	& u_3v_1 & + 	& u_1v_2 & + 	& u_3v_2 & -\,2	& u_1v_3 & -\,2	& u_2v_1), \\[.5ex]
{}x_5 =
 \frac13 ( 	& u_2v_3 & + 	& u_3v_1 & + 	& u_1v_2 & -\,2	& u_3v_2 & + 	& u_1v_3 & -\,2	& u_2v_1), \\[.5ex]
{}x_6 =
 \frac13 ( 	& u_2v_3 & + 	& u_3v_1 & + 	& u_1v_2 & -\,2	& u_3v_2 & -\,2	& u_1v_3 & + 	& u_2v_1).
\end{array}
\right.
\end{equation}
Indeed, $x_0$ defines in $\cY$ the ramification divisor of the map $\pi$, hence its pullback to~$\P(\WW_3) \times \P(\WW_3)$ coincides (up to a scalar)
with the equation~\eqref{eq:igusa-uv} of the ramification divisor of $\bar\pi_{4,2}$.
The pullbacks of $x_1, \ldots, x_6$ are given by the composition of~\eqref{eq:xyz} and~\eqref{eq:uv-yz}, which gives the required formulas.
Substituting those into~\eqref{eq:coble}, we see that the scalar in the formula for $x_0$ is $\pm 1$.
So, \eqref{eq:choice-rho-pp} gives one of the two maps $\rho_{4,2}$, while the other sign choice gives $\sigma \circ \rho_{4,2}$.

For the description of the exceptional locus of $\rho_{4,2}$ we combine the results of Lemmas~\ref{lemma:projection-1} and~\ref{lemma:projection-2}
with the simple observation (using~\eqref{eq:uv-yz})
that the map $\rho'_{4,2} \circ {\beta'}^{-1}$ from~\eqref{diagram:big-4-2} takes the two planes~$\P(\WW_3) \times P_0$ and $P_0 \times \P(\WW_3)$
to the planes $\Pi_w$, where $w$ are cycles of length~3; takes the three quadrics $\overline{P_0P_i} \times \overline{P_0P_i}$ to the planes $\Pi_w$,
where $w$ are transpositions; and takes the diagonal to~$\Pi_w$, where $w$ is the identity permutation.

The space $\WW_5$ by definition~\eqref{eq:w-w3w3} comes with an $\fS_{4,2}$ action,
such that the map~\mbox{$\pi_{4,2} \colon \cY_{4,2} \to \P(\WW_5)$} obtained by resolving the indeterminacy of the linear projection~$\bar\pi_{4,2}$ is $\fS_{4,2}$-equivariant.
It follows that its branch divisor, which was shown to be the Igusa quartic~$X_{1/4}$, is invariant under this action.
On the other hand, it is well known that~\mbox{$\Aut(X_{1/4})\cong\fS_6$}
(this follows for instance from~\cite[\S3]{Fnklberg} and~\cite[Proposition~3.3.1]{Hunt}, see also Lemma~\ref{lemma:quartics-Aut} below).
Thus, we obtain an embedding $\fS_{4,2} \hookrightarrow \fS_6$.

Moreover, for every element $g \in \fS_{4,2}$ the conjugation of the diagram \eqref{eq:vtoroj-romb} by $g$ gives a diagram of the same form.
Since $\rho_{4,2}$ is uniquely defined up to $\sigma$, we obtain an equality
\begin{equation*}
g \circ \rho_{4,2} \circ g^{-1} = \sigma^{k(g)} \circ \rho_{4,2},
\end{equation*}
where $k \colon \fS_{4,2} \to \ZZ/2\ZZ$
is a group homomorphism.
Using the explicit expression for~$x_0$ provided by~\eqref{eq:igusa-uv}
it is easy to see that transpositions in the group $\fS_{4,2}$ change the sign of $x_0$.
This means that $k$ is the homomorphism of parity $\fS_6 \to \ZZ/2\ZZ$ restricted to $\fS_{4,2}$, which means that
the map $\rho_{4,2}$ is equivariant with respect to the twisted action~\eqref{eq:action-twisted} of~$\fS_6$ on $\cY$.

Finally, to show that the embedding $\fS_{4,2} \hookrightarrow \fS_6$ is non-standard,
we use~\eqref{eq:choice-rho-pp} to observe that transpositions in~$\fS_{4,2}$ go to permutations of cycle type $[2,2,2]$ in~$\fS_6$.
Alternatively, we could notice that the restriction of the representation~\eqref{eq:hyperplane} with respect to a standard embedding $\fS_4 \hookrightarrow \fS_6$
decomposes as a direct sum of three irreducible representations of~$\fS_4$ (cf.~\eqref{eq:s5-s4} and Lemma~\ref{lemma:restriction-s6-s5}),
while~\eqref{eq:w-w3w3} is the sum of two irreducibles.
\end{proof}

Let us emphasize again that there are exactly two maps $\rho_{4,2}$ that fit into commutative diagram~\eqref{eq:vtoroj-romb}:
the first is given by~\eqref{eq:choice-rho-pp} and the second is obtained by its composition with $\sigma$, i.e., by the change of sign of $x_0$.
The particular choice~\eqref{eq:choice-rho-pp} will lead us to a particular choice of the map $\rho_{5,1}$ in the next subsection.

We write down here a simple consequence
of Proposition~\ref{proposition:first-Q-factorialization} concerning the Weil divisor class group of the Coble fourfold.

\begin{corollary}
\label{corollary:rkcl-y}
One has $\rkCl(\cY) = 6$.
\end{corollary}
\begin{proof}
Since the map $\rho_{4,2} \colon \cY_{4,2} \to \cY$ is a small resolution of singularities,
it induces an isomorphism $\Cl(\cY) \cong \Pic(\cY_{4,2})$,
and since $\cY_{4,2}$ is the blow up of $\P^2 \times \P^2$ in 4 points, its Picard rank equals 6.
\end{proof}

In~Theorem~\ref{theorem:Cl} we will describe the action of the group $\fS_6 \times \mumu_2$ on $\Cl(\cY) \otimes \Q$.

\begin{remark}\label{remark:rho-4-2-Weil-divisor}
For each three-element subset $I \subset \{1,\dots,6\}$ denote by $\bar{I} \subset \{1,\dots,6\}$ its complement.
Consider the hyperplane $H_I \subset \P^4$ defined in~\eqref{eq:hyperplane} by the equation
\begin{equation}
\label{eq:hi}
\sum_{i \in I} x_i = 0.
\end{equation}
Note that $H_I = H_{\bar{I}}$.
In the terminology of Appendix~\ref{section:CR} these are the ten jail hyperplanes~\eqref{eq:jail-hyperplane}
of the Cremona--Richmond configuration.
The preimage of $H_I$ on $\cY$ splits as the union of two irreducible components.
Indeed, consider the subvariety $\cH_I \subset \cY$ defined by the equation~\eqref{eq:hi} together with the equation
\begin{equation}
\label{eq:chi}
x_0 + \frac12 \left( \sum_{i \in I} x_i^2 - \sum_{i \in \bar{I}} x_i^2 \right) = 0.
\end{equation}
Then it is easy to check that
\begin{equation*}
\pi^{-1}(H_I) = \pi^{-1}(H_{\bar{I}}) = \cH_I \cup \cH_{\bar{I}}.
\end{equation*}

Even an easier way to see this splitting is provided by the morphism~$\rho_{4,2}$.
Indeed, using formulas~\eqref{eq:choice-rho-pp} one can check that the preimages on $\P^2\times\P^2$ of
the six hyperplanes $H_{124}$, $H_{125}$, $H_{134}$, $H_{136}$, $H_{235}$, and~$H_{236}$
are divisors given by equations
\begin{multline*}
(u_1-u_3)v_2=0,\quad
u_1(v_2-v_3)=0,\quad
u_3(v_1-v_2)=0,\quad \\
(u_2-u_3)v_1=0,\quad
(u_1-u_2)v_3=0,\quad
u_2(v_1-v_3)=0,
\end{multline*}
respectively.
Each of these divisors is a union of two irreducible components,
and each component is the product $\overline{P_iP_j} \times \P^2$ or $\P^2 \times \overline{P_iP_j}$ for appropriate $i$ and $j$.
Note that the action of $\fS_{4,2}$ on the set of all twelve of these irreducible components is transitive.
For each $I$ denote
\begin{equation*}
\cH^{4,2}_I := \rho_{4,2}^{-1}(\cH_I).
\end{equation*}
Therefore, if $I$ is one of the above six triples or one of their complements, then $\beta(\cH^{4,2}_I)$ is
one of the above twelve components, hence these divisors $\cH^{4,2}_I$ form a single $\fS_{4,2}$-orbit.

Similarly, formulas~\eqref{eq:choice-rho-pp} show that the preimages on $\P^2\times\P^2$
of the remaining four hyperplanes $H_{123}$, $H_{156}$, $H_{246}$, and~$H_{345}$
are irreducible divisors singular at the points~$\bP_0$, $\bP_1$, $\bP_2$, and~$\bP_3$, respectively.
This means that for each of the above four triples~$I$ the preimage $\pi_{4,2}^{-1}(H_I)$ of $H_I$ on $\cY_{4,2}$
consists of two irreducible components, one of them being the exceptional divisor of the blow up $\beta$ over the corresponding point~$\bP_r$.
A straightforward computation shows that
\begin{equation*}
\cH^{4,2}_{123},\
\cH^{4,2}_{156},\
\cH^{4,2}_{246},\
\cH^{4,2}_{345}
\end{equation*}
are the exceptional divisors, while
\begin{equation*}
\cH^{4,2}_{456},\
\cH^{4,2}_{234},\
\cH^{4,2}_{135},\
\cH^{4,2}_{126}
\end{equation*}
are the proper transforms of irreducible divisors from $\P^2 \times \P^2$.

Using the above observations we can write down the resolution $\rho_{4,2}$ as a blow up.
Set
\begin{equation*}
\cH^{4,2}_+ =
\cH^{4,2}_{123} +
\cH^{4,2}_{156} +
\cH^{4,2}_{246} +
\cH^{4,2}_{345},
\quad\text{and}\quad
\cH^{4,2}_- =
\cH^{4,2}_{456} +
\cH^{4,2}_{234} +
\cH^{4,2}_{135} +
\cH^{4,2}_{126}.
\end{equation*}
Then the divisor $-\cH^{4,2}_+$ is $\beta$-ample.
Since~\mbox{$\rkPic(\cY_{4,2})^{\fS_{4,2}}=2$} by definition~\eqref{eq:y42}
(indeed, the group $\fS_{4,2}$ acts transitively on the set $\{\bP_1,\bP_2,\bP_3,\bP_4\}$
and swaps the factors of~\mbox{$\P(W_3) \times \P(W_3)$})
the divisor $\cH^{4,2}_+$ is $\rho_{4,2}$-ample, so that the
divisor~\mbox{$-\cH^{4,2}_-$} is also $\rho_{4,2}$-ample.
We conclude that the small birational morphism~$\rho_{4,2}$
is the blow up of the Weil divisor $\cH_{456} + \cH_{234} + \cH_{135} + \cH_{126}$ on $\cY$.
Note that the other choice of an $\fS_{4,2}$-equivariant small resolution of singularities
of~$\cY$, that is, the morphism $\sigma\circ\rho_{4,2}$, is the blow up of the Weil divisor
$\cH_{123} + \cH_{156} + \cH_{246} + \cH_{345}$ on~$\cY$.
\end{remark}

\subsection{$\P^2$-bundle over the quintic del Pezzo surface}
\label{subsection:y51}

In this section we construct another resolution of the Coble fourfold, using geometry of the quintic del Pezzo surface.
Before explaining the construction, we start with recalling this geometry
(we refer the reader to~\cite[\S8.5]{Dolgachev-CAG} and~\cite[\S6.2]{CheltsovShramovBook} for more details).

Let $S$ be the (smooth) del Pezzo surface of degree $5$.
Recall that $S$ can be represented as the blow up of $\P^2$ in four points (in five different ways), and
one has~\mbox{$\Aut(S) \cong \fS_5$}.
The vector space~\mbox{$H^0(S,\omega_S^{-1})$} is the unique irreducible six-dimensional representation of~$\fS_5$
(corresponding to the partition $(3,1,1)$ in the notation of~\mbox{\cite[\S4.1]{FultonHarris}}),
see~\mbox{\cite[Lemma~1]{ShepherdBarron-Inv}}; in particular, this representation is invariant under the sign twist.
Moreover, the anticanonical line bundle $\omega_S^{-1}$ is very ample and defines an $\fS_5$-equivariant embedding
\begin{equation*}
S \hookrightarrow \P^5 = \P\left(H^0\left(S,\omega_S^{-1}\right)^\vee\right)
\end{equation*}
such that $S$ is an intersection of five quadrics in $\P^5$.
The five-dimensional space of quadrics passing through $S$ in $\P^5$ is an irreducible representation of $\fS_5$,
see~\cite[Proposition~2]{ShepherdBarron-Inv}.
We denote by
\begin{equation}
\label{eq:quadrics-through-S}
\WW_5 := H^0(\P^5,I_S(2))^\vee
\end{equation}
its dual space.
Later, we will identify this space with the space defined by~\eqref{eq:w-w3w3}.

Below we consider the Grassmannian $\Gr(2,\WW_5^\vee) \cong \Gr(3,\WW_5)$ of two-dimensional vector subspaces in $\WW_5^\vee$ (respectively,  three-dimensional subspaces in $\WW_5$)
and denote by~$\cU_2$ and~$\cU_3$ the tautological rank 2 and rank 3 subbundles in the trivial vector bundles
on this Grassmannian with fibers $\WW_5^\vee$ and $\WW_5$, respectively.

The following result is well known.

\begin{lemma}\label{lemma:dp5-gr}
There is an $\fS_5$-equivariant linear embedding $\P^5 \subset \P(\Lambda^3 \WW_5)$ such that
\begin{equation*}
S = \Gr(3,\WW_5) \cap \P^5 \subset \P(\Lambda^3 \WW_5)
\end{equation*}
is a complete intersection of the Grassmannian $\Gr(3,\WW_5)$ with $\P^5$.
\end{lemma}
\begin{proof}
We use the technique of excess conormal bundles developed in~\cite[Appendix~A]{DK}.
Since $S$ is an intersection of quadrics, the composition
\begin{equation*}
\WW_5^\vee \otimes \cO_{\P^5} \to I_S(2) \to (I_S/I_S^2)(2)
\end{equation*}
is surjective.
The conormal sheaf $\cN_{S/\P^5}^\vee \cong I_S/I_S^2$ is locally free of rank 3 on $S$,
hence the above surjection induces an $\fS_5$-equivariant map $S \to \Gr(3,\WW_5)$ such that
the pullback of the dual tautological bundle~$\cU_3^\vee$ from $\Gr(3,\WW_5)$ to $S$ is isomorphic to $(I_S/I_S^2)(2)$.
By adjunction formula we have
\[
\det(I_S/I_S^2) \cong \omega_{\P^5\vert_S} \otimes \omega_S^{-1} \cong \det(\WW_5^\vee) \otimes \omega_S^6 \otimes \omega_S^{-1},
\]
hence
\begin{equation*}
\det((I_S/I_S^2)(2)) \cong \det(\WW_5^\vee) \otimes \omega_S^{-1},
\end{equation*}
hence the pullback of
$\cO_{\Gr(3,\WW_5)}(1) \cong \det(\cU_3^\vee)$
to $S$ is isomorphic to $\det(\WW_5^\vee) \otimes \omega_S^{-1}$.
The induced map
\begin{equation*}
\Lambda^3\WW_5^\vee \cong H^0\big(\Gr(3,\WW_5),\cO_{\Gr(3,\WW_5)}(1)\big) \to H^0\big(S,\det(\WW_5^\vee) \otimes \omega_S^{-1}\big) \cong \det(\WW_5^\vee) \otimes H^0(S,\omega_S^{-1})
\end{equation*}
is $\fS_5$-equivariant and surjective (since the target space is an irreducible $\fS_5$-representation).
Moreover, since the $\fS_5$-representation $H^0(S,\omega_S^{-1})$ is invariant under the sign twist,
the above composition defines an embedding
\begin{equation*}
\P^5 = \P(H^0(S,\omega_S^{-1})^\vee) \hookrightarrow \P(\Lambda^3\WW_5)
\end{equation*}
such that $S \subset \Gr(3,\WW_5) \cap \P^5$.
It remains to show that this embedding of $S$ is an equality.

Since $\Gr(3,\WW_5) \subset \P(\Lambda^3\WW_5)$ is cut out by Pl\"ucker quadrics that are parameterized by the space $\WW_5^\vee {}\otimes \det(\WW_5^\vee)$, we obtain a map
(where the first isomorphism takes place by~\mbox{\cite[Proposition~A.7]{DK}})
\begin{equation}\label{eq:triv-det}
\WW_5^\vee {}\otimes \det(\WW_5^\vee) \cong H^0\big(\P(\Lambda^3\WW_5),I_{\Gr(3,\WW_5)}(2)\big) \to H^0(\P^5,I_S(2)) \cong \WW_5^\vee
\end{equation}
which by construction commutes with the natural $\fS_5$-action.
It is non-zero since~$\Gr(3,\WW_5)$ does not contain $\P^5$, hence it is an isomorphism by irreducibility of $\WW_5$.
Since~$S$ is an intersection of quadrics, it follows that $S = \Gr(3,\WW_5) \cap \P^5$.
\end{proof}

\begin{remark}[{cf. \cite[Corollary~3]{ShepherdBarron-Inv}}]
\label{remark:w-det}
In~\eqref{eq:triv-det} we obtained an $\fS_5$-equivariant
isomorphism~\mbox{$\WW_5^\vee \otimes \det(\WW_5^\vee) \cong \WW_5^\vee$}.
This allows to identify $\WW_5$ as the (unique) irreducible five-dimensional
representation of $\fS_5$ with $\det(\WW_5)$ being trivial.
It corresponds to the Young diagram of the partition $(3,2)$
in the notation of~\cite[\S4.1]{FultonHarris}.
\end{remark}

We denote the restriction of the tautological bundles $\cU_2$ and $\cU_3$ to $S$ also by $\cU_2$ and~$\cU_3$.
The tautological embeddings $\cU_2 \hookrightarrow \WW_5^\vee \otimes \cO_S$ and $\cU_3 \hookrightarrow \WW_5 \otimes \cO_S$ induce $\fS_5$-equivariant maps
\begin{equation*}
\P_S(\cU_2) \to \P(\WW_5^\vee)
\qquad \text{and} \qquad
\P_S(\cU_3) \to \P(\WW_5).
\end{equation*}
Below we describe these maps explicitly.
We start with the first of them.

\begin{lemma}\label{lemma:segre-cubic}
The image of the map $\varpi \colon \P_S(\cU_2) \to \P(\WW_5^\vee)$ is the Segre cubic hypersurface in~$\P(\WW_5^\vee) \cong \P^4$,
and $\P_S(\cU_2)$ provides its small $\fS_5$-equivariant resolution of singularities.
\end{lemma}
\begin{proof}
Let us describe the fiber of $\varpi$ over a point of $\P(\WW_5^\vee)$.
Thinking of such a point as of a four-dimensional subspace $U_4 \subset \WW_5$, we conclude that
\begin{equation*}
\varpi^{-1}([U_4]) = \Gr(3,U_4) \cap \P^5 \subset \Gr(3,\WW_5) \cap \P^5 = S.
\end{equation*}
Since $\Gr(3,U_4) \cong \P^3$, this intersection is a linear space contained in $S$, hence either is empty, or is a point, or is a line.
Conversely, if $L \subset S$ is a line, then
\begin{equation*}
\cU_2\vert_L \cong \cO_L \oplus \cO_L(-1)
\end{equation*}
because $\cU_2^\vee$ is globally generated with $\det(\cU_2^\vee) \cong \omega_S^{-1}$.
Moreover, the section
\begin{equation*}
L = \P_L(\cO_L) \hookrightarrow \P_L(\cU_2\vert_L) \hookrightarrow \P_S(\cU_2)
\end{equation*}
of the projection $\P_L(\cU_2\vert_L) \to L$
is contracted by the map $\varpi$.
This proves that $\varpi$ contracts precisely the exceptional sections over the ten lines of $S$, hence the image
\begin{equation*}
Z {}:={} \varpi(\P_S(\cU_2)) \subset \P(\WW_5^\vee)
\end{equation*}
is a hypersurface with ten isolated singular points and the map $\P_S(\cU_2) \to Z$ is a small resolution of singularities.
On the other hand, since $\det(\cU_2) \cong \omega_S$, it follows that
\begin{equation*}
\omega_{\P_S(\cU_2)} \cong \varpi^*\cO_{\P(\WW_5^\vee)}(-2).
\end{equation*}
Since the map $\P_S(\cU_2) \to Z$ is small, we have
$\omega_Z \cong \cO_{\P(\WW_5^\vee)}(-2)\vert_Z$, so that $Z$ is a cubic hypersurface.
It remains to notice that the only three-dimensional cubic with ten isolated singular points is the Segre cubic,
see e.g.~\cite[Proposition~2.1]{Dolgachev-Segre};
alternatively, one can deduce this from the fact that the group $\fS_5$
acting in the irreducible five-dimensional representation~$\WW_5^\vee$ has a unique cubic invariant,
which must thus define the Segre cubic.
\end{proof}

\begin{remark}[{cf.~\cite[\S2]{Dolgachev-Segre}, \cite[Proposition~4.6]{Prokhorov-FieldsOfInv}}]
\label{remark:M06}
The relation of the quintic del Pezzo surface $S$ and the Segre cubic threefold $Z$ extends to an $\fS_5$-equivariant diagram
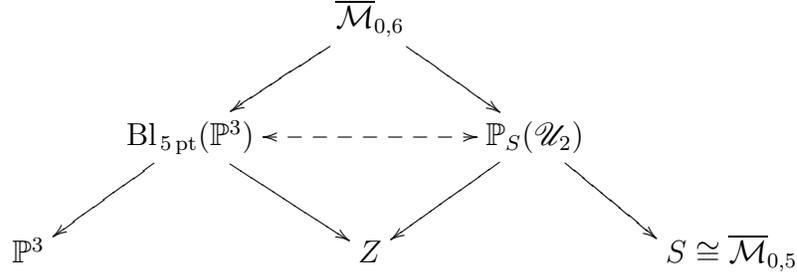
\begin{equation*}
\vcenter{\xymatrix{
&& \bcM_{0,6} \ar[dl] \ar[dr] \\
&
\Bl_{\textrm{\,5\,pt}}(\P^3) \ar[dl] \ar[dr] \ar@{<-->}[rr] &&
\P_S(\cU_2) \ar[dl] \ar[dr] \\
\P^3 && Z && S \hbox to 0pt{${} \cong \bcM_{0,5}$\hss}
}}
\end{equation*}
Here $\bcM_{0,n}$ is the moduli spaces of stable rational curves with~$n$ marked points,
the left outer diagonal arrows provide its Kapranov's representation (the lower left arrow is the blow up of five general points on $\P^3$),
the right outer diagonal arrows compose to the forgetful map $\bcM_{0,6} \to \bcM_{0,5}$,
the inner diagonal arrows contract ten smooth rational curves each (and provide two $\fS_5$-equivariant small resolutions of $Z$),
and the dashed arrow is a flop in these curves.
\end{remark}

The above diagram can be thought of as an $\fS_5$-Sarkisov link from the Mori fiber
space~\mbox{$\P_S(\cU_2) \to S$} to $\P^3$ centered at $Z$,
see~\S\ref{subsection:Cl-and-applications} below for explanation of terminology.
It is natural to ask what is the $\fS_5$-Sarkisov link starting from $\P_S(\cU_3) \to S$.
We will see in diagram~\eqref{eq:diagram-two-factorializations-Y} below that it is a symmetric link centered at the Coble fourfold~$\cY$.

So, we consider the projectivization $\P_S(\cU_3)$ of the rank 3 bundle $\cU_3$ and denote it by
\begin{equation*}
\cY_{5,1} := \P_S(\cU_3).
\end{equation*}
The embedding $\cU_3 \hookrightarrow \WW_5 \otimes \cO_S$ induces an $\fS_5$-equivariant diagram
\begin{equation}\label{eq:y51}
\vcenter{\xymatrix{
& \cY_{5,1} \ar[dl]_p \ar[dr]^{\pi_{5,1}} \\
S && \P(\WW_5)
}}
\end{equation}
where $p$ is the natural projection $\cY_{5,1} = \P_S(\cU_3) \to S$,
and $\pi_{5,1}$ is the composition of the embedding~\mbox{$\cY_{5,1}\hookrightarrow S\times\P(\WW_5)$} with the projection
to the second factor. In particular, the restriction of the map $p$ to any fiber of $\pi_{5,1}$ is an isomorphism to its image.
This allows to consider every fiber
\begin{equation*}
S_w := \pi_{5,1}^{-1}(w)
\end{equation*}
of the map $\pi_{5,1}$ as a closed subscheme of~$S$.
In the next lemma we describe these subschemes.

For each point $w \in \P(\WW_5)$ denote by $\WW_5/w$ the four-dimensional quotient of~$\WW_5$ by the line in~$\WW_5$ that corresponds to $w$.
Every two-dimensional subspace in $\WW_5/w$ gives (by taking preimage) a three-dimensional subspace in $\WW_5$ containing $w$.
This allows to consider $\Gr(2,\WW_5/w)$ as a subvariety of $\Gr(3,\WW_5)$.

\begin{lemma}
\label{lemma:s-w}
The fiber $S_w$ of the map $\pi_{5,1}$ over a point $w \in \P(\WW_5)$ can be described as
\begin{equation*}
S_w = \Gr(2,\WW_5/w) \cap \P^5 \subset \Gr(3,\WW_5) \cap \P^5 = S.
\end{equation*}
In particular, $S_w$ is either a zero-dimensional scheme of length $2$, or a line,
or a conic.
\end{lemma}
\begin{proof}
The first equality is obvious.
Consequently, $S_w$ is a linear section of the four-dimensional quadric $\Gr(2,\WW_5/w)$ of codimension at most~4.
So, if $S_w$ is zero-dimensional, it is a scheme of length 2.
Furthermore, if $S_w$ is one-dimensional, it is either a line or a conic.
It remains to notice that $\dim S_w < \dim S = 2$ since $S$ is irreducible.
\end{proof}

Our goal is to describe the map $\pi_{5,1}$ in~\eqref{eq:y51}.
We start by presenting some surfaces in~$\cY_{5,1}$ contracted by it.
Recall that $S$ contains 10 lines.
Recall also that $\cU_3$ is a subbundle in the trivial vector bundle with fiber $\WW_5$ over~$S$, so that $\cY_{5,1}$ is a subvariety in~\mbox{$S\times\P(\WW_5)$}.

\begin{lemma}\label{lemma:r-l}
For every line $L \subset S$ there is a unique line $L' \subset \P(\WW_5)$ such that for the surface $R_L = L\times L'$ one has
\begin{equation}\label{eq:r-l}
R_L \subset \cY_{5,1} \subset S \times \P(\WW_5).
\end{equation}
In particular, the map $\pi_{5,1}$ contracts $R_L$ onto the line $L'$.
Moreover,
if $L_1 \ne L_2$ are distinct lines on $S$ then the corresponding lines $L'_1,L'_2 \subset \P(\WW_5)$ are distinct as well.
\end{lemma}
\begin{proof}
Since $L$ is a line on $\Gr(3,\WW_5)$, there is a unique two-dimensional subspace $U_2 \subset \WW_5$ such that $L \subset \P(\WW_5/U_2) \subset \Gr(3,\WW_5)$.
Then for every point $[U_3]$ of $L$ we have $U_2 \subset U_3$, that is, $U_2 \otimes \cO_L \subset \cU_3\vert_L$,
hence
\begin{equation*}
L \times \P(U_2) = \P_L(U_2 \otimes \cO_L) \subset \P_S(\cU_3)  = \cY_{5,1}.
\end{equation*}
Thus, the line $L' = \P(U_2)  \subset \P(\WW_5)$ has the required property.

Furthermore, for any two-dimensional subspace $U_2 \subset \WW_5$ the intersection
\begin{equation*}
\P(\WW_5/U_2) \cap S = \P(\WW_5/U_2) \cap \P^5
\end{equation*}
is a linear space contained in $S$, hence is either empty, or a point, or a line.
In particular, two distinct lines $L_1$ and $L_2$ on $S$ cannot correspond to the same subspace $U_2 \subset \WW_5$,
hence the corresponding lines $L'_1$ and $L'_2$ in $\P(\WW_5)$ are distinct.
\end{proof}

As we already mentioned,
a quintic del Pezzo surface is classically represented as the blow up of $\P^2$ in four general points.
Let $\varphi \colon S \to \P^2$ be one of such blow up representations with exceptional divisors
$E_0$, $E_1$, $E_2$, and~$E_3$.
Denote by $e_i$ their classes in $\Pic(S)$, and by~$\ell$ the pullback of the line class from~$\P^2$ to~$S$, so that
\begin{equation*}
K_S \sim -3\ell + e_0 + e_1 + e_2 + e_3.
\end{equation*}
The line bundle $\cO_S(\ell)$ defines the contraction $\varphi \colon S \to \P^2$ and
the line bundle~\mbox{$\cO_S(2\ell - e_0 - e_1 - e_2 - e_3)$} defines a conic bundle $\bar\varphi \colon S \to \P^1$.
The combination of~$\varphi$ and $\bar\varphi$ defines an embedding
\begin{equation*}
\varphi \times \bar\varphi \colon S \hookrightarrow \P^2 \times \P^1,
\end{equation*}
whose image is a divisor of bidegree $(2,1)$.
Moreover, the composition of $\varphi \times \bar\varphi$ with the Segre embedding $\P^2 \times \P^1 \hookrightarrow \P^5$ is the anticanonical embedding of $S$, therefore we have an exact sequence of normal bundles
\begin{equation}\label{eq:norm-bundles}
0 \to \cN_{S/\P^2 \times \P^1} \to  \cN_{S/\P^5} \to \cN_{\P^2\times\P^1/\P^5}\vert_S \to 0.
\end{equation}
The first of these bundles is isomorphic to
\begin{equation*}
\varphi^*\cO_{\P^2}(2) \otimes \bar\varphi^*\cO_{\P^1}(1) \cong \cO_S(4\ell - e_0 - e_1 - e_2 - e_3),
\end{equation*}
and the second is isomorphic to $\cU_3(6\ell - 2e_0 - 2e_1 - 2e_2 - 2e_3)$ by the proof of Lemma~\ref{lemma:dp5-gr}.
The third vector bundle in~\eqref{eq:norm-bundles} can be computed as follows. Here we denote by~$\mathscr{T}_{\mathbb{P}^n}$ the tangent bundle of~$\mathbb{P}^n$.

\begin{lemma}\label{lemma:simple-computation}
For any positive integers $m,n$ we have $\cN_{\P^m\times\P^n/\P^{mn + m + n}} \cong \cT_{\P^m} \boxtimes \cT_{\P^n}$.
\end{lemma}
\begin{proof}
Let $A$ and $B$ be vector spaces of dimension $m+1$ and $n+1$ respectively.
Tensoring pullbacks to $\P(A) \times \P(B)$ of the Euler sequences
\begin{equation*}
0 \to \cO_{\P(A)} \to A \otimes \cO_{\P(A)}(1) \to \cT_{\P(A)} \to 0
\quad\text{and}\quad
0 \to \cO_{\P(B)} \to B \otimes \cO_{\P(B)}(1) \to \cT_{\P(B)} \to 0,
\end{equation*}
we obtain an exact sequence
\begin{multline*}
0 \to \cO_{\P(A) \times \P(B)} \to A \otimes \cO_{\P(A) \times \P(B)}(1,0) \oplus B \otimes \cO_{\P(A) \times \P(B)}(0,1) \\
\to A \otimes B \otimes \cO_{\P(A) \times \P(B)}(1,1) \to \cT_{\P(A)} \boxtimes \cT_{\P(B)} \to 0.
\end{multline*}
Comparing it with the restriction to $\P(A) \times \P(B)$ of the Euler sequence
\begin{equation*}
0 \to \cO_{\P(A) \times \P(B)} \to A \otimes B \otimes \cO_{\P(A) \times \P(B)}(1,1) \to \cT_{\P(A \otimes B)}\vert_{\P(A) \times \P(B)} \to 0
\end{equation*}
of $\P(A \otimes B)$ and with the pullbacks of the Euler sequences of $\P(A)$ and $\P(B)$, we obtain an exact sequence
\begin{equation*}
0 \to \operatorname{pr}_{\P(A)}^* \cT_{\P(A)} \oplus \operatorname{pr}_{\P(B)}^* \cT_{\P(B)}
\to \cT_{\P(A \otimes B)}\vert_{\P(A) \times \P(B)} \to \cT_{\P(A)} \boxtimes \cT_{\P(B)} \to 0,
\end{equation*}
where $\operatorname{pr}_{\P(A)}$ and $\operatorname{pr}_{\P(B)}$ are the projections, which proves the lemma.
\end{proof}

Applying Lemma~\ref{lemma:simple-computation} in the case $m = 2$, $n = 1$, we see that the third bundle in~\eqref{eq:norm-bundles}
is isomorphic to
\begin{equation*}
\varphi^*(\cT_{\P^2}) \otimes \bar\varphi^*(\cT_{\P^1}) \cong \varphi^*\cT_{\P^2} \otimes \cO_S(4\ell - 2e_0 - 2e_1 - 2e_2 - 2e_3).
\end{equation*}
So, twisting the normal bundle sequence~\eqref{eq:norm-bundles} by the line bundle~\mbox{$\cO_S(-6\ell + 2e_0 + 2e_1 + 2e_2 + 2e_3)$} we obtain
\begin{equation}\label{eq:u3-sequence}
0 \to \cO_S(-2\ell + e_0 + e_1 + e_2 + e_3) \to \cU_3 \to \varphi^*(\cT_{\P^2}(-2)) \to 0
\end{equation}
Denote by $r_\varphi \colon S \to \P_S(\cU_3)$ the section of the projection $p$ induced by the first map in~\eqref{eq:u3-sequence}.

\begin{lemma}\label{lemma:r-i}
There is a line $\Gamma_\varphi \subset \P(\WW_5)$ and a commutative diagram
\begin{equation*}
\xymatrix{
S \ar[r]^-{r_\varphi} \ar[d]_{\bar\varphi}  & \P_S(\cU_3) \ar[d]^{\pi_{5,1}} \\
\Gamma_\varphi \ar@{^{(}->}[r] & \P(\WW_5)
}
\end{equation*}
that identifies the line $\Gamma_\varphi$ with the base of the conic bundle $\bar\varphi$.
In particular, for any~\mbox{$w \in \Gamma_\varphi$} the fiber $S_w = \pi_{5,1}^{-1}(w)$ is a conic from the pencil~$\bar\varphi$.
\end{lemma}
\begin{proof}
By definition of $r_\varphi$ the composition
\begin{equation*}
\pi_{5,1} \circ r_\varphi \colon S \to \P(\WW_5)
\end{equation*}
is given by the line bundle $\cO_S(2\ell - e_0 - e_1 - e_2 - e_3)$ on $S$, hence factors as the projection~$\bar\varphi$ followed by a linear embedding.
This proves that we have the required diagram.
Moreover, it follows that for every $w \in \Gamma$ the fiber $\pi_{5,1}^{-1}(w)$ contains a conic from the pencil $\bar\varphi$.
By Lemma~\ref{lemma:s-w} the fiber coincides with this conic.
\end{proof}

For each contraction $\varphi \colon S \to \P^2$ (recall that for a quintic del Pezzo surface $S$ there are five such contractions),
define the surface
\begin{equation}\label{eq:r-i}
R_\varphi = r_\varphi(S) \subset \P_S(\cU_3),
\end{equation}
so that the map $\pi_{5,1}$ contracts it
onto the line $\Gamma_\varphi \subset \P(\WW_5)$.

\begin{lemma}\label{lemma:gamma-l}
The five lines $\Gamma_\varphi \subset \P(\WW_5)$ corresponding to the contractions~\mbox{$\varphi \colon S \to \P^2$} are pairwise disjoint.
Moreover, for each $\varphi$ the line $\Gamma_\varphi$ is distinct from the lines~\mbox{$L' \subset \P(\WW_5)$}
associated with lines $L$ on $S$ in Lemma~\textup{\ref{lemma:r-l}}.
\end{lemma}
\begin{proof}
If $w$ is a common point of the curves $\Gamma_\varphi$
and $\Gamma_{\varphi'}$, then by Lemma~\ref{lemma:r-i}
the fiber~$S_w$ is a conic that belongs to the corresponding pencils
$\bar\varphi$ and $\bar\varphi'$,
hence the pencils coincide, hence $\varphi = \varphi'$.

Assume that $\Gamma_\varphi = L'$, where $L'$ is associated with some line $L \subset S$ as in Lemma~\ref{lemma:r-l}.
By Lemma~\ref{lemma:r-l} we have~\mbox{$L \subset S_w$} for each $w \in L' = \Gamma_\varphi$,
and by Lemma~\ref{lemma:r-i} when $w$ runs over $\Gamma_\varphi$ the curves $S_w$ run over the corresponding pencil of conics $\bar\varphi$.
So, the assumption we made implies that every conic in the pencil contains the line $L$, which is absurd.
\end{proof}

Now we are ready to prove the main result of this subsection.

\begin{proposition}\label{proposition:Stein-factorization}
The $\fS_5$-equivariant morphism $\pi_{5,1} \colon \cY_{5,1} \to \P(\WW_5)$ gives rise to a
commutative diagram
\begin{equation}
\label{eq:Y-Y51}
\vcenter{\xymatrix{
&
\cY_{5,1} \ar[rr]^-{\rho_{5,1}} \ar[dr]_{\pi_{5,1}} \ar[dl]_{p} &&
\cY \ar[dl]^\pi
\\
S &&
\P(\WW_5)
}}
\end{equation}
where $\cY$ is the Coble fourfold, $\pi \colon \cY \to \P(\WW_5)$ is the double covering, and $\rho_{5,1}$ is a small resolution of singularities,
defined uniquely up to composition with the Galois involution~\mbox{$\sigma \colon \cY \to \cY$}.
Furthermore, the exceptional locus of~$\rho_{5,1}$ is the union of~$15$ irreducible rational surfaces
$\{R_L\}_{L \subset S} \cup \{ R_\varphi \}_{\varphi \colon S \to \P^2}$,
such that
\begin{itemize}
\item
$R_L \cong \P^1 \times \P^1$;
each of these surfaces is contracted by $p$ onto the line $L \subset S$ and by $\pi_{5,1}$ onto the line $L' \subset \P(\WW_5)$;
\item
$R_\varphi \cong S$ with the map $p \colon R_\varphi \to S$ being an isomorphism, and with the map $\pi_{5,1}\vert_{R_\varphi}$ being
the conic bundle $\bar\varphi \colon R_\varphi \twoheadrightarrow \Gamma_\varphi$ over the line $\Gamma_\varphi \subset \P(\WW_5)$.
\end{itemize}
Moreover, the morphism $\pi_{5,1}$ induces a non-standard embedding $\fS_5 \to \fS_6$ such that $\rho_{5,1}$ is $\fS_5$-equivariant
with respect either to the natural or to the twisted action of $\fS_{5}$ on $\cY$.
\end{proposition}

Using a compatibility result from Proposition~\ref{proposition:compatibility}, we will
show in~\S\ref{subsection:proof-first} that~$\rho_{5,1}$ is $\fS_5$-equivariant with respect to the twisted action
of a non-standard~$\fS_5$.

\begin{proof}
By Lemma~\ref{lemma:s-w} the map $\pi_{5,1}$ is generically finite of degree~2.
Denote by $R \subset \cY_{5,1}$ the ramification locus of the morphism $\pi_{5,1} \colon \cY_{5,1} \to \P(\WW_5)$ and by $B = \pi_{5,1}(R) \subset \P(\WW_5)$ its image.
Let us show that $B$ is the Igusa quartic.
For this we show that $B$ is projectively dual to the Segre cubic
$Z = \varpi(\P_S(\cU_2))$, see Lemma~\ref{lemma:segre-cubic}.

Indeed, by Lemma~\ref{lemma:s-w} we know that $B$ is the locus of $w \in \P(\WW_5)$ such that $S_w$ is either a double point or a curve.
On the other hand, $w$ defines a hyperplane $\P(w^\perp) \subset \P(\WW_5^\vee)$ in the dual projective space, and
\begin{equation*}
\varpi^{-1}(Z \cap \P(w^\perp)) = \P_S(\cU_2) \times_{\P(\WW_5^\vee)} \P(w^\perp)
\end{equation*}
is a relative hyperplane in the $\P^1$-bundle $\P_S(\cU_2) \to S$.
Moreover, the zero locus of the corresponding section of $\cU_2^\vee$ is precisely the scheme $S_w$.
If $S_w$ is zero-dimensional then by~\cite[Lemma~2.1]{Kuznetsov} we have
\begin{equation*}
\varpi^{-1}\big(Z \cap \P(w^\perp)\big) = \Bl_{S_w}(S),
\end{equation*}
and if it is one-dimensional, then $\varpi^{-1}(Z \cap \P(w^\perp))$ contains the surface $\P_{S_w}(\cU_2\vert_{S_w})$, hence is reducible.
Thus, $\varpi^{-1}(Z \cap \P(w^\perp))$ is singular if and only if $w \in B$.
Since the singular points of $Z$ are nodes, and $\varpi$ resolves them, it follows that $B$ is the projective dual of $Z$. Hence $B = X_{1/4}$ is the Igusa quartic (see~\cite[Proposition~3.3.1]{Hunt}).

It follows from Lemma~\ref{lemma:s-w} that the map $\pi_{5,1}$ is an \'etale double cover over $\P(\WW_5) \setminus B$, and that the Stein factorization of the map $\pi_{5,1}$
provides a (unique up to $\sigma$) decomposition
\begin{equation*}
\cY_{5,1} \xrightarrow{\ \rho_{5,1}\ } \cY \xrightarrow{\ \pi\ } \P(\WW_5),
\end{equation*}
where $\rho_{5,1}$ is a birational map.

Let us show that $\rho_{5,1}$ is small.
Indeed, since $\det(\cU_3) \cong \omega_S$, it follows that
\begin{equation}\label{eq:omega-pullback}
\omega_{\cY_{5,1}} \cong \pi_{5,1}^*\cO_{\P(\WW_5)}(-3) \cong \rho_{5,1}^*\pi^*\cO_{\P(\WW_5)}(-3).
\end{equation}
On the other hand, $\pi$ is a double covering branched over a quartic, hence one has~\mbox{$\omega_\cY \cong \pi^*\cO_{\P(\WW_5)}(-3)$}.
Thus $\omega_{\cY_{5,1}} \cong \rho_{5,1}^*\omega_\cY$, i.e., the map $\rho_{5,1}$ is crepant.
Since $\cY_{5,1}$ is smooth
it follows that the map $\rho_{5,1}$ is an isomorphism over the smooth locus of $\cY$, hence the exceptional locus of $\rho_{5,1}$ is contained in
\begin{equation*}
\rho_{5,1}^{-1}\big(\Sing(\cY)\big) = \pi_{5,1}^{-1}\big(\Sing(X_{\frac14})\big) = \pi_{5,1}^{-1}\big(\CR\big),
\end{equation*}
i.e., in the preimage of the Cremona--Richmond configuration of 15 lines.
But by Lemma~\ref{lemma:s-w} the fibers of $\pi_{5,1}$ are at most one-dimensional, hence $\dim(\pi_{5,1}^{-1}(\CR)) \leqslant 2$.
This proves that~$\rho_{5,1}$ is small.

Next, let us show that
\begin{equation}
\label{eq:preimage-cr}
\pi_{5,1}^{-1}\left(\CR\right) = \left(\bigcup_{\varphi} R_\varphi \right) \cup \left(\bigcup_{L} R_L \right).
\end{equation}
By Lemmas~\ref{lemma:r-l} and~\ref{lemma:r-i} the surfaces $R_L$ and $R_\varphi$ are contracted onto the union of ten lines $L'$ and five lines $\Gamma_\varphi$ in $\P(\WW_5)$,
which are pairwise distinct by Lemmas~\ref{lemma:r-l} and~\ref{lemma:gamma-l}.
Therefore
\begin{equation*}
\CR = \left(\bigcup_{\varphi} \Gamma_\varphi \right) \cup \left( \bigcup_L L' \right).
\end{equation*}
It remains to show that for any $w \in \Gamma_\varphi$ or $w \in L'$ the fiber $S_w = \pi_{5,1}^{-1}(w)$ is contained either in $R_\varphi$ or in $R_L$.
If $w \in \Gamma_\varphi$, this is proved in Lemma~\ref{lemma:r-i}.
Now take $w \in L'$.
By Lemma~\ref{lemma:r-l} we have $L \subset S_w$, hence by Lemma~\ref{lemma:s-w} the curve $S_w$ is either the line $L$ (hence $S_w \subset R_L$) or a conic
(hence $S_w \subset R_\varphi$ for appropriate $\varphi$).
This proves~\eqref{eq:preimage-cr}.

The vector space $\WW_5$ by definition~\eqref{eq:quadrics-through-S} comes with a natural $\fS_5$-action and, moreover, the map $\pi_{5,1} \colon \cY_{5,1} \to \P(\WW_5)$ is $\fS_5$-equivariant.
It follows that its branch divisor $B = X_{1/4}$
is invariant under this action.
This gives an embedding
\begin{equation*}
\fS_5 \hookrightarrow \Aut(X_{1/4}) \cong \fS_6 \subset \Aut(\cY),
\end{equation*}
such that for every element $g \in \fS_5$
the conjugation of the diagram~\eqref{eq:Y-Y51} by~$g$ gives a diagram of the same form.
Therefore, one has
\begin{equation*}
g \circ \rho_{5,1} \circ g^{-1} = \sigma^{k(g)} \circ \rho_{5,1},
\end{equation*}
where $k \colon \fS_5 \to \ZZ/2\ZZ$ is a group homomorphism.
If it is trivial, then $\rho_{5,1}$ is equivariant with respect to the natural action, and
if $k$ is the homomorphism of parity, then $\rho_{5,1}$ is equivariant with respect to the twisted action
(as we mentioned above, we will show in~\S\ref{subsection:proof-first} that $k$ is indeed the homomorphism of parity).

To show that the embedding $\fS_5 \hookrightarrow \fS_6$ is non-standard we use the same argument as in the proof of Proposition~\ref{proposition:first-Q-factorialization}.
The restriction of the five-dimensional representation~\eqref{eq:hyperplane}
to the image of a standard embedding $\fS_5 \hookrightarrow \fS_6$
decomposes as a direct sum of two irreducible representations (cf.~Lemma~\ref{lemma:restriction-s6-s5}),
while the $\fS_5$-representation $\WW_5$ is irreducible by~\eqref{eq:quadrics-through-S} and \cite[Proposition~2]{ShepherdBarron-Inv}.
\end{proof}

Similarly to the case of $\rho_{4,2}$, the morphism $\rho_{5,1}$ is not uniquely defined
even when the corresponding non-standard subgroup $\fS_5$ is fixed.
Moreover, there is a commutative diagram
\begin{equation}
\label{eq:diagram-two-factorializations-Y}
\vcenter{\xymatrix{
&
\cY_{5,1} \ar@{-->}^{\rho_{5,1}^{-1} \circ \sigma \circ \rho_{5,1}}[rr] \ar@{->}[rd]_{\sigma \circ \rho_{5,1}} \ar[dl]_p &&
\cY_{5,1} \ar@{->}[ld]^{\rho_{5,1}} \ar[dr]^p
\\
S &&
\cY
&&
S
}}
\end{equation}
Here $\rho_{5,1}^{-1} \circ \sigma \circ \rho_{5,1}$ is a small birational map.
In fact, we know that $\rkPic(S)^{\fS_5}=1$, see for instance \cite[Lemma~6.2.2(i)]{CheltsovShramovBook};
this means that
\begin{equation*}
\rkCl(\cY_{5,1})^{\fS_5} = \rkPic(\cY_{5,1})^{\fS_5}=2,
\end{equation*}
and therefore $\rkPic(\cY)^{\fS_5} = 1$.
The latter implies that $\rho_{5,1}$ and $\sigma\circ\rho_{5,1}$ are the only $\fS_5$-equivariant small resolutions of
singularities of~$\cY$, and that~\mbox{$\rho_{5,1}^{-1} \circ \sigma \circ \rho_{5,1}$} is an $\fS_5$-flop.
Consequently, the diagram~\eqref{eq:diagram-two-factorializations-Y} is an $\fS_5$-Sarkisov link
between two copies of the Mori fiber space $\cY_{5,1} \to S$ centered at~$\cY$ (see~\S\ref{subsection:Cl-and-applications}).

\begin{remark}
\label{remark:rho-5-1-Weil-divisor}
Recall the notation of Remark~\ref{remark:rho-4-2-Weil-divisor}.
Denote
\begin{equation*}
\cH^{5,1}_I := \rho_{5,1}^{-1}(\cH_I),
\end{equation*}
so that one has $\pi_{5,1}^{-1}(H_I) = \pi_{5,1}^{-1}(H_{\bar{I}}) = \cH^{5,1}_I \cup \cH^{5,1}_{\bar{I}}$.
One can check that ten out of twenty divisors $\cH^{5,1}_I \subset \cY_{5,1}$ are the preimages of lines on $S$ via the map $p$,
and the other ten are relative hyperplane sections for $p$
(this decomposition is the orbit decomposition for the action of~$\fS_5$).
We denote by~$\cH^{5,1}_+$ the sum of the divisors of the first type, and by~$\cH^{5,1}_-$ the sum of the divisors of the second type.
The divisor $\cH^{5,1}_+$ is the $p$-pullback of an ample divisor on~$S$, hence it is $\rho_{5,1}$-ample.
Consequently, $-\cH^{5,1}_-$ is $\rho_{5,1}$-ample, hence the small birational morphism $\rho_{5,1}$
is the blow up of the Weil divisor $\rho_{5,1}(\cH^{5,1}_-)$ on~$\cY$.
See Remark~\ref{remark:rho-5-1-Weil-final} below for an explicit description of this blow up.
\end{remark}

\subsection{Compatibility of resolutions}\label{subsection:compatibility}

In this section we relate the resolutions $\cY_{4,2}$ and~$\cY_{5,1}$ of the Coble fourfold.
Recall that the first of them is associated with a non-standard embedding $\fS_{4,2} \hookrightarrow \fS_6$,
and the second is associated with a non-standard embedding $\fS_5 \hookrightarrow \fS_6$.
Note that each (standard or non-standard) subgroup $\fS_4 \subset \fS_6$ can be extended
to a subgroup $\fS_{4,2} \subset \fS_6$ and such extension is unique.
Indeed, the second factor $\fS_2$ in $\fS_{4,2} \cong \fS_4 \times \fS_2$ is just the centralizer of $\fS_4$ in $\fS_6$.
Recall also that for each~$\fS_4 \subset \fS_5 = \Aut(S)$
there is a unique $\fS_4$-equivariant contraction $\varphi \colon S \to \P^2$ of the quintic del Pezzo surface $S$ onto the plane.

\begin{proposition}\label{proposition:compatibility}
Let $\fS_5 \hookrightarrow \fS_6$ be a non-standard embedding.
Choose a subgroup~\mbox{$\fS_4 \subset \fS_5$} and let $\fS_{4,2} \subset \fS_6$ be its unique extension.
Let $\rho_{4,2} \colon \cY_{4,2} \to \cY$ be the $\fS_{4,2}$-equivariant resolution of singularities constructed in Proposition~\textup{\ref{proposition:first-Q-factorialization}}
and let $\varphi \colon S \to \P^2$ be the unique $\fS_4$-equivariant contraction of the quintic del Pezzo surface.
Then there is a unique $\fS_5$-equivariant resolution $\rho_{5,1} \colon \cY_{5,1} \to \cY$ as in Proposition~\textup{\ref{proposition:Stein-factorization}}
and a unique $\fS_4$-equivariant small birational map $\theta_1 \colon \cY_{5,1} \dashrightarrow \cY_{4,2}$ such that the diagram~\eqref{diagram:varphi}
is commutative.
\end{proposition}

Of course, if $\rho_{5,1}$ is fixed, there is only one $\theta_1$ such that the inner triangle in the diagram~\eqref{diagram:varphi} commutes, namely, $\theta_1 = \rho_{4,2}^{-1} \circ \rho_{5,1}$.
But it is a priori not clear why the outer square commutes.
So, to prove Proposition~\ref{proposition:compatibility} we move in the opposite direction: we first construct $\theta_1$ such that the outer square commutes,
and after that check that the inner triangle commutes for this $\theta_1$ and for an appropriate choice of $\rho_{5,1}$.

We start with some notation and a lemma.
Let~\mbox{$\varphi \colon S \to \P^2$} be the $\fS_4$-equivariant contraction, and, as before,
denote by $E_0$, $E_1$, $E_2$, and~$E_3$ the exceptional divisors of the blow up~$\varphi$,
by $e_i$ their classes in $\Pic(S)$ and by $\ell$ the pullback of the line class of $\P^2$.
Recall also the rank $3$ bundle $\cU_3$ on $S$.

Since $\cU_3^\vee$ is globally generated and $\det(\cU_3^\vee)\vert_{E_i} \cong \omega_S^{-1}\vert_{E_i} \cong \cO_{E_i}(1)$, we have
\begin{equation*}
\cU_3\vert_{E_i} \cong \cO_{E_i} \oplus \cO_{E_i} \oplus \cO_{E_i}(-1).
\end{equation*}
Therefore, we have a canonical surjective morphism $\cU_3 \to \cO_{E_i}(-1)$ of sheaves on $S$.
The sum of these morphisms gives an exact $\fS_4$-equivariant sequence
\begin{equation}\label{eq:ce}
0 \to \cE \to \cU_3 \to \bigoplus_{i=0}^3 \cO_{E_i}(-1) \to 0
\end{equation}
and defines a rank $3$ vector bundle $\cE$ on $S$ equivariant with respect to~$\fS_4$.

\begin{lemma}
One has
$\cE \cong \cO_S(-\ell)^{\oplus 3}$.
\end{lemma}

\begin{proof}
Consider the composition of the embedding
\begin{equation*}
\cO_S(-2\ell + e_0 + e_1 + e_2 + e_3) \hookrightarrow \cU_3
\end{equation*}
from~\eqref{eq:u3-sequence} with the projection $\cU_3 \to \cO_{E_i}(-1)$.
If it is equal to zero, then the map~\mbox{$\cU_3 \to \cO_{E_i}(-1)$} factors through a map $\varphi^*(\cT_{\P^2}(-2)) \to \cO_{E_i}(-1)$.
But the sheaf~$\varphi^*(\cT_{\P^2}(-2))$ restricts to~$E_i$ trivially, hence no such map exists.
This contradiction shows that the composition is non-trivial.
But since
\begin{equation*}
\cO_S(-2\ell + e_0 + e_1 + e_2 + e_3)\vert_{E_i} \cong \cO_{E_i}(-1),
 \end{equation*}
any non-trivial morphism
$\cO_S(-2\ell + e_0 + e_1 + e_2 + e_3) \to \cO_{E_i}(-1)$
is surjective.
Therefore, the sum of these morphisms
$\cO_S(-2\ell + e_0 + e_1 + e_2 + e_3) \to \bigoplus_{i=0}^3 \cO_{E_i}(-1)$
is surjective, hence its kernel is $\cO_S(-2\ell)$ and we have a commutative diagram
\begin{equation*}
\xymatrix{
0 \ar[r] & \cO_S(-2\ell) \ar[r] \ar[d] & \cO_S(-2\ell + e_0 + e_1 + e_2 + e_3) \ar[r] \ar[d] & \bigoplus_{i=0}^3 \cO_{E_i}(-1) \ar[r] \ar@{=}[d] & 0
\\
0 \ar[r] & \cE \ar[r] & \cU_3 \ar[r] & \bigoplus_{i=0}^3 \cO_{E_i}(-1) \ar[r] & 0
}
\end{equation*}
Taking into account~\eqref{eq:u3-sequence}, we see that the first column extends to an exact sequence
\begin{equation}\label{eq:e-sequence}
0 \to \cO_S(-2\ell) \to \cE \to \varphi^*(\cT_{\P^2}(-2)) \to 0.
\end{equation}
It remains to show that it coincides with the pullback of a twist of the Euler sequence on~$\P^2$.
Since the pullback functor $\varphi^*$ is fully faithful, and the Euler sequence is the unique non-split extension of $\cT_{\P^2}$ by $\cO_{\P^2}$, it is enough to show that~\eqref{eq:e-sequence} is non-split.

Assume on the contrary that there is a splitting $\varphi^*(\cT_{\P^2}(-2)) \to \cE$.
Composing it with the embedding $\cE \hookrightarrow \cU_3$, we obtain a splitting $\varphi^*(\cT_{\P^2}(-2)) \to \cU_3$ of~\eqref{eq:u3-sequence}.
It induces an embedding
\begin{equation*}
S \times_{\P^2} \Fl(1,2;3) \cong \P_S(\varphi^*(\cT_{\P^2}(-2))) \hookrightarrow \P_S(\cU_3) = \cY_{5,1},
\end{equation*}
such that its composition with $\pi_{5,1}$ coincides with the projection
\begin{equation*}
S \times_{\P^2} \Fl(1,2;3) \to \Fl(1,2;3) \to (\P^2)^\vee.
\end{equation*}
But this contradicts the fact that $\rho_{5,1}$ is a small contraction.
\end{proof}

\begin{proof}[Proof of Proposition~\textup{\ref{proposition:compatibility}}]
Let us construct the map $\theta_1$.
Let $V_1$ be a three-dimensional vector space such that the target plane of $\varphi$ is $\P(V_1)$.
We can choose an isomorphism
\begin{equation*}
\alpha_1 \colon \P(V_1) \xrightarrow{\ \sim\ } \P(\WW_3)
\end{equation*}
such that the points of $\P(V_1)$ to which the divisors $E_i$ are contracted by $\varphi$ go to the points~$P_i$ of $\P(\WW_3)$ defined by~\eqref{eq:points}.
Note that such an isomorphism is unique and $\fS_4$-equivariant.

Next, let $V_2$ be the three-dimensional vector space such that $\cE \cong V_2 \otimes \cO_S(-\ell)$.
Note that~\mbox{$V_2 \cong H^0(S, \cE(\ell))$} has a natural structure of an $\fS_4$-representation,
and the isomorphism $\cE \cong V_2 \otimes \cO_S(-\ell)$ is $\fS_4$-equivariant.
Under this identification the first map in~\eqref{eq:ce} becomes an $\fS_4$-equivariant embedding of sheaves
\begin{equation}
\label{eq:morphism}
V_2 \otimes \cO_S(-\ell) \xrightarrow{\ \xi\ } \cU_3,
\end{equation}
which is an isomorphism away from the union of $E_i$.
Its dual map extends to an exact $\fS_4$-equivariant sequence
\begin{equation}\label{eq:sequence}
0 \to \cU_3^\vee \xrightarrow{\ \xi^\vee\ } V_2^\vee \otimes \cO_S(\ell) \to \bigoplus_{i=0}^3 \cO_{E_i} \to 0.
\end{equation}
The second map defines four linear functions on $V_2^\vee$, i.e., four points on $\P(V_2)$.
We can choose an isomorphism
\begin{equation*}
\alpha_2 \colon \P(V_2) \xrightarrow{\ \sim\ } \P(\WW_3)
\end{equation*}
such that these points go to the points~$P_i$ of $\P(\WW_3)$ defined by~\eqref{eq:points}.
Again, such an isomorphism is unique and $\fS_4$-equivariant.

Now we put all the above constructions together.
The morphism $\xi$ defined by~\eqref{eq:morphism} induces a birational map
\begin{equation*}
\xymatrix@1{S \times \P(V_2) \cong \P_S(V_2 \otimes \cO_S(-\ell)) \ar@{-->}[r]^-\xi & \P_S(\cU_3) = \cY_{5,1}}.
\end{equation*}
We define $\theta_1$ as the composition
\begin{equation*}
\xymatrix@1@C=3em{\cY_{5,1} \ar@{-->}[r]^-{\xi^{-1}} & S \times \P(V_2) \ar[r]^-{\ \varphi \times  \operatorname{\mathrm{id}}\ } & \P(V_1) \times \P(V_2) \ar[r]^-{\alpha_1 \times\alpha_2} & \P(\WW_3) \times \P(\WW_3) \ar@{-->}[r]^-{\beta^{-1}} & \cY_{4,2},}
\end{equation*}
where the last map is the inverse of the blow up~\eqref{eq:y42}.
Clearly, $\theta_1$ is birational and $\fS_4$-equivariant, since all the maps used in its definition are.
Finally, its composition with~$p_1$ equals $\varphi \circ p$ by construction, hence the outer square in~\eqref{diagram:varphi} commutes.

Next, let us show an equality of the maps
\begin{equation}\label{eq:pi-theta}
\pi_{4,2} \circ \theta_1 = \pi_{5,1}
\end{equation}
from $\cY_{5,1}$ to $\P(\WW_5)$.
For this, consider the diagram
\begin{equation*}
\xymatrix@C=7em@R=1.9ex{
\WW_5^\vee \otimes \cO_S \ar[r] \ar@{=}[d] &
\WW_3^\vee \otimes \WW_3^\vee \otimes \cO_S \ar[r]^-{(\bP_0,\bP_1,\bP_2,\bP_3)} \ar[d]^{(\alpha_2^\vee \otimes \alpha_1^\vee)}_{\cong} &
\displaystyle\bigoplus_{i=0}^3 \hbox to 0 pt{$\cO_S$\hss} \ar@{=}[d]
\\
\WW_5^\vee \otimes \cO_S \ar[r]^{H^0(S,\xi^\vee)} \ar[d] &
V_2^\vee \otimes V_1^\vee \otimes \cO_S \ar[r] \ar[d] &
\displaystyle\bigoplus_{i=0}^3 \hbox to 0 pt{$\cO_S$\hss} \ar[d]
\\
\cU_3^\vee \ar[r]^{\xi^\vee} &
V_2^\vee \otimes \cO_S(\ell) \ar[r] &
\displaystyle\bigoplus_{i=0}^3 \hbox to 0 pt{$\cO_{E_i}$\hss}
}
\end{equation*}
Here the bottom line is~\eqref{eq:sequence}, the middle line is obtained from it by passing to global sections and tensoring with $\cO_S$,
and the maps between these lines are induced by evaluation of sections (hence the lower squares commute).
The top line is obtained by identification~\eqref{eq:w-w3w3}, the upper-right square commutes by definition of $\alpha_1$ and $\alpha_2$.
Therefore, there is a unique identification of the spaces $\WW_5^\vee$ in this diagram
(note that the one in the top line is defined by~\eqref{eq:w-w3w3},
while the other is defined by~\eqref{eq:quadrics-through-S}) such that the upper-left square commutes.
From now on we use implicitly the induced identification of the spaces~$\WW_5$.

As a result of this commutativity two morphisms $\WW_5^\vee \otimes \cO_S \to V_2^\vee \otimes \cO_S(\ell)$ in the diagram  coincide.
One of them induces the rational map
\begin{equation*}
\xymatrix@1{S \times \P(V_2) \ar@{-->}[r]^-\xi & \cY_{5,1} \ar[r]^-{\pi_{5,1}} & \P(\WW_5)},
\end{equation*}
and the other induces the rational map
\begin{equation*}
\xymatrix@1@C=2.5em{S \times \P(V_2) \ar[r]^-{\varphi \times \operatorname{\mathrm{id}}} & \P(V_1) \times \P(V_2) \ar[r]^-{\alpha_1 \times \alpha_2} & \P(\WW_3) \times \P(\WW_3) \ar@{-->}[r]^-{\bar\pi_{4,2}} & \P(\WW_5)};
\end{equation*}
the map $\varphi$ appears here because all the global sections of $\cO_S(\ell)$ are pullbacks via $\varphi$.
So, we have an equality of rational maps
\begin{equation*}
\bar\pi_{4,2} \circ (\alpha_1 \times \alpha_2) \circ (\varphi \times \operatorname{\mathrm{id}}) = \pi_{5,1} \circ \xi
\end{equation*}
from $S \times \P(V_2)$ to $\P(\WW_5)$.
Composing it with the map $\xi^{-1}$ on the right and using~\eqref{diagram:y42} and the definition of $\theta_1$, we deduce the required equality~\eqref{eq:pi-theta}.

From~\eqref{eq:pi-theta} we further deduce an equality
\begin{equation*}
\pi \circ (\rho_{4,2} \circ \theta_1) = \pi_{4,2} \circ \theta_1 = \pi_{5,1}.
\end{equation*}
Therefore, the composition $\rho_{4,2} \circ \theta_1$ provides one of the two possible factorizations $\rho_{5,1}$ of the morphism $\pi_{5,1}$.
This shows that for one of the two choices of $\rho_{5,1}$, the inner triangle in~\eqref{diagram:varphi} is commutative.
\end{proof}

It is worth noting that if we want to replace the projection $p_1$ in the diagram~\eqref{diagram:varphi} by another projection $p_2$ and preserve its commutativity,
we will have to replace the subgroup $\fS_5$ containing $\fS_4$ by the unique other such subgroup
(more precisely, we will have to replace the embedding $\fS_5 \hookrightarrow \fS_6$ with the one obtained from it by a conjugation with the factor~$\fS_2$
in $\fS_{4,2}$).

\begin{remark}\label{remark:rho-5-1-Weil-final}
Recall the notation of Remarks~\ref{remark:rho-4-2-Weil-divisor} and~\ref{remark:rho-5-1-Weil-divisor},
and assume that we are in the situation of Proposition~\ref{proposition:compatibility}:
the resolution $\rho_{4,2}$ is defined by~\eqref{eq:choice-rho-pp}
and the resolution~$\rho_{5,1}$ is such that the diagram~\eqref{diagram:varphi} commutes.
Then we have
\begin{align*}
\cH^{5,1}_+ &=
\cH^{5,1}_{123} +
\cH^{5,1}_{156} +
\cH^{5,1}_{246} +
\cH^{5,1}_{345} +
\cH^{5,1}_{124} +
\cH^{5,1}_{136} +
\cH^{5,1}_{235} +
\cH^{5,1}_{145} +
\cH^{5,1}_{256} +
\cH^{5,1}_{346},\\
\cH^{5,1}_- &=
\cH^{5,1}_{456} +
\cH^{5,1}_{234} +
\cH^{5,1}_{135} +
\cH^{5,1}_{126} +
\cH^{5,1}_{356} +
\cH^{5,1}_{245} +
\cH^{5,1}_{146} +
\cH^{5,1}_{236} +
\cH^{5,1}_{134} +
\cH^{5,1}_{125}.
\end{align*}
Consequently, $\rho_{5,1}$ is the blow up of the Weil divisor
\begin{equation*}
\cH_{456} +
\cH_{234} +
\cH_{135} +
\cH_{126} +
\cH_{356} +
\cH_{245} +
\cH_{146} +
\cH_{236} +
\cH_{134} +
\cH_{125}
\end{equation*}
on~$\cY$.
\end{remark}

\subsection{Proof of Theorem~\ref{theorem:coble}}
\label{subsection:proof-first}

In Proposition~\ref{proposition:first-Q-factorialization} we constructed the morphism $\rho_{4,2}$ for some non-standard subgroup $\fS_{4,2} \subset \fS_6$,
and checked that it is $\fS_{4,2}$-equivariant for the twisted action and small.
To construct $\rho_{4,2}$ for any other non-standard embedding, we may use a conjugation by an appropriate element of~$\fS_6$.
This proves assertion~(i).

Similarly to the above, in Proposition~\ref{proposition:Stein-factorization} we constructed a morphism $\rho_{5,1}$
for some non-standard embedding $\fS_5 \hookrightarrow \fS_6$
(and the same trick as above then gives $\rho_{5,1}$ for any other non-standard $\fS_5 \subset \fS_6$)
and checked that it is small.
Moreover, the compatibility isomorphism $\theta_1$ was constructed in Proposition~\ref{proposition:compatibility};
by the way it proves assertion~(iv).

Furthermore, we checked that the morphism $\rho_{5,1}$ is $\fS_5$-equivariant
with respect either to the natural or to the twisted action of $\fS_5$ on $\cY$.
To show that the action is twisted, we use Proposition~\ref{proposition:compatibility}.
Choose a subgroup $\fS_4 \subset \fS_5$, a transposition $g \in \fS_4$, and consider the commutative diagram~\eqref{diagram:varphi}.
Since $\theta_1$ is $\fS_4$-equivariant and~\mbox{$g \circ \rho_{4,2} \circ g^{-1} = \sigma \circ \rho_{4,2}$}
(as~$\rho_{4,2}$ is equivariant with respect to the twisted action),
we have~
\begin{equation*}
g \circ \rho_{5,1} \circ g^{-1} =
g \circ \rho_{4,2} \circ \theta_1 \circ g^{-1} =
g \circ \rho_{4,2} \circ g^{-1} \circ \theta_1 =
\sigma \circ \rho_{4,2} \circ \theta_1 =
\sigma \circ \rho_{5,1},
\end{equation*}
hence $\rho_{5,1}$ is equivariant with respect to the twisted action as well.
This completes the proof of assertion~(ii).

Finally, recall that we checked in Propositions~\ref{proposition:first-Q-factorialization} and~\ref{proposition:Stein-factorization}
that $\rho_{4,2}$ and~$\rho_{5,1}$ are isomorphisms over the complement of the Cremona--Richmond
configuration~\mbox{$\CR = \Sing(X_{1/4}) \subset \P^4$}.
This gives the proof of assertion~(iii) and completes the proof of Theorem~\ref{theorem:coble}.
\qed


\section{Conic bundle structures on $\fS_6$-invariant quartics}
\label{section:conic-bundles}

Recall the pencil $\{X_t\}$ of $\fS_6$-invariant quartics defined by the equation~\eqref{eq:beauville} inside the hyperplane $\P^4 \subset \P^5$ given by~\eqref{eq:hyperplane}.
In this section we discuss the conic bundle structures on the quartics~$X_t$ induced by the resolutions of the Coble fourfold.

\subsection{$\fS_6$-invariant quartics revisited}
\label{subsection:xt}

We start by collecting
some facts about automorphism groups of $X_t$, their singularities and class groups.

Let $\CR$ be the Cremona--Richmond configuration of 15 lines with 15 intersection points, see Appendix~\ref{section:CR}.
The intersection points of the lines of $\CR$ form the orbit
\begin{alignat*}{3}
\Upsilon_{15} &= \{ g \cdot (2:2:-1:-1:-1:-1)			&&\mid g \in \fS_6\}.\\
\intertext{Besides this, we consider also the orbits}
\Sigma_{6}  &= \{ g \cdot (5:-1:-1:-1:-1:-1) 			&& \mid g \in \fS_6\},\\
\Sigma_{10} &= \{ g \cdot (1:1:1:-1:-1:-1)   			&& \mid g \in \fS_6\},\\
\Sigma_{15} &= \{ g \cdot (1:-1:0:0:0:0)     			&& \mid g \in \fS_6\},\\
\Sigma_{30} &= \{ g \cdot (1:1:\omega:\omega:\omega^2:\omega^2)	&& \mid g \in \fS_6\},
\end{alignat*}
where $\omega$ is a primitive cubic root of unity and the lower index on the left hand side stands for cardinality of the orbit.
We note that
\begin{equation*}
\Upsilon_{15}\subset \CR,\qquad \Sigma_{30} \subset \CR, 
\qquad
(\Sigma_6 \cup \Sigma_{10} \cup \Sigma_{15}) \cap \CR = \varnothing.
\end{equation*}

\begin{remark}\label{remark:lines-vs-quadric}
The
quartic $X_\infty$ defined by equation~\eqref{eq:beauville} with $t = \infty$ is the quadric $Q_\infty$ given by the equation
\begin{equation}\label{eq:q-infty}
x_1^2 + x_2^2 + x_3^2 + x_4^2 + x_5^2 + x_6^2=0
\end{equation}
taken with multiplicity 2.
Note that
\begin{equation*}
Q_\infty \cap \Upsilon_{15} = \varnothing,
\qquad
Q_\infty \cap \CR = \Sigma_{30},
\end{equation*}
and the intersection is transversal.
\end{remark}

The singularities of the quartics $X_t$ have been described by van der Geer in \cite{Geer} in terms of these orbits.
Recall the discriminant set $\bDD$ defined by~\eqref{def:dd}.

\begin{theorem}[{\cite[Theorem~4.1]{Geer}}]
\label{theorem:Geer}
One has
\begin{equation*}\renewcommand{\arraystretch}{1.5}
\begin{array}{|c||c|c|c|c|c|}
\hline
t & t \not\in \bDD \cup \{\infty\} & t = \frac14 & t = \frac12 & t = \frac16 & t = \frac7{10}
\\
\hline
\Sing\big(X_t\big) & \Sigma_{30} & \CR & \Sigma_{30} \cup \Sigma_{15} & \Sigma_{30} \cup \Sigma_{10} & \Sigma_{30} \cup \Sigma_{6}
\\
\hline
\end{array}
\end{equation*}
In particular, $X_t$ is normal if $t \ne \infty$.

Moreover, all singular points of the quartics $X_t$ are nodes provided that $t\ne 1/4,\infty$.
\end{theorem}

One can describe automorphism groups of the quartics~$X_t$.

\begin{lemma}
\label{lemma:quartics-Aut}
The following assertions hold.
\begin{itemize}
\item[(i)] One has $\Aut(X_{1/2})\cong\PSp$, where $\mathbf{F}_3$ is the field of three elements.

\item[(ii)] One has $\Aut(X_t)\cong\fS_6$ provided that
$t\not\in\left\{1/2,\infty\right\}$.

\item[(iii)] If $X$ is a normal quartic hypersurface with a faithful action
of the group $\fS_6$, then~$X$ is isomorphic to one
of the quartics $X_t$.
\end{itemize}
\end{lemma}
\begin{proof}
Assertion~(i) is well known, see e.g. \cite{Coble-PSp}.

Take any $t \ne \infty$.
Since the quartic $X_t$ is normal by Theorem~\ref{theorem:Geer}, its hyperplane section is the anticanonical class, hence
the group~$\Aut(X_t)$ is naturally embedded into~$\PGL_5(\Bbbk)$.
Moreover, one has $\fS_6\subset\Aut(X_t)$ by the definition of $X_t$.
It follows from the classification of finite subgroups of~$\PGL_5(\Bbbk)$ that either $\Aut(X_t)\cong\fS_6$, or $\Aut(X_t)\cong\PSp$, see~\cite[\S8.5]{Feit}.
But the group $\PSp$ has a unique invariant quartic hypersurface in~$\P^4$, which is the Burkhardt quartic $X_{1/2}$.
This proves assertion~(ii).

Finally, assume that $X$ is a normal quartic hypersurface invariant under some faithful action of the group $\fS_6$ on $\P^4$.
Using the classification of projective representations of the group $\fS_6$
we deduce that this action comes from an irreducible five-dimensional representation of $\fS_6$;
in fact, it is enough to look at the classification of projective representations of the smaller group $\fA_6$, which
can be found for instance in~\cite[p.5]{Atlas}.
The latter $\fS_6$-representation is unique up to an outer automorphism and a sign twist
(cf.~Lemma~\ref{lemma:restriction-s6-s5}).
This implies assertion~(iii).
\end{proof}

\begin{corollary}\label{corollary:aut-y}
We have $\Aut(\cY) \cong \fS_6 \times \mumu_2$.
\end{corollary}
\begin{proof}
The group on the right hand side acts on $\cY$ by~\eqref{eq:action-standard} and~\eqref{eq:sigma},
and the action is clearly faithful.
It remains to show that any automorphism of $\cY$ belongs to this group.
For this we note that the morphism $\pi \colon \cY \to \P^4$ is defined by the ample generator of~$\Pic(\cY)$.
Indeed, $\rkPic(\cY)=1$ by Lefschetz hyperplane section theorem (see~\mbox{\cite[Theorem~4.2.2]{Dolgachev-WPS}}), because~$\cY$ is a hypersurface in the weighted projective space~\mbox{$\P(2,1^5)$}.
The pullback of the hyperplane in $\P^4$ via $\pi$ is not divisible in $\Pic(\cY)$ by degree reasons, and thus
generates $\Pic(\cY)$.
Hence $\pi$ is equivariant with respect to any automorphism of~$\cY$.
This induces a homomorphism~\mbox{$\Aut(\cY) \to \PGL_5(\Bbbk)$} whose kernel is generated by the Galois involution $\sigma$.
The image of the homomorphism is the
subgroup of~$\PGL_5(\Bbbk)$ that fixes the branch divisor~$X_{1/4}$ of~$\pi$.
Moreover, the latter subgroup acts faithfully on~$X_{1/4}$, hence is contained in~\mbox{$\Aut(X_{1/4}) \cong \fS_6$}.
\end{proof}

For further reference we state here a description of the class groups of $X_t$.

\begin{lemma}
\label{lemma:cl-xt}
The following table lists the ranks of the class groups of the quartics $X_t$:
\begin{equation*}\renewcommand{\arraystretch}{1.5}
\begin{array}{|c||c|c|c|c|c|}
\hline
t & t \not\in \bDD \cup \{\infty\} & t = \frac14 & t = \frac12 & t = \frac16 & t = \frac7{10}
\\
\hline
\rkCl(X_t) & 6 & 1 & 16 & 11 & 7
\\
\hline
\end{array}
\end{equation*}
\end{lemma}
\begin{proof}
First,
assume $t \not\in \bDD \cup \{\infty\}$.
Let $\widetilde{X}_t$ be the blow up of $X_t$ at its singular points.
Then $\widetilde{X}_t$ is smooth by Theorem~\ref{theorem:Geer}.
Now the assertion follows from~\cite[Theorem~2]{Cynk} and~\cite[Lemma~2]{Beauville}.

The cases $t = 1/2$, $t = 1/6$, and $t = 7/10$,  are discussed in~\cite[Theorem~1.1(iii)]{Kaloghiros} and~\cite[\S\S5--6]{CheltsovShramov}.

Finally, consider the case $t = 1/4$.
As it was already mentioned, the Igusa quartic~$X_{1/4}$ is projectively dual to the Segre cubic threefold $Z\subset\mathbb{P}^4$.
In fact, projective duality gives an $\fS_6$-equivariant birational map $Z\dasharrow X_{1/4}$ that blows up $10$ ordinary double points of $Z$
and blows down the proper transforms of $15$ planes on $Z$, see e.g.\ the proof of~\mbox{\cite[Lemma~3.10]{Prokhorov-FieldsOfInv}}.
In particular, one has
\begin{equation*}
\rkCl\left(X_{1/4}\right) = \rkCl(Z) + 10 - 15,
\end{equation*}
and since the class group of the Segre cubic $Z$ has rank $6$ (see e.g.~\cite[Theorem~7.1]{Prokhorov}),
we obtain~\mbox{$\rkCl\left(X_{1/4}\right) = 1$}.
\end{proof}

In~Theorem~\ref{theorem:Cl} we will describe the action of the group $\fS_6$ on $\Cl(X_t) \otimes \Q$.

\subsection{Wiman--Edge pencil}
\label{subsection:WE}

Consider the projective plane $\P^2$ with homogeneous coordinates $w_1$, $w_2$, and $w_3$ and the following two polynomials of degree six
\begin{equation}\label{eq:w-pi}
\begin{aligned}
\oP_0(w_1,w_2,w_3) &= (w_2^2 - w_3^2)(w_3^2 - w_1^2)(w_1^2 - w_2^2),\\
\oP_\infty(w_1,w_2,w_3)   &= w_1^6 + w_2^6 + w_3^6 + (w_1^2 + w_2^2 + w_3^2)(w_1^4 + w_2^4 + w_3^4) - 12w_1^2w_2^2w_3^2.
\end{aligned}
\end{equation}
It is easy to see that the sextic curves on $\P^2$ defined by these polynomials are singular at the following four points
\begin{equation}\label{eq:points-other}
(1:1:1),\quad
(1:-1:-1),\quad
(-1:1:-1),\quad
(-1:-1:1),
\end{equation}
hence they induce a pair of global sections
\begin{equation*}
\Phi_0, \Phi_\infty \in H^0(S,\omega_S^{-2})
\end{equation*}
of the double anticanonical line bundle on the blow up $S$ of $\P^2$ at the points~\eqref{eq:points-other}, i.e., on the quintic del Pezzo surface.
By~\cite{Edge} the section $\Phi_\infty$ is invariant with respect to the action of~\mbox{$\Aut(S)\cong\fS_5$},
while the $\Phi_0$ is acted on by $\fS_5$ via the sign character.
Therefore, there is an $\fS_5$-invariant pencil of $\fA_5$-invariant curves $\Delta_s \subset S$ given by the equation
\begin{equation}\label{eq:delta-s}
\Phi_0 + s\Phi_\infty = 0,\quad s\in\Bbbk\cup\{\infty\}.
\end{equation}
As we already mentioned, the curves $\Delta_s$ are double anticanonical divisors on $S$.
We refer to the pencil \eqref{eq:delta-s} as the {\sf Wiman--Edge pencil}.
It was studied in various contexts in
\cite{Wiman}, \cite{Edge}, \cite{InoueKato},
\cite[\S6.2]{CheltsovShramovBook}, \cite{DFL}, \cite{Zamora}, etc.

\begin{theorem}[{see \cite{Edge} or \cite[Theorem~6.2.9]{CheltsovShramovBook}}]
\label{theorem:WE-singular-curves}
The Wiman--Edge pencil contains exactly five singular curves: $\Delta_0$, $\Delta_{\pm 1/\sqrt{125}}$, and $\Delta_{\pm 1/\sqrt{-3}}$.
They can be described as follows:
\begin{itemize}
\item
$\Delta_0$ is the union of $10$\/ lines on $S$; it has $15$\/ singular points.

\item
$\Delta_{\pm 1/\sqrt{-3}}$
are unions of $5$\/ smooth conics; each of these curves has $10$\/ singular points.

\item
$\Delta_{\pm 1/\sqrt{125}}$
are irreducible rational curves; each of these curves has $6$\/ singular points.
\end{itemize}
Every singular point of any of these curves is a node.
The group $\fA_5$ acts transitively on the set of singular points and on the set of
irreducible components of each of these curves.
\end{theorem}

\begin{remark}
The curves $\Delta_0$ and~$\Delta_\infty$ in the Wiman--Edge pencil are not just $\fA_5$-invariant, but also $\fS_5$-invariant.
The first of them, as we already mentioned, is the union of~$10$ lines.
The other one is a smooth curve of genus $6$ known as the Wiman's sextic curve, see~\mbox{\cite{Wiman,Edge}};
it should not be confused with a smooth plane sextic curve studied by Wiman in~\cite{Wiman-extra}.
By construction, $\Delta_\infty$ admits a faithful action of the group $\mathfrak{S}_5$, and one can show that its full automorphism group is also $\mathfrak{S}_5$.
\end{remark}

\subsection{Preimages of $\fS_6$-invariant quartics in the Coble fourfold}
\label{subsection:cxtau}

Recall that the Coble fourfold $\cY$ is defined as a
complete intersection in the weighted projective space~$\P(2,1^6)$ of the hyperplane~\eqref{eq:hyperplane} with the hypersurface~\eqref{eq:coble}.
It comes with a double covering $\pi \colon \cY \to \P^4$ over the projective space in which the pencil $\{X_t\}$ of $\fS_6$-invariant quartics sits,
and with the Galois involution $\sigma \colon \cY \to \cY$ of the double covering.

As in~\S\ref{section:intro}, we define a pencil of hypersurfaces $\cX_\tau \subset \cY$ by~\eqref{eq:cxtau}.
By definition each of the varieties~$\cX_\tau$ is $\fS_6$-invariant with respect to the natural $\fS_6$-action.
Moreover, $\cX_0$ and~$\cX_\infty$ are invariant under the whole group $\Aut(\cY) = \fS_6 \times \mumu_2$.

\begin{lemma}
\label{lemma:cxtau-xt}
For every $\tau \neq \infty$ we have
\begin{equation*}
\pi^{-1}\left(X_{\frac{\tau^2 + 1}{4}}\right) = \cX_\tau \cup \cX_{-\tau},
\end{equation*}
and the involution $\sigma$ induces an $\fS_6$-equivariant isomorphism $\sigma \colon \cX_\tau \to \cX_{-\tau}$
for the natural action of~$\fS_6$.
The map $\pi \colon \cX_\tau \to X_{(\tau^2+1)/4}$ is an  isomorphism for all $\tau \ne \infty$, and
the map $\pi \colon \cX_\infty \to X_\infty$ factors through the double covering over the quadric $Q_\infty = (X_\infty)_{\mathrm{red}}$
defined by~\eqref{eq:q-infty} that is branched over~$X_{1/4} \cap Q_\infty$.
The map $\pi$ is $\fS_6 \times \mumu_2$-equivariant for~$\tau = 0,\infty$ and $\fS_6$-equivariant otherwise.
\end{lemma}
\begin{proof}
The hypersurface $\pi^{-1}(X_{(\tau^2 + 1)/4}) \subset \cY$ is defined by the equation
\begin{equation*}
\Big(x_1^4 + x_2^4 + x_3^4 + x_4^4 + x_5^4 + x_6^4\Big) -
\frac{\tau^2 + 1}4\Big(x_1^2 + x_2^2 + x_3^2 + x_4^2 + x_5^2 + x_6^2\Big)^2
= 0,
\end{equation*}
which in view of equation~\eqref{eq:coble} of $\cY$ can be rewritten as
\begin{align*}
0 & =
x_0^2 - \frac{\tau^2}4\Big(x_1^2 + x_2^2 + x_3^2 + x_4^2 + x_5^2 + x_6^2\Big)^2 = \\
&=
\left(x_0 + \frac\tau2\big(x_1^2 + x_2^2 + x_3^2 + x_4^2 + x_5^2 + x_6^2\big)^2\right)
\left(x_0 - \frac\tau2\big(x_1^2 + x_2^2 + x_3^2 + x_4^2 + x_5^2 + x_6^2\big)^2\right).
\end{align*}
Hence $\pi^{-1}(X_{(\tau^2 + 1)/4})$ is the union of $\cX_\tau$ and $\cX_{-\tau}$.
The Galois involution $\sigma$ acts by~\mbox{$x_0 \mapsto -x_0$}, hence defines an isomorphism between $\cX_\tau$ and $\cX_{-\tau}$.
To check that the map~\mbox{$\pi \colon \cX_\tau \to X_{(\tau^2+1)/4}$} is an isomorphism, just use~\eqref{eq:cxtau} to express $x_0$ in terms of other~$x_i$;
plugging it into the equation of the Coble fourfold~$\cY$, we deduce the equation of the quartic~$X_t$.
For~\mbox{$\tau = \infty$} this of course does not work, but the equations of $\cX_\infty$ just give
\begin{equation*}
x_1^2 + x_2^2 + x_3^2 + x_4^2 + x_5^2 + x_6^2 = x_0^2 - (x_1^4 + x_2^4 + x_3^4 + x_4^4 + x_5^4 + x_6^4) = 0
\end{equation*}
which defines a double covering of $Q_\infty$ whose branch locus is~\mbox{$X_0 \cap Q_\infty = X_{1/4} \cap Q_\infty$}.

The equivariance of the maps $\sigma$ and $\pi$ is obvious.
\end{proof}

\begin{remark}
\label{remark:double-quadric-Sing}
The singular locus of $\cX_{\infty}$ consists of the unique $\fS_6$-orbit of length $30$ that is projected by $\pi$
to the $\fS_6$-orbit $\Sigma_{30}$, see e.g.~\cite[\S6]{PrzyjalkowskiShramov}.
\end{remark}

Now we say a couple of words about the Weil divisor class groups of the threefolds~$\cX_\tau$.
Consider the set
\begin{equation}
\label{def:dd-tau}
\bDDtau := \left\{ 0, \pm1, \pm\frac{1}{\sqrt{-3}}, \pm\frac3{\sqrt{5}} \right\},
\end{equation}
that is, the preimage of the discriminant set $\bDD$ defined in~\eqref{def:dd} under the map~\eqref{eq:t-tau}.

\begin{lemma}
\label{lemma:Cl-Z6-upstairs}
The following table lists the ranks of the class groups of the threefolds $\cX_\tau$:
\begin{equation*}\renewcommand{\arraystretch}{1.5}
\begin{array}{|c||c|c|c|c|c|}
\hline
\tau & \tau \not\in \bDDtau & \tau = 0 & \tau = \pm1 & \tau = \pm\frac1{\sqrt{-3}} & \tau = \pm\frac3{\sqrt{5}}
\\
\hline
\rkCl(\cX_\tau) & 6 & 1 & 16 & 11 & 7
\\
\hline
\end{array}
\end{equation*}
\end{lemma}
\begin{proof}
If we assume that $\tau\neq\infty$, then the assertion follows from Lemma~\ref{lemma:cl-xt}
in view of Lemma~\ref{lemma:cxtau-xt}.
For $\tau=\infty$ we argue similarly to the proof of Lemma~\ref{lemma:cl-xt} (cf. the proof of~\mbox{\cite[Proposition~6.3]{PrzyjalkowskiShramov}}).
Let~$\widetilde{\cX}_\infty$ be the blow up of $\cX_\infty$ along its singular locus, i.e.,
the preimage of the $\fS_6$-orbit $\Sigma_{30}$, see Remark~\ref{remark:double-quadric-Sing}.
Then $\widetilde{\cX}_\infty$ is smooth, and one proceeds as in~\cite[Theorem~2]{Cynk}, using the computation of~\cite[Lemma~2]{Beauville}.
\end{proof}

\subsection{Pencil of Verra threefolds}
\label{subsection:cx42-tau}

We consider the pullbacks $\cX_\tau^{5,1}$ and $\cX_\tau^{4,2}$ of the threefolds~$\cX_\tau$ to the resolutions
$\cY_{5,1}$ and $\cY_{4,2}$ of singularities of the Coble fourfold,
so that~\mbox{$\cX^{5,1}_\tau \subset \cY_{5,1}$} and~\mbox{$\cX^{4,2}_\tau \subset \cY_{4,2}$} are defined by~\eqref{eq:beauville-pullback-51-42}.
In the next section we will study the first of them,
but now let us consider the second one.
We assume that the map $\rho_{4,2}$ is defined by~\eqref{eq:choice-rho-pp}.

To simplify the situation, we consider the images of the threefolds $\cX_\tau^{4,2}$ with respect
to the contraction~\mbox{$\beta \colon \cY_{4,2} \to \P^2 \times \P^2 = \P(\WW_3) \times \P(\WW_3)$},    see~\S\ref{subsection:y42}.
Define
\begin{equation*}
\bar\cX_\tau^{4,2} = \beta(\cX_\tau^{4,2}) \subset \P^2 \times \P^2.
\end{equation*}

As in~\S\ref{subsection:y42} we use $(u_1:u_2:u_3)$ and $(v_1:v_2:v_3)$ for coordinates on the factors of~$\P^2 \times \P^2$,
and let $\bP_i = (P_i,P_i)$ with $P_i$ defined by~\eqref{eq:points}.

Below we consider divisors of bidegree~$(2,2)$ in $\P^2 \times \P^2$ (and call them {\sf Verra threefolds}) as conic bundles over the first factor.
We write their equations as symmetric $3\times 3$-matrices with coefficients being quadratic polynomials in $u_1$, $u_2$, $u_3$.
So, if $q(u)=(q_{ij}(u))$ is such a matrix, the corresponding equation is $q(u)(v) := \sum q_{ij}(u)v_iv_j = 0$.

\begin{proposition}\label{proposition:verra}
The subvariety $\bar\cX_\tau^{4,2} \subset \P^2 \times \P^2$ is a Verra threefold given by the equation
\begin{equation}\label{eq:pencil-q}
q_0(u)(v) + \tau q_\infty(u)(v) = 0,
\end{equation}
where
\begin{align}
q_0(u) &= \frac12 \left(
\begin{xsmallmatrix}
0 & u_3(u_2-u_1) & u_2(u_1-u_3) \\
u_3(u_2-u_1) & 0 & u_1(u_3-u_2) \\
u_2(u_1-u_3) & u_1(u_3-u_2) & 0
\end{xsmallmatrix} \right),
\qquad\text{and}
\label{eq:q0}
\\[1ex]
q_\infty(u) &= \frac16 \left(
\begin{xsmallmatrix}
4(u_2^2-u_2u_3+u_3^2) &  u_3(u_1+u_2)-2u_1u_2-2u_3^2 &  u_2(u_1+u_3)-2u_1u_3-2u_2^2 \\
u_3(u_1+u_2)-2u_1u_2-2u_3^2 &  4(u_1^2-u_1u_3+u_3^2) &  u_1(u_2+u_3)-2u_2u_3-2u_1^2 \\
u_2(u_1+u_3)-2u_1u_3-2u_2^2 &  u_1(u_2+u_3)-2u_2u_3-2u_1^2 &  4(u_1^2-u_1u_2+u_2^2)
\end{xsmallmatrix} \right).
\label{eq:q8}
\end{align}
\end{proposition}
\begin{proof}
By~\eqref{eq:cxtau}, the variety $\bar\cX_0^{4,2}$ is given by the equation $x_0 = 0$.
Writing the formula for~$x_0$ from~\eqref{eq:choice-rho-pp} in the matrix form, we get~\eqref{eq:q0}.
Similarly, $\bar\cX_{\infty}^{4,2}$ is given by the equation
\begin{equation*}
\frac12(x_1^2 + x_2^2 + x_3^2 + x_4^2 + x_5^2 + x_6^2) = 0.
\end{equation*}
Substituting expressions for $x_i$ from~\eqref{eq:choice-rho-pp} and rewriting everything in the matrix form, we get~\eqref{eq:q8}.
Therefore, equation~\eqref{eq:pencil-q} is the same as~\eqref{eq:cxtau}.
\end{proof}

\begin{remark}
Of course, one can cancel the common factor $1/2$ in~\eqref{eq:q0} and~\eqref{eq:q8}.
However, we prefer to keep it so that~\mbox{$q_0(u)(v)$} and~\mbox{$q_\infty(u)(v)$}
are the same as the two summands in~\eqref{eq:cxtau}.
\end{remark}

Since the maps $\beta \colon \cX_\tau^{4,2} \to \bar\cX_\tau^{4,2}$ and
$\pi_{4,2} = \pi \circ \rho_{4,2} \colon \cX_\tau^{4,2} \to X_{(\tau^2+1)/4}$
are birational for all~$\tau \ne \infty$,
the projection $p_1 \colon \bar\cX_\tau^{4,2} \to \P^2$ provides every (reduced) $\fS_6$-invariant quartic with a birational structure of a conic bundle.
Similarly, the map $p_1 \colon \bar\cX_\infty^{4,2} \to \P^2$ provides a
birational structure of a conic bundle on the threefold $\cX_{\infty}$.
The explicit formulas of Proposition~\ref{proposition:verra} allow to compute their discriminant loci.

\begin{lemma}\label{lemma:WE-discriminant}
The discriminant curve of the conic bundle $p_1 \colon \bar\cX^{4,2}_\tau \to \P^2$ is the curve~\mbox{$\overline\Delta_\tau\subset \P^2$} defined by the equation
\begin{equation}\label{eq:delta-tau}
(5\tau^2 + 3)\oP_0 + (\tau^3 - \tau)\oP_\infty = 0,
\end{equation}
where $\oP_0$ and $\oP_\infty$ are the sextic polynomials~\eqref{eq:w-pi}, and the coordinates
$(w_1 : w_2 : w_3)$ are related to $(u_1 : u_2 : u_3)$ by the formula
\begin{equation*}
u_1 = w_2 + w_3,
\qquad
u_2 = w_1 + w_3,
\qquad
u_3 = w_1 + w_2.
\end{equation*}
\end{lemma}
\begin{proof}
A straightforward computation shows that
\begin{equation*}
12\det\big(q_0(u) + \tau q_\infty(u)\big)=(5\tau^2 + 3)\oP_0 + (\tau^3 - \tau)\oP_\infty.\qedhere
\end{equation*}
\end{proof}

The drawback of this conic bundle model is the lack of flatness.
Indeed, it is easy to see that over each of the points $P_i$ (see~\eqref{eq:points}) the matrix $q_0(u)$ is identically zero,
so the fiber of $\bar\cX^{4,2}_0$ over $P_i$ is the whole $\P^2$.
In the next subsection we check that using the resolution~$\cY_{5,1}$ of the Coble fourfold, we obtain flat conic bundles.

\subsection{Pencil of conic bundles over the quintic del Pezzo surface}
\label{subsection:cx51-tau}

Recall that~$\cX^{5,1}_\tau \subset \cY_{5,1}$ is defined in~\eqref{eq:beauville-pullback-51-42} as the preimage
of the threefold $\cX_\tau \subset \cY$ under the resolution $\rho_{5,1} \colon \cY_{5,1} \to \cY$.
For its investigation it will be very convenient to use explicit formulas of~\S\ref{subsection:cx42-tau}.
So, to benefit from those we assume that we are in the situation of Proposition~\ref{proposition:compatibility},
i.e., a subgroup $\fS_4 \subset \fS_5$ and a non-standard embedding $\fS_5 \hookrightarrow \fS_6$
are chosen, the choice of $\rho_{4,2}$ is fixed as in~\eqref{eq:choice-rho-pp},
the map $\theta_1 \colon \cY_{5,1} \dashrightarrow \cY_{4,2}$ is a birational isomorphism for which the outer square
of diagram~\eqref{diagram:varphi} commutes, and $\rho_{5,1} = \rho_{4,2} \circ \theta_1$.

\begin{remark}
\label{remark:rho-equivariance}
As we already discussed, for $\tau \ne 0, \infty$ the subvariety $\cX_\tau$ is invariant with respect to the \emph{natural} action of $\fS_6$,
while the map $\rho_{5,1} \colon \cY_{5,1} \to \cY$ is equivariant with respect to the \emph{twisted} action of~$\fS_5 \subset \fS_6$.
As a result, the subvariety $\cX^{5,1}_\tau \subset \cY^{5,1}$ is only invariant under the action of the subgroup $\fA_6 \cap \fS_5 = \fA_5$, on which the two actions agree.
Similarly, the projection $\rho_{5,1} \colon \cX^{5,1}_\tau \to \cX_\tau$ is only $\fA_5$-equivariant.
On the other hand, for $\tau = 0$ or $\tau = \infty$, the subvariety $\cX_\tau^{5,1} \subset \cY$ is $\fS_5$-invariant and
the map $\rho_{5,1}$ is~$\fS_5$-equivariant.
\end{remark}

\begin{lemma}\label{lemma:x51}
The map $p \colon \cX^{5,1}_\tau \to S$ is a flat conic bundle with the discriminant curve~$\Delta_{s(\tau)} \subset S$ defined by~\eqref{eq:delta-s}, where
\begin{equation}\label{eq:s-tau}
s(\tau) = \frac{\tau^3 - \tau}{5\tau^2 + 3}.
\end{equation}
The map~$p$ is $\fA_5$-equivariant for $\tau \ne 0,\infty$ and $\fS_5$-equivariant for $\tau = 0,\infty$.
\end{lemma}
\begin{proof}
Equivariance of the maps $p \colon \cX^{5,1}_\tau \to S$ follows from invariance of $\cX^{5,1}_\tau$ discussed in Remark~\ref{remark:rho-equivariance}
and $\fS_5$-equivariance of the $\P^2$-bundle $p \colon \cY_{5,1} \to S$.
The restriction of~\eqref{diagram:varphi} gives a commutative diagram
\begin{equation}\label{diagram:varphi-x}
\vcenter{\xymatrix{
\cX^{5,1}_\tau \ar[dd]_{p} \ar@{-->}[rr]^-{\theta_1} \ar[dr]_{\rho_{5,1}} &&
\cX^{4,2}_\tau \ar[d]^{\beta} \ar[dl]^{\rho_{4,2}} \\
& \cX_\tau &
\bar\cX^{4,2}_\tau \ar[d]^{p_1} \\
S \ar[rr]^\varphi &&
\P^2
}}
\end{equation}
The divisor $\cX^{5,1}_\infty \subset \cY$ is the preimage of the quadric $Q_\infty \subset \P(\WW_5)$ with respect to
the map~\mbox{$\pi_{5,1} \colon \cY_{5,1} \to \P(\WW_5)$},
hence it is the zero locus of a section of the line bundle~\mbox{$\cO_{\P_S(\cU_3)}(2)$}.
Since~$\cX_\tau^{5,1}$ form a pencil, all of them are the zero loci of sections of the same line bundle,
hence correspond to symmetric morphisms~$\cU_3 \to \cU_3^\vee$ on~$S$ (in particular, $p \colon \cX^{5,1}_\tau \to S$ is a conic bundle).
Therefore, the discriminant curve of $\cX_\tau^{5,1}$ is the zero locus of a morphism
\begin{equation*}
\omega_S \cong \det(\cU_3) \to \det(\cU_3^\vee) \cong \omega_S^{-1},
\end{equation*}
i.e., a double anticanonical divisor.

On the other hand, the above diagram shows that the discriminant locus of $\cX_\tau^{5,1}$ contains
the proper transform of the discriminant curve $\overline\Delta_\tau$ of $\bar\cX_\tau^{4,2}$ whose equation is~\eqref{eq:delta-tau}.
If $\tau^3  - \tau \ne 0$ it is a sextic curve passing with multiplicity 2 through each of the points $P_i$, hence its proper transform to $S$ is
a curve on $S$ with equation
\begin{equation}\label{eq:53Phi0-ttPhiinfty}
(5\tau^2 + 3)\Phi_0 + (\tau^3 - \tau)\Phi_\infty = 0,
\end{equation}
i.e., the curve $\Delta_{s(\tau)}$.
In the case when $\tau^3 - \tau = 0$, the curve $\overline{\Delta}_\tau$ is the union of six lines on $\P^2$, and its proper transform on $S$ is the union of six lines on $S$.
But the conic bundle~\mbox{$p\colon \cX_\tau^{5,1}\to S$} is $\fA_5$-equivariant.
Thus its discriminant curve is $\fA_5$-invariant, and hence it should also contain the other four lines of~$S$.
We conclude this case
by noting that the sum of the ten lines $\Delta_0$ on $S$ is a double anticanonical divisor, and it is indeed given by the equation~\eqref{eq:53Phi0-ttPhiinfty}
with~\mbox{$\tau^3 - \tau = 0$}.

It remains to show that the conic bundle is flat.
For this we note that a non-flat point of a conic bundle
is a point of multiplicity at least 3 on its discriminant curve.
But by Theorem~\ref{theorem:WE-singular-curves} all singular points of these curves are nodes.
\end{proof}

Before going further, we discuss some properties of the map $s \colon \P^1 \to \P^1$ defined by~\eqref{eq:s-tau}.

\begin{lemma}
\label{lemma:s-map}
The map $s \colon \P^1 \to \P^1$ is a triple covering with simple ramification at four points $\tau = \pm\sqrt{-3}$ and $\tau = \pm 1/\sqrt{5}$.
\end{lemma}

\begin{proof}
A direct computation.
\end{proof}

In the next table we list some special values of $\tau$
together with the values of the functions~$s(\tau)$ and $t(\tau) = (\tau^2 + 1)/4$ at these points.
\begin{equation*}\renewcommand{\arraystretch}{1.5}
\begin{array}{|c||c|c||c|c||c|c||c|c|}
\hline
\tau &
\hfill\hspace{.4em}0\hspace{.4em}\hfill
& \pm1 &
\pm\frac{1}{\sqrt{-3}} & \cellcolor{lightgray}{\pm\sqrt{-3}} &
\mp\frac3{\sqrt{5}} & \cellcolor{lightgray}\pm\frac1{\sqrt{5}} &
\infty & \pm\sqrt{-\frac{3}{5}}\\[1ex]
\hline
s(\tau)
& \multicolumn{2}{c||}{0} &
\multicolumn{2}{c||}{\mp\frac{1}{\sqrt{-3}}}&
\multicolumn{2}{c||}{\mp\frac1{5\sqrt{5}}} &
\multicolumn{2}{c|}{\infty} \\[1ex]
\hline
t(\tau) &
\cellcolor{lightgray}{\frac14} & \cellcolor{lightgray}{\frac12} &
\cellcolor{lightgray}{\frac16} & -\frac12 &
\cellcolor{lightgray}{\frac7{10}} & \frac3{10} &
\cellcolor{lightgray}{\infty} & \frac{1}{10} \\[1ex]
\hline
\end{array}
\end{equation*}
\smallskip
The second row contains the values of the parameter $s$ that correspond to singular members of the Wiman--Edge pencil (see Theorem~\ref{theorem:WE-singular-curves}) and infinity.
The first row contains their preimages; gray cells mark ramification points of the map $s(\tau)$, see Lemma~\ref{lemma:s-map}.
The third row contains the values of the map $t(\tau)$ at these points; gray cells mark the points of the discriminant set $\bDD$ and infinity.

Since the degree of the map $s$ is $3$, the same singular curves in the Wiman--Edge pencil may appear as the discriminant loci
of the preimages $\cX^{5,1}_\tau$ of different quartics~$X_t$.
For instance, the Igusa and the Burkhardt quartics both correspond to the union~$\Delta_0$ of the ten lines on $S$.
Note also that the quartics $X_{1/6}$ and $X_{7/10}$ share their discriminant curves with non-special quartics $X_{-1/2}$ and $X_{3/10}$ respectively.
As we will see in Proposition~\ref{proposition:x51-smooth},
these two are characterized by the fact that the corresponding curves in the Wiman--Edge pencil are singular,
while the total spaces of the threefolds~$\cX_{\tau}^{5,1}$ are smooth.
In~\S\ref{section:applications} we will see that this subtle difference has a drastic effect on rationality properties.

To proceed we will need the following general result.
Its proof can be found in~\mbox{\cite[Proposition~1.2]{BeauvillePrym}} or \cite[Proposition~1.8]{Sarkisov}, except for the fact that the singularity of $\cX_P$ is a node,
but this can also be extracted from the arguments in either of these two papers.

\begin{lemma}[{\cite[Proposition~1.2]{BeauvillePrym}, \cite[Proposition~1.8]{Sarkisov}}]
\label{lemma:conic-bundle-singularity}
Let $p \colon \cX \to S$ be a flat conic bundle over a smooth surface $S$.
Assume that its discriminant locus $\Delta \subset S$ has a node at a point $P \in S$.
Then $\cX$ has a singular point over $P$ if and only if the fiber~\mbox{$\cX_P = p^{-1}(P)$} is a conic of corank $1$ \textup{(}that is, a union of two distinct lines\textup{)},
and in this case the singularity of $\cX$ over $P$ is a node at the \textup{(}unique\textup{)} singular point of $\cX_P$.
\end{lemma}

The next assertion describes the singular loci of the threefolds $\cX^{5,1}_\tau$.
Recall the morphism $\pi_{5,1}$ defined in~\eqref{eq:y51} and the discriminant set $\bDDtau$ from~\eqref{def:dd-tau}.

\begin{proposition}
\label{proposition:x51-smooth}
The threefold $\cX^{5,1}_\tau$ is smooth for all $\tau \not\in \bDDtau$ \textup{(}including~\mbox{$\tau = \infty$}\textup{)}.
For~$\tau\in\bDDtau$ the singular locus of $\cX^{5,1}_\tau$ is mapped by $\pi_{5,1}$ isomorphically to a subset of~$\P^4$ as follows:
\begin{equation*}\renewcommand{\arraystretch}{1.5}
\begin{array}{|c||c|c|c|c|}
\hline
\tau & 0 & \pm1 & \pm\frac1{\sqrt{-3}} & \pm\frac3{\sqrt{5}}
\\
\hline
\pi_{5,1}(\Sing(\cX^{5,1}_\tau)) & \Upsilon_{15} & \Sigma_{15} & \Sigma_{10} & \Sigma_{6}
\\
\hline
\end{array}
\end{equation*}%
For $\tau \in \bDDtau$ the singularities of $\cX^{5,1}_\tau$ form a single $\fA_5$-orbit,
every singular point $Q$ of~$\cX^{5,1}_\tau$ is a node, and the fiber $p^{-1}(p(Q))$ 
of the conic bundle $p \colon \cX^{5,1}_\tau \to S$ passing through $Q$
is the union of two distinct lines intersecting at $Q$.
\end{proposition}
\begin{proof}
To start with, let us show that for $\tau \ne 0$ the threefold $\cX^{5,1}_\tau$ is smooth along the
exceptional locus of the morphism $\rho_{5,1}$, which by Proposition~\ref{proposition:Stein-factorization} is the reducible surface
\begin{equation}
\label{eq:exc-locus-rho51}
\left( \bigcup_L R_L \right) \cup \left(\bigcup_{\varphi} R_\varphi \right)  = \pi_{5,1}^{-1}(\CR) \subset \cY_{5,1}.
\end{equation}
Recall that each of its irreducible components is a smooth surface in $\cY_{5,1}$ (see Lemmas~\ref{lemma:r-l} and~\ref{lemma:r-i}).
Note that a Cartier divisor in a smooth fourfold is smooth along its intersection with a smooth surface provided that
their scheme intersection is a smooth curve.
So, it is enough to check that the intersections $\cX^{5,1}_\tau \cap R_L$ and $\cX^{5,1}_\tau \cap R_\varphi$ are smooth curves for all~\mbox{$\tau \ne 0$}.
But the divisors $\cX^{5,1}_\tau$ form a pencil, and $\cX^{5,1}_0$
(which by definition is equal to the ramification divisor of $\pi_{5,1}$)
contains all these surfaces.
Therefore,
\begin{equation*}
\cX^{5,1}_\tau \cap R_L = \cX^{5,1}_\infty \cap R_L
\qquad\text{and}\qquad
\cX^{5,1}_\tau \cap R_\varphi = \cX^{5,1}_\infty \cap R_\varphi.
\end{equation*}
So, it is enough to show that $\cX^{5,1}_\infty \cap R_\varphi$ and~$\cX^{5,1}_\infty \cap R_L$ are smooth curves.
But~\mbox{$\cX^{5,1}_\infty = \pi_{5,1}^{-1}(Q_\infty)$},
while $R_\varphi$ and~$R_L$ are the preimages of the 15 lines of
the Cremona--Richmond configuration~$\CR$.
The quadric~$Q_\infty$ intersects all these lines transversally at two points
away from the intersection points of the lines by Remark~\ref{remark:lines-vs-quadric}.
Taking into account Lemma~\ref{lemma:s-w} and Proposition~\ref{proposition:Stein-factorization} we conclude that
$\cX^{5,1}_\infty \cap R_L$ is the union of two disjoint lines, and $\cX^{5,1}_\infty \cap R_\varphi$
is the union of two disjoint smooth conics.

Since the map $\rho_{5,1} \colon \cX_\tau^{5,1} \to \cX_\tau$ is an isomorphism over $\P^4 \setminus \CR$
(because so is the map~\mbox{$\cY_{5,1} \to \cY$}), it follows that for all $\tau \ne 0$ we have
\begin{equation*}
\Sing(\cX_\tau^{5,1}) = \Sing(\cX_\tau) \setminus \CR,
\end{equation*}
and in view of Lemma~\ref{lemma:cxtau-xt}, Theorem~\ref{theorem:Geer}, Remark~\ref{remark:double-quadric-Sing}, and Lemma~\ref{lemma:conic-bundle-singularity},
we obtain the required description of singularities of $\cX_\tau^{5,1}$ for $\tau \ne 0$.

Next, consider the case $\tau = 0$.
The map $p \colon \cX^{5,1}_0 \to S$ is a flat
conic bundle with the discriminant locus being the curve $\Delta_0$, i.e., the union of the 10 lines on $S$.
It follows that~$\cX^{5,1}_0$ is smooth over the complement of the 15 intersection points of the lines on $S$.
Since all these points are nodes of $\Delta_0$, Lemma~\ref{lemma:conic-bundle-singularity} shows that the threefold $\cX^{5,1}_0$
has a singularity over such a point $P$ if and only if the conic $(\cX^{5,1}_0)_P = p^{-1}(P)$ is the union of two distinct lines
(and then the singular point is a node located at the intersection point of these lines).
Since the 15 intersection points of the lines on $S$ form a single $\fA_5$-orbit (see Theorem~\ref{theorem:WE-singular-curves}),
it is enough to check everything over one of them.

Take the intersection point $P \in S$ such that $\varphi(P) = (0:1:1)$.
We know from diagram~\eqref{diagram:varphi-x} that the conic $(\cX_0^{5,1})_P$ is isomorphic to the conic $(\bar\cX_0^{4,2})_{\varphi(P)}$, hence by Proposition~\ref{proposition:verra} it is given by the matrix
\begin{equation}
\label{eq:conic-igusa}
q_0(0:1:1) = \frac12
\begin{pmatrix}
\hphantom{-}0 & \hphantom{-}1 & -1 \\
\hphantom{-}1 & \hphantom{-}0 & \hphantom{-}0 \\
-1 & \hphantom{-}0 & \hphantom{-}0
\end{pmatrix}.
\end{equation}
Its rank equals 2, hence $(\bar\cX_0^{4,2})_{\varphi(P)}$, and thus also $(\cX_0^{5,1})_P$, is a union of two lines.
Moreover, the intersection point of the irreducible components of $(\bar\cX_0^{4,2})_{\varphi(P)}$
is the point~$(0:1:1)$, and using~\eqref{eq:choice-rho-pp} we compute that
\begin{equation*}
\bar{\pi}_{4,2}((0:1:1),(0:1:1)) = (2:-1:-1:2:-1:-1) \in \Upsilon_{15}.
\end{equation*}
By $\fA_5$-equivariance of the map $\pi_{5,1}$ and transitivity of $\fA_5$-action on $\Upsilon_{15}$ (see Corollary~\ref{corollary:a5-action-upsilon15})
we conclude that $\pi_{5,1}(\Sing(\cX^{5,1}_0)) = \Upsilon_{15}$.
\end{proof}

\begin{corollary}\label{corollary:r51-small}
For all $\tau \ne 0,\infty$ the morphism $\pi_{5,1} \colon \cX^{5,1}_\tau \to X_{(\tau^2+1)/4}$ is birational and small.
Also, the morphism $\rho_{5,1} \colon \cX^{5,1}_\infty \to\cX_{\infty}$ is birational and small.
\end{corollary}

\begin{proof}
Indeed, as we have seen in the proof of Proposition~\ref{proposition:x51-smooth}, for $\tau \ne 0$ the non-trivial fibers of $\cX^{5,1}_\tau \to \cX_\tau$ are 30 rational curves,
one over each of the 30 intersection points of~\mbox{$\Sigma_{30} = \CR \cap Q_\infty$}.
Since the map $\pi \colon \cX_\tau \to X_{(\tau^2+1)/4}$ is an isomorphism for $\tau\neq \infty$
by Lemma~\ref{lemma:cxtau-xt}, the assertion follows.
\end{proof}

\begin{remark}
For $\tau = 0$ the surface~\eqref{eq:exc-locus-rho51} is equal to the exceptional locus
of~\mbox{$\pi_{5,1} \colon \cX_0^{5,1} \to X_{1/4}$},
hence this morphism is not small, but is still birational.
\end{remark}

\subsection{Proofs of Theorems~\ref{theorem:three-pencils} and~\ref{theorem:conic-bundles}}
\label{subsection:proofs-second}

For $\tau \ne 0$ the map $\rho_{5,1} \colon \cX^{5,1}_\tau \to \cX_\tau$ is small and birational by Corollary~\ref{corollary:r51-small}.
The same argument works for $\rho_{4,2} \colon \cX^{4,2}_\tau \to \cX_\tau$ without changes.
Finally, smoothness of $\cX^{5,1}_\tau$ for non-special $\tau$ is proved in Proposition~\ref{proposition:x51-smooth}.
The maps $\rho_{5,1}$ and $\pi \circ \rho_{5,1}$ have required equivariance by Remark~\ref{remark:rho-equivariance} and Lemma~\ref{lemma:cxtau-xt}.
The same arguments prove equivariance of the maps~$\rho_{4,2}$ and $\pi \circ \rho_{4,2}$.
This completes the proof of Theorem~\ref{theorem:three-pencils}.

Now let us prove Theorem~\ref{theorem:conic-bundles}.
By Proposition~\ref{proposition:x51-smooth} the total spaces of the conic bundles~\mbox{$p \colon \cX_\tau^{5,1} \to S$} are smooth for $\tau \not\in \bDDtau$,
so since $\rkPic(S)=5$, to show that $p$ is a standard conic bundle for $\tau \not\in \bDDtau$
it is enough to check that $\rkPic(\cX^{5,1}_\tau)=6$ for these~$\tau$.
But since the map $\rho_{5,1} \colon \cX^{5,1}_\tau \to \cX_\tau$ is small, we have
\begin{equation*}
\Pic(\cX^{5,1}_\tau) \cong \Cl(\cX_\tau).
\end{equation*}
Thus the assertion of the theorem follows from Lemma~\ref{lemma:Cl-Z6-upstairs}.
\qed

\begin{remark}\label{remark:Weil-divisors-threefolds}
Assume the notation of Remark~\ref{remark:rho-4-2-Weil-divisor},
and suppose that~\mbox{$t\not\in\{1/4, 1/2, \infty\}$}.
One can check that the restrictions of each hyperplane
$H_{ijk}\subset\P^4$
to $X_t$ splits as the union of two smooth quadric surfaces in $H_{ijk}\cong\P^3$. For $t=1/4$ these two quadric surfaces
collide into a smooth quadric with a non-reduced structure, and
for $t=1/2$ they degenerate into unions of pairs of planes.
Considering the preimages of these surfaces on~$\cX_\tau$, where
as usual $t=(\tau^2+1)/4$, and using Remarks~\ref{remark:rho-4-2-Weil-divisor}
and~\ref{remark:rho-5-1-Weil-divisor}, one can describe the
small resolutions $\rho_{4,2}$ and $\rho_{5,1}$ of singularities of $\cX_\tau$ as blow ups of certain
Weil divisors on~$\cX_\tau$.
\end{remark}

\section{Rationality}
\label{section:applications}

In this section we provide some applications of the results obtained earlier.
Namely, we check that all quartics $X_t$ are unirational,
give a new and uniform proof of irrationality of $\fS_6$-invariant quartics~$X_t$ for $t \not\in \bDD \cup \{\infty\}$ (and also of the threefold $\cX_\infty$),
and rationality of $X_t$ for $t \in \bDD$.

\subsection{Unirationality of $\fS_6$-invariant quartics}
\label{subsection:unirationality}

We start with a short proof of unirationality of the quartics $X_t$ and the threefold $\cX_\infty$.
The next fact is well known.

\begin{lemma}\label{lemma:Verra-unirational}
Let $V$ be an irreducible Verra threefold, i.e., an irreducible
hypersurface of bidegree $(2,2)$ in $\P^2\times\P^2$.
Then $V$ is unirational.
\end{lemma}
\begin{proof}
Let $p_i\colon V\to\P^2$, $i=1,2$, be the natural projections.
Both $p_i$ are (possibly non-flat) conic bundles.
Let $L\subset\P^2$ be a general line, and put $T = p_2^{-1}(L)$.
Since $V$ is irreducible and $L$ is general, the surface $T$ is irreducible by Bertini's theorem.
Also, the map $p_2$ provides the surface $T$ with a conic bundle structure over $L \cong \P^1$,
hence~$T$ is rational.
Note also that $T = V \cap (\P^2 \times L)$ is a divisor of bidegree $(2,2)$ in~\mbox{$\P^2 \times \P^1$},
hence the projection~\mbox{$p_1 \colon T \to \P^2$} is dominant
(actually, $T$ is a rational $2$-section of~$p_1$).
Since~$p_1 \colon V \to \P^2$ is a conic bundle, the standard base change argument implies unirationality of $V$.
\end{proof}

Combining Lemma~\ref{lemma:Verra-unirational} with
Proposition~\ref{proposition:verra}, we obtain

\begin{corollary}
\label{corollary:xt-unirational}
The quartics $X_t$, $t\neq\infty$, and the threefold $\cX_\infty$, are unirational.
\end{corollary}

\begin{remark}
\label{remark:Burkhardt-rationality}
One can use the same approach to prove rationality of the Burkhardt quartic~$X_{1/2}$
(this is a classical fact going back to~\cite{Todd36}, see also Theorem~\ref{theorem:Cheltsov-Shramov} below).
For this consider the corresponding Verra threefold $\bar\cX_1^{4,2} \subset \P^2 \times \P^2$ and
let~\mbox{$T = p_2^{-1}(\overline{P_1P_2}) \subset \bar\cX_1^{4,2}$} be the preimage of the line
passing through two of the points~\eqref{eq:points}, that is, the line~\mbox{$v_3=0$}.
As before, $T$ is a divisor of bidegree~$(2,2)$ in $\P^2 \times \P^1$.
Using~\eqref{eq:q0} and~\eqref{eq:q8} we can rewrite explicitly its equation~\mbox{$q_0(u)(v_1,v_2,0)+q_{\infty}(u)(v_1,v_2,0) = 0$} as
\begin{multline*}
\big(q_0(u)+q_{\infty}(u)\big)(v_1,v_2,0)=\\=
\frac23
\big(u_1v_2 + \omega u_2v_1 + \omega^2 u_3v_1 + \omega u_3v_2\big)
\big(u_1v_2 + \omega^2 u_2v_1 + \omega u_3v_1 + \omega^2 u_3v_2\big),
\end{multline*}
where $\omega$ is a primitive cubic root of unity.
Thus we see that  $T = T_1 \cup T_2$, where $T_i$ is a divisor of bidegree $(1,1)$.
In particular, each~$T_i$ provides a rational section of the conic bundle $p_1 \colon \bar\cX_1^{4,2} \to \P^2$
and rationality of~$\bar\cX_1^{4,2}$ follows.
Since the threefold~$\bar\cX_1^{4,2}$ is birational to the quartic $X_{1/2}$, the rationality of the latter follows as well.
\end{remark}

\subsection{Irrationality of non-special $\fS_6$-invariant quartics}
\label{section:Beauville}

Beauville proved in \cite{Beauville} that the quartic $X_{t}$
is irrational provided that \mbox{$t\not\in\bDD\cup\{\infty\}$}
by using the $\fS_6$-action on the intermediate Jacobian of a suitable resolution of singularities of~$X_t$.
By~\cite{Beauville}, the intermediate Jacobian $J_t$ of the blow up of the~30 singular points of $X_t$ is five-dimensional,
and the action of $\fS_6$ on $J_t$ is faithful; on the other hand,
if it is a product of Jacobians of curves, it cannot have a faithful $\fS_6$-action.
Irrationality of the threefold $\cX_\infty$
was proved using the same approach in \cite[Proposition~6.3]{PrzyjalkowskiShramov}.
With the help of the conic bundle structure on these varieties
constructed in Theorem~\ref{theorem:conic-bundles}, we can give another proof of their irrationality.

\begin{theorem}
\label{theorem:Beauville}
If $t\not\in\bDD\cup\{\infty\}$, then $X_{t}$ is irrational.
Also, the variety $\cX_\infty$ is irrational.
\end{theorem}

\begin{proof}
By Theorem~\ref{theorem:three-pencils} it is enough to show that the threefold $\cX^{5,1}_\tau$ is irrational
for~\mbox{$\tau \not\in \bDDtau$}.
By Theorem~\ref{theorem:conic-bundles} the map $p \colon \cX^{5,1}_\tau \to S$ is a standard conic bundle
with the nodal discriminant curve $\Delta_{s}$ contained in the linear system $|-2K_{S}|$.
Here $s=s(\tau)$ is given by the formula~\eqref{eq:s-via-tau}.
The conic bundle $p$ induces a double cover $\widehat{\Delta}_{s}\to\Delta_{s}$
that by Lemma~\ref{lemma:conic-bundle-singularity} is branched only over the nodes of the curve $\Delta_{s}$.
Applying \cite[Proposition~2.8]{BeauvillePrym}, we see that the intermediate Jacobian of the threefold $\cX^{5,1}_\tau$
is isomorphic as a principally polarized abelian variety to the Prym variety $\mathrm{Prym}(\widehat{\Delta}_{s},\Delta_{s})$.
Now \cite[Main~Theorem]{Shokurov} implies that $\mathrm{Prym}(\widehat{\Delta}_{s},\Delta_{s})$
is not a product of Jacobians of curves, hence~$\cX^{5,1}_\tau$ is irrational.
\end{proof}

\begin{remark}
The intermediate Jacobian of $\cX_\tau^{5,1}$ can be described fairly explicitly.
For instance, it was observed by Dimitri Markushevich that it is isogenous
to the fifth power of an elliptic curve (whose $j$-invariant depends on $\tau$).
Here is a sketch of his argument.
Let~$\widehat{\cX}_\tau \to \cX_\tau$ be a minimal $\fS_6$-equivariant resolution of singularities,
so that~\mbox{$\operatorname{Jac}(\widehat{\cX}_\tau) \cong \operatorname{Jac}(\cX_\tau^{5,1})$}.
The action of the group $\fS_6$ on $\operatorname{Jac}(\widehat{\cX}_\tau)$ can be lifted to
an action of the semidirect product $\fS_6 \ltimes H^3(\widehat{\cX}_\tau,\ZZ)$ on $H^3(\widehat{\cX}_\tau,\CC)$.
Now~\cite[Theorem~3.1]{BS06} proves that there is an $\fS_6$-equivariant isomorphism
\begin{equation*}
H^3(\widehat{\cX}_\tau,\Q) \cong \Q(\fS_6) \oplus \lambda \Q(\fS_6),
\end{equation*}
where $\Q(\fS_6)$ is the root lattice associated with the group $\fS_6$ considered as a Weyl group of Dynkin type~$\mathrm{A}_5$,
and $\lambda = \lambda(\tau)$ is a complex number with positive imaginary part.
Therefore, $\operatorname{Jac}(\widehat{\cX}_\tau)$ is isogenous to $E(\lambda)^5$, where $E(\lambda) = \CC/(\ZZ \oplus \lambda \ZZ)$.
\end{remark}

Note by the way, that there is another popular family of threefolds with five-dimensional intermediate Jacobians, namely, smooth cubic threefolds.
However, it was pointed out by Beauville that the quartics~$X_t$ are not birational to smooth cubics.
Indeed, if a quartic~$X_t$ is birational to a smooth cubic threefold $Y$, then the intermediate Jacobian $J(Y)$ is isomorphic to $J_t$,
and thus there is a faithful $\fS_6$-action on $J(Y)$ (note that $J_t$ must coincide with its Griffiths component in this case).
Torelli theorem for smooth cubic threefolds (see~\cite[Proposition~6]{Beauville-Torelli}) implies that there is a faithful $\fS_6$-action on~$Y$ itself,
which is impossible, because the only cubic threefold with a faithful $\fS_6$-action is the Segre cubic that has ten singular points.

It would be interesting to find out if the quartics $X_t$ with $t\not\in \bDD$ are stably rational or not.

\subsection{Rationality of special $\fS_6$-invariant quartics}

The result of Theorem~\ref{theorem:Beauville} is sharp:
the threefolds $X_{1/2}$, $X_{1/4}$, $X_{1/6}$, and $X_{7/10}$ are rational.
In fact, rationality of the Burkhardt quartic $X_{1/2}$ was proved by Todd in \cite{Todd36} (see also Remark~\ref{remark:Burkhardt-rationality}),
rationality of the Igusa quartic~$X_{1/4}$ follows from rationality of its projectively dual variety (which is the Segre cubic),
and rationality of the quartics $X_{1/6}$ and $X_{7/10}$ is also known, see \cite{Todd33,Todd35,CheltsovShramov}.
However, using our results one can give a uniform proof of rationality of all these threefolds; this proof does not use
explicit rationality constructions.

\begin{theorem}
\label{theorem:Cheltsov-Shramov}
The quartics $X_{1/2}$, $X_{1/4}$, $X_{1/6}$, and $X_{7/10}$ are rational.
\end{theorem}
\begin{proof}
Suppose that $\tau \in \bDDtau$, so that $t\in\bDD$ and $s \in \left\{0, \pm1/\sqrt{125}, \pm1/\sqrt{-3} \right\}$,
where as usual~\mbox{$t = (\tau^2 + 1)/4$} and $s = s(\tau)$, see~\eqref{eq:s-tau}.
By Theorem~\ref{theorem:three-pencils} it is enough to show that~$\cX^{5,1}_\tau$ is rational.

Consider the conic bundle $p \colon \cX^{5,1}_\tau \to S$.
The singular locus of its discriminant $\Delta_s$
is a finite set of nodes, see Theorem~\ref{theorem:WE-singular-curves}.
Actually, by Lemma~\ref{lemma:x51} the set $\mathrm{Sing}(\Delta_s)$ consists of 15 points
when $t = 1/4$ or $t = 1/2$, of~10 points when $t = 1/6$, and of~6 points when~\mbox{$t = 7/10$}.
We also know from Proposition~\ref{proposition:x51-smooth} that all singularities of $\cX^{5,1}_\tau$ are nodes,
and for every singular point $Q$ of $\cX^{5,1}_\tau$ the fiber~$p^{-1}(p(Q))$ is the union of two lines,
with~$Q$ being their intersection point.

The conic bundle $p$ is not standard because the threefold $\cX^{5,1}_\tau$ is singular, so we start by transforming it to a standard one.
Let $\nu\colon\tS\to S$ be the blow up of the quintic del Pezzo surface $S$ at $\mathrm{Sing}(\Delta_s)$,
and consider the base change $p' \colon \cX^{5,1}_\tau\times_S\tS\to\tS$ of the conic bundle~$p$.
Its discriminant curve is the preimage on~$\tS$ of the discriminant curve of~$p$.
In particular, it contains all exceptional curves of the blow up~$\nu$ as irreducible components of multiplicity $2$,
and the corank of the fibers of $p'$ over the points of each of these curves equals~$1$.
Modifying the conic bundle along these lines as in~\cite[Lemma~1.14]{Sarkisov} (see also~\cite[\S2.5]{DK-moduli}),
we can get rid of the corresponding components of the discriminant.
In other words, we obtain a small birational map
\begin{equation}
\label{eq:cx-wcx}
\cX^{5,1}_\tau\times_S\tS\dashrightarrow \widetilde\cX^{5,1}_\tau
\end{equation}
over $\tS$, such that the threefold $\widetilde\cX^{5,1}_\tau$ comes with a flat conic bundle
$\widetilde{p}\colon \widetilde\cX^{5,1}_\tau\to \tilde{S}$ whose discriminant curve is the proper transform $\widetilde\Delta_s\subset \tS$ of $\Delta_s$ with respect to $\nu$.
In particular, the curve $\widetilde\Delta_s$ is smooth (hence also $\widetilde\cX^{5,1}_\tau$ is smooth), and
by Theorem~\ref{theorem:WE-singular-curves}
has ten connected components when $t=1/4$ or $t=1/2$,
five components when~$t=1/6$, and just one component when~$t=7/10$.
Moreover, every connected component of $\widetilde\Delta_s$ is rational.

Since $\widetilde\Delta_s$ is smooth, the conic bundle $\widetilde{p}$ has only simple degenerations.
In particular, it induces an \'etale double covering over $\widetilde\Delta_s$.
Since every connected component $\widetilde\Delta_s^{(i)}\subset \widetilde\Delta_s$ is smooth and rational,
the double covering is trivial, hence the preimage $\widetilde{p}^{-1}(\widetilde\Delta_s^{(i)})$ consists of two irreducible components
\begin{equation*}
\widetilde{p}^{-1}\big(\widetilde\Delta_s^{(i)}\big)=\Theta'_i \cup \Theta''_i,
\end{equation*}
each being a $\P^1$-bundle over $\widetilde\Delta_s^{(i)}$.
Choosing for each $i$ one of them and contracting all chosen components simultaneously over~$\tS$ (see \cite[1.17]{Sarkisov}),
we obtain a commutative diagram
\begin{equation*}
\vcenter{\xymatrix{
\widetilde\cX^{5,1}_\tau \ar@{->}[dr]_{\widetilde{p}} \ar@{->}[rr] &&
\overline\cX^{5,1}_\tau \ar@{->}[dl]^{\bar{p}}
\\
& \tS}}
\end{equation*}
Here the horizontal arrow is a birational morphism, and $\bar{p}$ is an everywhere non-degenerate conic bundle.
Since $\tS$ is a rational surface, its Brauer group is trivial, hence this $\P^1$-bundle is a projectivization of a vector bundle,
hence birational to $\tS \times \P^1$, hence rational.
This means that $\cX^{5,1}_\tau$ is also rational.
\end{proof}

\begin{remark}
The birational transformation $\cX^{5,1}_\tau \dashrightarrow \cX^{5,1}_\tau\times_S\tS\dashrightarrow \widetilde\cX^{5,1}_\tau$
can be described very explicitly, see Construction~I in the proof of~\cite[Theorem~4.2]{CheltsovPrzyjalkowskiShramov}.
It is a composition of the blow ups of all singular points $Q \in \cX^{5,1}_\tau$ followed by the Atiyah flops in the union
of proper transforms of the two irreducible components of the conic~$p^{-1}(p(Q))$, see Proposition~\ref{proposition:x51-smooth}.
\end{remark}

The construction that we used in the proof of Theorem~\ref{theorem:Cheltsov-Shramov} has the following consequence,
which we will need in~\S\ref{section:representation-class-groups}.
Recall the notation of~\eqref{eq:sn1n2}.

\begin{corollary}
\label{corollary:relaive-class-group}
For $\tau \in \bDDtau$ the relative divisor class group $\Cl(\cX^{5,1}_\tau/S) \otimes \Q$ has the following structure as a representation of the group $\fA_5$:
\begin{equation*}\renewcommand{\arraystretch}{1.5}
\begin{array}{|c||c|c|c|c|}
\hline
\tau &  0 &  \pm1 &  \pm\frac1{\sqrt{-3}} & \pm\frac3{\sqrt{5}}
\\
\hline
\Cl(\cX^{5,1}_\tau/S) \otimes \Q &
\mathbf{1} \oplus \Ind_{\fA_{3,2}}^{\fA_5}(\mathbf{1}) &
\mathbf{1} \oplus \Ind_{\fA_{3,2}}^{\fA_5}(\mathbf{-1}) &
\mathbf{1} \oplus \Ind_{\fA_{4}}^{\fA_5}(\mathbf{1}) &
\mathbf{1} \oplus \mathbf{1}
\\
\hline
\end{array}
\end{equation*}%
Here $\Ind_G^{\fA_5}$ stands for the induction functor from the subgroup $G = \fA_4$ or $G = \fA_{3,2} \cong \fS_3$ in $\fA_5$,
while $\mathbf{1}$ stands for the trivial representation, and $\mathbf{-1}$ stands for the sign representation of~$\fS_3$.
The first summand $\mathbf{1}$ in each cell is generated by the canonical class of~$\cX^{5,1}_\tau$.
\end{corollary}
\begin{proof}
The canonical class $K_{\cX^{5,1}_\tau}$ is invariant with respect to the group action, hence generates a trivial subrepresentation in $\Cl(\cX^{5,1}_\tau/S) \otimes \Q$.
Consider the quotient
\begin{equation*}
\Cl_0(\cX^{5,1}_\tau/S) \otimes \Q := \left( \Cl(\cX^{5,1}_\tau/S) \otimes \Q\right) / \Q K_{\cX^{5,1}_\tau}.
\end{equation*}
To describe it we use the notation introduced in the proof of Theorem~\ref{theorem:Cheltsov-Shramov}.
First, we have
\begin{equation*}
\Cl_0(\cX^{5,1}_\tau/S) \cong \Cl_0((\cX^{5,1}_\tau \times_S \tS)/\tS).
\end{equation*}
Furthermore, since~\eqref{eq:cx-wcx} is a small birational map, we have
\begin{equation*}
\Cl_0((\cX^{5,1}_\tau \times_S \tS)/\tS) \cong \Cl_0(\widetilde\cX^{5,1}_\tau/\tS).
\end{equation*}
Finally, it is clear that $\Cl_0(\widetilde\cX^{5,1}_\tau/\tS) \otimes \Q$ is contained in an $\fA_5$-equivariant exact sequence
\begin{equation*}
0 \to \bigoplus \Q \left[\widetilde\Delta^{(i)}_{s(\tau)}\right] \to
\bigoplus (\Q \left[\Theta'_i\right] \oplus \Q \left[\Theta''_i\right]) \to
\Cl_0(\widetilde\cX^{5,1}_\tau/\tS) \otimes \Q \to 0,
\end{equation*}
where we sum up over the set of irreducible components of $\widetilde\Delta_{s(\tau)}$,
and the first map takes the class $\left[\widetilde\Delta^{(i)}_{s(\tau)}\right] \in \Pic(\tS)$
to $\left[\Theta'_i\right] + \left[\Theta''_i\right] \in \Pic(\widetilde\cX^{5,1}_\tau)$.
It follows that $\Cl_0(\widetilde\cX^{5,1}_\tau/\tS) \otimes \Q$ has the basis $\left[\Theta'_i\right] - \left[\Theta''_i\right]$,
and the group $\fA_5$ permutes the basis vectors, possibly changing their signs.

Recall that the group $\fA_5$ acts transitively on the set of irreducible components of~$\widetilde\Delta_{s(\tau)}$ by Theorem~\ref{theorem:WE-singular-curves}.
Let $G \subset \fA_5$ be the stabilizer of some irreducible component of $\widetilde\Delta_{s(\tau)}$, say, of~$\widetilde\Delta^{(0)}_{s(\tau)}$.
The action of $G$ on the set $\{\Theta'_0,\Theta''_0\}$ defines a homomorphism~\mbox{$\upsilon\colon G \to \{\pm 1\}$},
i.e., a one-dimensional representation of $G$, and we conclude that
\begin{equation*}
\Cl_0(\widetilde\cX^{5,1}_\tau/\tS) \otimes \Q \cong \Ind_G^{\fA_5}(\upsilon).
\end{equation*}%
So, it remains to identify the possible stabilizers $G$ for various $\tau$, and the homomorphisms~$\upsilon$.

When $\tau = \pm 3/\sqrt{5}$, the curve $\widetilde\Delta_{s(\tau)}$ is irreducible, hence $G$ is the whole group $\fA_5$,
and since it has no non-trivial one-dimensional representations,
we conclude that
\begin{equation*}
\Cl_0(\widetilde\cX^{5,1}_{\pm 3/\sqrt{5}}/\tS) \otimes \Q \cong \Ind_{\fA_5}^{\fA_5}(\mathbf{1}) \cong \mathbf{1}.
\end{equation*}

When $\tau = \pm 1/\sqrt{-3}$, the curve $\widetilde\Delta_{s(\tau)}$ has five components, $G$ is the subgroup $\fA_4$ of $\fA_5$,
and since again it has no non-trivial one-dimensional representations,
we conclude that
\begin{equation*}
\Cl_0(\widetilde\cX^{5,1}_{\pm 1/\sqrt{-3}}/\tS) \otimes \Q \cong \Ind_{\fA_4}^{\fA_5}(\mathbf{1}).
\end{equation*}

When $\tau = 0$ or $\tau = \pm1$, the curve $\widetilde\Delta_{s(\tau)}$ has ten components (corresponding to the lines on~$S$)
and $G$ is the subgroup $\fA_{3,2} \cong \fS_3$ of $\fA_5$.
It remains to show that it fixes the components $\Theta'_0$ and $\Theta''_0$ when $\tau = 0$, and swaps them when $\tau = \pm 1$.

The stabilizer $\fA_{3,2}$ of a line $L \subset S$ permutes three points of its intersection with other lines on $S$.
Each of these points, in its turn, is stabilized by a transposition in $\fA_{3,2} \cong \fS_3$.
So, it is enough to check how these transpositions act on $\Theta'_0$ and $\Theta''_0$.

Consider the point $P = (0 : 1 : 1)$ as in the proof of Proposition~\ref{proposition:x51-smooth}.
Then it is easy to see that the subgroup of $\fA_5$ that preserves both lines passing through $P$ is generated by the automorphism
\begin{equation*}
g =
\begin{pmatrix}
1 & \hphantom{-}0 & \hphantom{-}0 \\
1 & \hphantom{-}0 & -1 \\
1 & -1 & \hphantom{-}0
\end{pmatrix}
\end{equation*}
of order two of the plane, while the fiber $p^{-1}(P)$ is given by~\eqref{eq:conic-igusa} in the case $\tau = 0$, and by
\begin{equation*}
q_0(0:1:1) + q_\infty(0:1:1) = \frac13
\begin{pmatrix}
\hphantom{-}2 & \hphantom{-}1 & -2 \\
\hphantom{-}1 & \hphantom{-}2 & -1 \\
-2 & -1 & \hphantom{-}2
\end{pmatrix}
\end{equation*}
in the case $\tau = 1$.
Now verifying that $g$ fixes the components of the conic $p^{-1}(P)$ if~$\tau=0$ and swaps the components if~$\tau=1$ is straightforward.

The computation in the case~\mbox{$\tau=-1$} is similar to that in the case $\tau=1$.
\end{proof}

The part of the above argument that identifies the relative class group of a conic bundle in terms of the induced representation
is completely general and can be proved for any conic bundle with only simple degenerations, and for an arbitrary group acting on it.

\section{Representation structure of the class groups}
\label{section:representation-class-groups}

The main result of this section is the description of the $\fS_6$-action on the class groups of the Coble fourfold and of the quartics $X_t$,
and its applications to the equivariant birational geometry of these varieties.
We will be mostly interested in the quartics $X_t$ with $t\neq 1/4,\infty$, because the quartic $X_{1/4}$
has non-isolated singularities, and at the same time its class group is not very
intriguing by Lemma~\ref{lemma:cl-xt} (cf. Remark~\ref{remark:Segre} below),
while the quartic $X_\infty$ is non-reduced;
however, we will also perform the same computations for the threefold~$\cX_\infty$.

\subsection{The result and its applications}
\label{subsection:Cl-and-applications}

We start by stating our main result and its consequences.
We will use the following notation for representations of the symmetric groups.
For each partition $\lambda = (\lambda_1,\lambda_2,\ldots,\lambda_r)$ of an integer $n$ (i.e., a non-increasing sequence of positive integers summing up to $n$) we denote~by
\begin{equation*}
\rR(\lambda) = \rR(\lambda_1,\lambda_2,\ldots,\lambda_r)
\end{equation*}
the irreducible $\Q$-representation of the group $\fS_n$ as described in~\cite[\S4.1]{FultonHarris}.
For instance, $\rR(n)$ is the trivial representation, while $\rR(1^n)$
is the sign representation.
Note that the standard permutation representation is the direct sum $\rR(n) \oplus \rR(n-1,1)$.

We denote by $\rR(\lambda) \boxtimes 1$ and $\rR(\lambda) \boxtimes (-1)$ the representations of the group $\fS_6 \times\mumu_2$,
which are isomorphic to $\rR(\lambda)$ when restricted to $\fS_6$
and on which the non-trivial element of~$\mumu_2$ acts by 1 or $-1$, respectively.

\begin{theorem}\label{theorem:Cl}
The group $\mathrm{Cl}(\mathscr{Y})$ is torsion free and there are
the following isomorphisms of $\fS_6 \times \mumu_2$-representations:
\begin{equation*}
\Cl(\cY) \otimes \Q  \cong \Cl(\cX_\infty) \otimes \Q \cong (\rR(6) \boxtimes 1) \oplus (\rR(3,3) \boxtimes (-1)).
\end{equation*}
In particular, for the natural action of $\fS_6$ there are isomorphisms
of $\fS_6$-representations
\begin{align*}
\Cl(\cY) \otimes \Q  &\cong \Cl(\cX_\infty) \otimes \Q \cong \rR(6)\oplus \rR(3,3),\\
\intertext{while for the twisted action of $\fS_6$ there are isomorphisms of $\fS_6$-representations}
\Cl(\cY) \otimes \Q  &\cong \Cl(\cX_\infty) \otimes \Q \cong \rR(6)\oplus \rR(2,2,2).\\
\intertext{Finally, there are the following isomorphisms of $\fS_6$-representations:}
\Cl(X_t) \otimes \Q 				& \cong \rR(6) \oplus \rR(3,3),\qquad\qquad\qquad\text{for $t \not\in \bDD\cup\{\infty\}$};\\
\Cl\left(X_{\frac12}\right) \otimes \Q 		& \cong \rR(6) \oplus \rR(3,3) \oplus \rR(3,1^3);\\
\Cl\left(X_{\frac16}\right) \otimes \Q 		& \cong \rR(6) \oplus \rR(3,3) \oplus \rR(2,2,2);\\
\Cl\left(X_{\frac7{10}}\right) \otimes \Q  	& \cong \rR(6) \oplus \rR(3,3) \oplus \rR(1^6).
\end{align*}
\end{theorem}

The proof of Theorem~\ref{theorem:Cl} takes the next subsection, and now we discuss its applications to equivariant birational geometry.

Recall that an $n$-dimensional variety $X$ with an action of a group $G$ is {\sf $G$-rational} if
there exists a $G$-equivariant birational map between $X$ and $\P^n$ for some action of~$G$
on~$\P^n$. Also recall that a $G$-equivariant morphism $\phi\colon X\to S$ of normal
varieties acted on by a finite group
$G$ is called a {\sf $G$-Mori fiber space},
if $X$ has terminal singularities, one has~\mbox{$\rkPic(X)^G=\rkCl(X)^G$},
the fibers of $\phi$ are connected and of positive dimension, the anticanonical divisor~$-K_X$
is $\phi$-ample,
and the relative $G$-invariant Picard rank~\mbox{$\rkPic(X/S)^G$} equals~$1$.

The first application of Theorem~\ref{theorem:Cl} is due to the following expectation, which is proved in several particular cases,
cf.~\cite{Mella}, \cite{Shramov2008}, \mbox{\cite[Proof of Theorem~1.1]{CheltsovPrzyjalkowskiShramov-BarBur}}.

\begin{conjecture}
\label{conjecture:Mella}
Let $X$ be either a nodal quartic threefold, or a nodal double covering of a smooth three-dimensional quadric branched over its intersection with a quartic.
Let $G$ be a finite subgroup in $\Aut(X)$ such that
\begin{equation*}
\rkCl(X)^G = 1.
\end{equation*}
If there is a $G$-equivariant birational map $X\dasharrow X'$, where~\mbox{$X'\to S'$} is a $G$-Mori fibre space, then $X\cong X'$. In particular, $X$ is not $G$-rational.
\end{conjecture}

Of course, this applies to each of the $\fS_6$-invariant quartics $X_t$ with $t \ne 1/4,\infty$, and to the threefold $\cX_\infty$ as well.
For each subgroup $G \subset \fS_6$ the rank of the invariant class group~$\Cl(X_t)^G$ can be easily computed from the result of Theorem~\ref{theorem:Cl}
by restricting the representation and computing the multiplicity of the trivial summand.
We used the~GAP package~\cite{GAP} to perform this computation, see~\cite{Code} for the source code.
To state our result in a precise form we first introduce our notation for the (conjugacy classes of) subgroups of~$\fS_6$
that will be used until the end of~\S\ref{subsection:Cl-and-applications}.
We will also use notation~\eqref{eq:sn1n2}.

\begin{notation}\label{notation:groups}
Given a subgroup $G \subset \fS_6$ we denote by $\overline{G} \subset \fS_6$
the image of $G$ under an outer automorphism of $\fS_6$ (it is well-defined up to conjugation).
Furthermore, if~\mbox{$G_1\subset \fS_{n_1},\ldots, G_r\subset \fS_{n_r}$} are subgroups and $n_1+\ldots+n_r\le 6$,
then by~\mbox{$G_1\times\ldots\times G_r$} we denote the corresponding subgroup in
\begin{equation*}
\fS_{n_1}\times\ldots\times\fS_{n_r} \cong \fS_{n_1,\ldots,n_r} \subset \fS_{n_1 + \ldots + n_r} \subset\fS_6.
\end{equation*}

Next, we use the notation $\mumu_d[c_1,\ldots,c_r]$ for a cyclic subgroup of order $d$ generated by a permutation of cycle type $[c_1,\ldots,c_r]$.
We abbreviate $\mumu_5[5]$ to just $\mumu_5$.

By $\VV_4$ we denote the Klein four-group, i.e., the unique subgroup of order 4 in $\fA_4 \subset  \fS_4$.
By $\VV_{4,2}$ we denote a subgroup of $\fS_{4,2} \subset \fS_6$ whose projection to the first factor $\fS_4$ gives an isomorphism
with $\VV_4$, while the projection to the second factor $\fS_2$ is surjective.

By $\DD_{2n}$ we denote the dihedral group of order $2n$.
It is naturally embedded into the group $\fS_n$, so for $n\le 6$ it is a subgroup of $\fS_6$;
note that $\DD_{12} = \overline{\fS_{3,2}}$.

There are four conjugacy classes of subgroups isomorphic to $\DD_8$ in $\fS_6$. They can be described as follows.
The first class contains subgroups of (the standard) $\fS_4$ in $\fS_6$; according to the above conventions,
we will refer to subgroups from this conjugacy class simply as $\DD_8$.
There are three non-trivial homomorphisms
\begin{equation*}
\upsilon_\circ \colon \DD_8 \to \mumu_2,\quad
\upsilon_+ \colon \DD_8 \to \mumu_2,\quad
\upsilon_\times \colon \DD_8 \to \mumu_2,
\end{equation*}
determined by their kernels
\begin{equation*}
\Ker(\upsilon_\circ) = \mumu_4[4],\qquad
\Ker(\upsilon_+) = \VV_4,\qquad
\Ker(\upsilon_\times) = \fS_{2,2}.
\end{equation*}
Thinking of these as of subgroups of symmetries of a square, the first is generated by rotations,
the second by reflections with respect to the lines passing through the middle points of its opposite sides,
and the third by reflections with respect to the diagonals; this is the mnemonics for the notation $\circ$, $+$, and $\times$.
We denote by $\DD_8^\circ$, $\DD_8^+$, and $\DD_8^\times$ the images of the map
\begin{equation*}
\DD_8 \xrightarrow{\ (\mathrm{id},\upsilon)\ } \fS_4 \times \mumu_2  \cong \fS_{4,2} \subset \fS_6
\end{equation*}
for $\upsilon = \upsilon_\circ$, $\upsilon_+$, and $\upsilon_\times$, respectively.
Note that $\DD_8^\circ = \overline{\DD_8}$.
\end{notation}

The intersection $\fS_5 \cap \overline{\fS_5}$ of a standard and a non-standard subgroups $\fS_5$ is a subgroup of order $20$ isomorphic to $\mumu_5\rtimes\mumu_4$, and such groups form a unique conjugacy class of subgroups
of order $20$ in $\fS_6$. Also, the subgroups $\mumu_4\times\mumu_2$, $\mumu_3\times\mumu_3$, $\DD_{10}$, \mbox{$\DD_8\times\fS_2$},
$\CCZ$, $\CCZz$, and $\SSZ$ of $\fS_6$ are unique up to conjugation.

Finally, recall the definitions~\eqref{eq:action-standard} of the natural and~\eqref{eq:action-twisted} of the twisted actions of $\fS_6$ on the Coble
fourfold $\cY$ and on the threefold $\cX_\infty\subset\cY$.
Theorem~\ref{theorem:Cl} implies:

\begin{corollary}
\label{corollary:rk-1}
The following table contains a complete list \textup{(}ordered by cardinality\textup{)} of subgroups $G \subset \fS_6$ such that~\mbox{$\rkCl(X)^G = 1$},
where $X$ is either $X_t$, or $\cX_\infty$, or $\cY$.
\begin{equation*}\renewcommand{\arraystretch}{1.3}
\begin{array}{|c|c|}
\hline
X,\ \text{action of}\ \fS_6 & G \\
\hline
\parbox[t]{4cm}{$X_t$,\hfill $t \not\in \bDD \cup \{ \infty \}$;\\
$\cX_\infty$,\hfill natural action;\\
$\cY$,\hfill natural action\hphantom{;}\vstrut}
&
\parbox[t]{11cm}{$\fS_6$, $\fA_6$, $\fS_5$, $\overline{\fS_5}$, $\SSZ$, $\fA_5$, $\fS_{4,2}$, $\overline{\fS_{4,2}}$,
$\CCZz$, $\overline{\fS_{3,3}}$, $\fS_4$, $\overline{\fS_4}$, $\fA_{4,2}$, $\fA_4\times\fS_2$,
$\Karz$, $\overline{\fS_3\times\mumu_3}$, $\DD_8\times\fS_2$, $\fA_4$, $\overline{\fS_{3,2}}$,
$\mumu_4\times\mumu_2$, $\VV_4\times\mumu_2$, $\overline{\DD_8}$, $\DD_8^{\times}$, $\overline{\fS_3}$, $\VV_{4,2}$\vstrut}
\\
\hline
\parbox[t]{4cm}{$\cX_\infty$,\hfill twisted action;\\
$\cY$,\hfill twisted action\hphantom{;}\vstrut}
& \parbox[t]{11cm}{$\fS_6$, $\fA_6$, $\fS_5$, $\SSZ$, $\fA_5$,
$\fS_{4,2}$, $\CCZz$, $\fS_{3,3}$,
$\fS_4$, $\fA_{4,2}$, $\fA_4\times\fS_2$,
$\fS_3\times\mumu_3$, $\fS_{3,2}$, $\fA_4$, $\fS_3$, $\mumu_6[3,2]$\vstrut}
\\
\hline
X_{1/2} & \parbox[t]{11cm}{$\fS_6$, $\fA_6$, $\fS_5$, $\overline{\fS_5}$,
$\SSZ$, $\fA_5$, $\fS_{4,2}$, $\overline{\fS_{4,2}}$, $\CCZz$, $\overline{\fS_{3,3}}$, $\fS_4$, $\fA_{4,2}$,
$\Karz$, $\DD_8\times\fS_2$\vstrut}
\\
\hline
X_{1/6} & \parbox[t]{11cm}{$\fS_6$, $\fA_6$, $\fS_5$,
$\SSZ$, $\fA_5$, $\fS_{4,2}$, $\CCZz$,
$\fS_4$, $\fA_{4,2}$, $\fA_4\times\fS_2$, $\fA_4$\vstrut}
\\
\hline
X_{7/10} & \parbox[t]{11cm}{$\fS_6$, $\fS_5$, $\overline{\fS_5}$,
$\SSZ$, $\fS_{4,2}$, $\overline{\fS_{4,2}}$, $\overline{\fS_{3,3}}$,
$\fS_4$, $\overline{\fS_4}$, ${\fA_4\times\fS_2}$, $\Karz$,
$\overline{\fS_3\times\mumu_3}$, $\DD_8\times\fS_2$, $\overline{\fS_{3,2}}$,
$\mumu_4\times\mumu_2$, $\overline{\DD_8}$, $\DD_8^{\times}$, $\VV_4\times\mumu_2$,
$\overline{\fS_3}$, $\VV_{4,2}$\vstrut}
\\
\hline
\end{array}
\end{equation*}
If $X$ is either $\cX_\infty$ or $\cY$, and $G$ is any subgroup of $\fS_6\times\mumu_2$ that contains the second factor,
then one also has $\rkCl(X)^G=1$.
\end{corollary}

In particular, Conjecture~\ref{conjecture:Mella} suggests that the varieties
listed in Corollary~\ref{corollary:rk-1} are not $G$-rational with respect to the corresponding groups.

Another interesting case of $G$-equivariant behavior arises when $\rkCl(X)^G=2$.
The following result is well known to experts.

\begin{proposition}[{cf.~\cite{Corti}, \cite{HaconMcKernan}}]
\label{proposition:Corti}
Let $X$ be a terminal Fano variety \textup{(}so that, in particular, the canonical class $K_X$ is a $\Q$-Cartier
divisor\textup{)}. Let $G$ be a finite subgroup in $\Aut(X)$ such that~\mbox{$\rkCl(X)^G=2$} and~\mbox{$\rkPic(X)^G = 1$}.
Then there exists a unique $G$-equivariant diagram
\begin{equation}\label{eq:G-Sarkisov}
\vcenter{\xymatrix{
& X_+' \ar@{->}[dl]_{p_+}
&& X_+\ar@{->}[dr]_{f_+}
\ar@{-->}[ll]_{\psi_+} \ar@{-->}[rr]^{\iota} &&
X_-\ar@{->}[dl]^{f_-}
\ar@{-->}[rr]^{\psi_-} && X_-'\ar@{->}[dr]^{p_-} & \\
Z_+ &&&& X &&&& Z_-
}}
\end{equation}
Here $X_{\pm}$ are varieties with terminal singularities such that
\begin{equation*}
\rkPic(X_\pm)^G = \rkCl(X_\pm)^G = 2,
\qquad
\rkPic(X_\pm/X)^G=1,
\end{equation*}
the maps~$f_{\pm}$ are small birational morphisms,
the map~$\iota$ is a non-trivial $G$-flop,
the~maps~$\psi_{\pm}$ are small and birational \textup{(}and possibly are just isomorphisms\textup{)},
the varieties~$X_{\pm}'$ have terminal singularities,
\begin{equation*}
\rkPic(X_\pm')^G = \rkCl(X_\pm')^G = 2,
\end{equation*}
and each of the maps~$p_\pm$, is either a $K_{X_\pm'}$-negative divisorial contraction onto a terminal Fano variety $Z_\pm$ with
$\rkCl(Z_\pm)^G=1$, or a $G$-Mori fibration.
\end{proposition}

The diagram~\eqref{eq:G-Sarkisov} is a special case of a so-called {\sf $G$-Sarkisov link}
(that is a $G$-equivariant version of a usual Sarkisov link, see e.g.~\cite[Definition~3.4]{Corti}
or~\cite[Theorem~1.6.14]{Cheltsov-UMN} for notation).
One sometimes says that
the link~\eqref{eq:G-Sarkisov} is {\sf centered at $X$}.

Theorem~\ref{theorem:Cl} allows us to write down a complete
list of subgroups $G \subset \fS_6$ for which Proposition~\ref{proposition:Corti} can be used
(as before, we obtained it with the help of the GAP package~\cite{GAP}, see~\cite{Code} for the source code).

\begin{corollary}\label{corollary:rk-2}
The following table contains a complete list \textup{(}ordered by cardinality\textup{)} of subgroups $G \subset \fS_6$ such that~\mbox{$\rkCl(X)^G = 2$},
where $X$ is either $X_t$, or $\cX_\infty$, or $\cY$.
\begin{equation*}\renewcommand{\arraystretch}{1.3}
\begin{array}{|c|c|}
\hline
X,\ \text{action of}\ \fS_6 & G \\
\hline
\parbox[t]{4cm}{$X_t$,\hfill $t \not\in \bDD \cup \{ \infty \}$;\\
$\cX_\infty$,\hfill natural action;\\
$\cY$,\hfill natural action\hphantom{;}\vstrut}
&
\parbox[t]{11cm}{$\overline{\fA_5}$, $\fS_{3,3}$,
$\overline{\fA_{4,2}}$, $\overline{\fA_4\times\fS_2}$, $\fS_3\times\mumu_3$,
$\CCZ$, $\fS_{3,2}$, $\DD_{10}$, $\mumu_3\times\mumu_3$, $\DD_8$, $\DD_8^{+}$,
$\mumu_2\times\mumu_2\times\mumu_2$, $\fS_3$, $\fA_{3,2}$, $\mumu_6[6]$, $\mumu_6[3,2]$,
$\mumu_5$, $\mumu_4[4]$, $\mumu_4[4,2]$, $\mumu_2[2,2]\times\mumu_2[2]$, $\mumu_3[3]$,
$\mumu_2[2,2,2]$\vstrut}
\\
\hline
\parbox[t]{4cm}{$\cX_\infty$,\hfill twisted action;\\
$\cY$,\hfill twisted action\hphantom{;}\vstrut}
&
\parbox[t]{11cm}{$\overline{\fS_5}$, $\overline{\fA_5}$, $\overline{\fS_{4,2}}$, $\overline{\fS_{3,3}}$,
$\overline{\fA_{4,2}}$, $\overline{\fA_4\times\fS_2}$, $\Karz$, $\overline{\fS_3\times\mumu_3}$, $\CCZ$,
$\DD_8\times\fS_2$, $\DD_{10}$,
$\mumu_3\times\mumu_3$,  $\DD_8$, $\DD_8^{\times}$, $\DD_8^{+}$,
$\mumu_4\times\mumu_2$, $\mumu_2\times\mumu_2\times\mumu_2$,
$\fA_{3,2}$, $\mumu_5$, $\mumu_4[4,2]$,
$\mumu_2[2]\times\mumu_2[2]$, $\mumu_3[3]$\vstrut}
\\
\hline
X_{1/2} &
\parbox[t]{11cm}{$\overline{\fA_5}$, $\fS_{3,3}$, $\overline{\fS_4}$,
$\overline{\fA_{4,2}}$, $\fA_4\times\fS_2$, $\overline{\fA_4\times\fS_2}$,
$\overline{\fS_3\times\mumu_3}$, $\CCZ$, $\fS_{3,2}$, $\overline{\fS_{3,2}}$,
$\fA_4$, $\DD_{10}$, $\DD_8$, $\overline{\DD_8}$, $\DD_8^{\times}$,
$\DD_8^{+}$, $\mumu_4\times\mumu_2$, $\mumu_2\times\mumu_2\times\mumu_2$,
$\VV_4\times\mumu_2$\vstrut}
\\
\hline
X_{1/6} &
\parbox[t]{11cm}{$\overline{\fS_5}$, $\overline{\fS_{4,2}}$, $\fS_{3,3}$, $\overline{\fS_{3,3}}$,
$\Karz$, $\fS_3\times\mumu_3$, $\overline{\fS_3\times\mumu_3}$, $\DD_8\times\mumu_2$, $\fS_{3,2}$,
 $\DD_8^{\times}$, $\mumu_4\times\mumu_2$, $\fS_3$, $\mumu_6[3,2]$\vstrut}
\\
\hline
X_{7/10} &
\parbox[t]{11cm}{$\fA_6$, $\fA_5$, $\fS_{3,3}$, $\CCZz$, $\fA_{4,2}$, $\overline{\fA_4\times\fS_2}$, $\fS_3\times\mumu_3$,
$\fS_{3,2}$, $\fA_4$, $\DD_8$, $\mumu_2\times\mumu_2\times\mumu_2$, $\fS_3$, $\mumu_6[6]$, $\mumu_6[3,2]$,
$\mumu_4[4]$, ${\mumu_2[2,2]\times\mumu_2[2]}$, $\mumu_2[2,2,2]$\vstrut}
\\
\hline
\end{array}
\end{equation*}
In particular, for each of these varieties there is a $G$-Sarkisov link~\eqref{eq:G-Sarkisov} centered at $X$ with respect to the corresponding groups.
\end{corollary}

\begin{example}
If $t\not\in\bDD\cup\{\infty\}$ and $G = \overline{\fA_5}$, the $G$-Sarkisov link~\eqref{eq:G-Sarkisov}
is obtained by restricting the diagram~\eqref{eq:diagram-two-factorializations-Y}:
\begin{equation*}
\vcenter{\xymatrix{
& \cX^{5,1}_\tau \ar@{->}[dl]_{p} \ar@{->}[drr]^{\pi \circ \rho_{5,1}} \ar@{-->}[rrrr]^{\iota} &&&&
\cX^{5,1}_{-\tau} \ar@{->}[dr]^{p} \ar@{->}[dll]_{\pi \circ \rho_{5,1}}
\\
S &&& X_{t} &&& S
}}
\end{equation*}
Here $t = (\tau^2+1)/4$, $\iota$ is the restriction of the map $\rho_{5,1}^{-1} \circ \sigma \circ \rho_{5,1}$ to~$\cX^{5,1}_\tau$
(it is a composition of $30$ Atiyah flops), and $\psi_\pm$ are the identity maps.
The map $\iota$ can be also defined as the map induced by an action of an odd permutation
in the subgroup $\overline{\fS_5} \subset \fS_6$ containing~$\overline{\fA_5}$.
\end{example}

\begin{example}
If $G=\overline{\fS_5}$,
then the $G$-Sarkisov link~\eqref{eq:G-Sarkisov} for $\cX_{\infty}$ comes
from a restriction of the commutative
diagram~\eqref{eq:diagram-two-factorializations-Y}
to $\cX_\infty$ (recall that $\cX_\infty$ is
$\Aut(\cY)$-invariant).
\end{example}

\subsection{Class group computation}

In this section we prove Theorem~\ref{theorem:Cl}.
We start with a description of the $\fS_5$-action on the Picard group of the quintic del Pezzo surface.

\begin{lemma}\label{lemma:Pic-S-representation}
There is an isomorphism of $\fS_5$-representations
\begin{equation*}
\Pic(S) \otimes \Q \cong \rR(5) \oplus \rR(4,1).
\end{equation*}
\end{lemma}

\begin{proof}
The surface $S$ can be obtained as a blow up of $\P^2$ in four points, and this blow up
is $\fS_4$-invariant. Therefore, one has
\begin{equation*}
(\Pic(S) \otimes \Q)\vert_{\fS_4}\cong \rR(4)\oplus \rR(4)\oplus \rR(3,1).
\end{equation*}
Here the first summand is the pullback of the line class, and the last two form the permutation representation spanned by the classes of the exceptional divisors of the blow up.
Now the assertion easily follows, since
\begin{equation}
\label{eq:s5-s4}
\rR(5)\vert_{\fS_4} \cong \rR(4),\qquad
\rR(4,1)\vert_{\fS_4} \cong \rR(4) \oplus \rR(3,1),
\end{equation}
and moreover, by Pieri's rule (see~\cite[Exercise~4.44]{FultonHarris}) the irreducible $\fS_5$-representations~$\rR(5)$ and~$\rR(4,1)$ are
the only ones that restrict to~$\fS_4$ as sums of~$\rR(4)$'s and~$\rR(3,1)$'s.
\end{proof}

Further on we will use a similar argument to describe an $\fS_6$-representation from its restriction to a non-standard subgroup $\fS_5$.
For this the following calculation is quite useful.

\begin{lemma}
\label{lemma:restriction-s6-s5}
The following table contains all irreducible representations $V$ of $\fS_6$, their images $\overline{V}$ under an outer automorphism of $\fS_6$,
and the restrictions of $V$ and $\overline{V}$ to a standard subgroup~$\fS_5$.
\begin{equation*}\renewcommand{\arraystretch}{1.3}
\begin{array}{|c|c|c|c|c|}
\hline
\dim V & V & \overline{V} & V\vert_{\fS_5} & \overline{V}\vert_{\fS_5} \\
\hline
1 & \multicolumn{2}{c|}{\rR(6)} & \multicolumn{2}{c|}{\rR(5)} \\
\hline
1 & \multicolumn{2}{c|}{\rR(1^6)} & \multicolumn{2}{c|}{\rR(1^5)} \\
\hline
5 & \rR(5,1) & \rR(2^3) & \rR(5) \oplus \rR(4,1) & \rR(2^2,1) \\
\hline
5 & \rR(2,1^4) & \rR(3^2) & \rR(2,1^3) \oplus \rR(1^5) & \rR(3,2) \\
\hline
9 & \multicolumn{2}{c|}{\rR(4,2)} & \multicolumn{2}{c|}{\rR(4,1) \oplus \rR(3,2)} \\
\hline
9 & \multicolumn{2}{c|}{\rR(2^2,1^2)} & \multicolumn{2}{c|}{\rR(2^2,1) \oplus \rR(2,1^3)} \\
\hline
10 & \rR(4,1^2) & \rR(3,1^3) & \rR(4,1) \oplus \rR(3,1^2) & \rR(3,1^2) \oplus \rR(2,1^3) \\
\hline
16 & \multicolumn{2}{c|}{\rR(3,2,1)} & \multicolumn{2}{c|}{\rR(3,2) \oplus \rR(3,1^2) \oplus \rR(2^2,1)} \\
\hline
\end{array}
\end{equation*}
\end{lemma}
\begin{proof}
The restrictions to~$\fS_5$ are computed by Pieri's rule, so we only need to explain the action of an outer automorphism.
For this note that an outer automorphism acts on the conjugacy classes of $\fS_6$ by swapping the following cycle types
\begin{equation*}
[2] \leftrightarrow [2,2,2],
\qquad
[3] \leftrightarrow [3,3],
\qquad
[6] \leftrightarrow [3,2],
\end{equation*}
and fixing the other types.
By using the character table of $\fS_6$ (see for instance~\mbox{\cite[Example~19.17]{JamesLiebeck}}) it is then straightforward to check that an outer automorphism swaps
\begin{equation*}
\rR(5,1) \leftrightarrow \rR(2,2,2),
\qquad
\rR(2,1^4) \leftrightarrow \rR(3,3),
\qquad
\rR(4,1^2) \leftrightarrow \rR(3,1^3),
\qquad
\end{equation*}
and fixes the other irreducible representations.
\end{proof}

Now we are ready to prove the part of Theorem~\ref{theorem:Cl} concerning the Coble fourfold.

\begin{proposition}
\label{proposition:cl-cy}
The group $\Cl(\cY)$ is torsion free, and there is an isomorphism
\begin{equation*}
\Cl(\cY) \otimes \Q \cong (\rR(6) \boxtimes 1) \oplus (\rR(3,3) \boxtimes (-1))
\end{equation*}
of representations of the group $\Aut(\cY) \cong \fS_6 \times \mumu_2$.
\end{proposition}
\begin{proof}
Since $\cY_{5,1} \to \cY$ is a small $\fS_5$-equivariant resolution, we have an $\fS_5$-equivariant isomorphism  $\Cl(\cY) \cong \Pic(\cY_{5,1})$
with respect to the twisted action of a non-standard subgroup $\fS_5$.
Since $\cY_{5,1}$ is a $\P^2$-bundle over the quintic del Pezzo surface $S$, we have an $\fS_5$-equivariant direct sum decomposition
\begin{equation*}
\Pic(\cY_{5,1}) = \ZZ H \oplus p^*(\Pic(S)).
\end{equation*}
Here the first summand is generated by the pullback of the hyperplane class of $\P^4$ under the map $\pi \circ \rho_{5,1}$,
and so is $\fS_5$-invariant.
This proves that $\Cl(\cY)$ is torsion free.

Furthermore, it follows from Lemma~\ref{lemma:Pic-S-representation} that there is an isomorphism of $\fS_5$-representations
\begin{equation*}
(\Cl(\cY) \otimes \Q)\vert_{\fS_5} \cong \rR(5) \oplus \rR(5) \oplus \rR(4,1).
\end{equation*}
Since the embedding of $\fS_5 \hookrightarrow \fS_6$ is non-standard,
it follows from Lemma~\ref{lemma:restriction-s6-s5} that
\begin{equation*}
(\Cl(\cY) \otimes \Q)\vert_{\fS_6} \cong \rR(6) \oplus \rR(2,2,2);
\end{equation*}
we emphasize the fact that this isomorphism holds for the \emph{twisted} action of $\fS_6$ on $\cY$.
The first summand $\rR(6)$ is generated by the class $H$, hence lifts to $\rR(6) \boxtimes 1$ as a representation of $\fS_6 \times \mumu_2$.
Since the quotient of $\cY$ by the Galois involution $\sigma$ is $\P^4$ and its class group is of rank 1, it follows that the action of $\mumu_2$
on the second summand $\rR(2,2,2)$ is non-trivial.
Hence the \emph{natural} action of $\fS_6$ on the second summand is obtained from $\rR(2,2,2)$ by the sign twist,
i.e., the corresponding representation is $\rR(3,3)$
(recall that the sign twist modifies an irreducible representation by a transposition of its partition),
and the assertion of the proposition follows.
\end{proof}

Below we will also need to describe certain $\fS_5$-representations from their restrictions to~$\fA_5$.
For this the following calculation is useful.
Denote by $R_1$, $R'_3$, $R''_3$, $R_4$, and~$R_5$
the irreducible representations of the group $\fA_5$ of dimensions $1,3,3,4$, and~$5$, respectively,
see for instance~\cite[Exercise~3.5]{FultonHarris}.

\begin{lemma}
\label{lemma:restriction-s5-a5}
The following table contains all irreducible representation of $\fS_5$ and their restrictions to $\fA_5$.
\begin{equation*}\renewcommand{\arraystretch}{1.3}
\begin{array}{|c|c|c|c|c|c|c|c|}
\hline
\rR(\lambda) & \rR(5) & \rR(1^5) & \rR(4,1) & \rR(2,1^3) & \rR(3,2) & \rR(2^2,1) & \rR(3,1^2) \\
\hline
\rR(\lambda)\vert_{\fA_5} & \multicolumn{2}{|c|}{R_1} & \multicolumn{2}{c|}{R_4} & \multicolumn{2}{c|}{R_5} & R'_3 \oplus R''_3\\
\hline
\end{array}
\end{equation*}
\end{lemma}
\begin{proof}
It is enough to know that a restriction of an $\fS_5$-representation $\rR(\lambda)$ to $\fA_5$
contains the trivial subrepresentation $R_1$ if and only if
$\rR(\lambda)$ is trivial or is the sign representation, i.e., if $\lambda = (5)$ or $\lambda = (1^5)$.
This follows from Frobenius reciprocity, because
\begin{equation*}
\Ind_{\fA_5}^{\fS_5}(R_1) \cong \rR(5) \oplus \rR(1^5).
\end{equation*}
With this in mind, there is only one way to represent the dimensions of $\rR(\lambda)$ as sums of dimensions of irreducible $\fA_5$-representations.
It remains to notice that the $\fS_5$-representation $\rR(3,1^2)$ is defined over
$\Q$, while both three-dimensional $\fA_5$-representations~$R_3'$ and~$R_3''$ are not, so the restriction
of $\rR(3,1^2)$ to $\fA_5$ splits as $R_3'\oplus R_3''$.
\end{proof}

Now we are almost ready to attack the class groups of the quartics $X_t$.
For each $\tau$ we have a natural composition
\begin{multline}
\label{eq:Cl-long-composition}
\Cl(\cY) \cong
\Cl(\cY \setminus \CR) \cong
\Pic(\cY \setminus \CR) \xrightarrow{\ \operatorname{\mathsf{res}}\ } \\
\xrightarrow{\ \operatorname{\mathsf{res}}\ }
\Pic(\cX_\tau \setminus \CR) \hookrightarrow
\Cl(\cX_\tau \setminus \CR) \cong \Cl(\cX_\tau).
\end{multline}
Here $\operatorname{\mathsf{res}}$ denotes the restriction map.
The first and the last isomorphisms take place since the Cremona--Richmond configuration $\CR=\Sing(\cY)$
has codimension greater than 1 both in $\cY$ and $\cX_\tau$,
and the second isomorphism follows from smoothness of $\cY \setminus \CR$.

\begin{lemma}
\label{lemma:morphism-cly-clcx}
For all $\tau \ne 0$ the composition $\Cl(\cY) \to \Cl(\cX_\tau)$ of the maps in~\eqref{eq:Cl-long-composition}
is an~$\fS_6$-equivariant embedding with respect to the natural action of $\fS_6$.
For $\tau = \infty$ it is an~$\fS_6 \times \mumu_2$-equivariant embedding.
Moreover, for $\tau \not\in \bDDtau$ it is an isomorphism.
\end{lemma}
\begin{proof}
All the maps in~\eqref{eq:Cl-long-composition} are equivariant with respect to the natural action of~$\fS_6$ (or of the whole group $\fS_6 \times \mumu_2$ in case $\tau = \infty$),
hence so is the composition, and it remains to prove injectivity.
For this we forget about the $\fS_6$-action and consider the diagram
\begin{equation}
\label{diagram:pic-y-cx}
\vcenter{\xymatrix@C=4em{
\Pic(\cY_{5,1}) \ar[r]^{\operatorname{\mathsf{res}}} \ar[d]_{(\rho_{5,1})_*} &
\Pic(\cX^{5,1}_\tau) \ar[d]^{(\rho_{5,1})_*}
\\
\Cl(\cY) \ar[r] &
\Cl(\cX_\tau)
}}
\end{equation}
which is easily seen to be commutative.
The vertical arrows are isomorphisms, since the birational maps $\rho_{5,1} \colon \cY_{5,1} \to \cY$ and~$\rho_{5,1} \colon \cX^{5,1}_\tau \to \cX_\tau$ for $\tau \ne 0$
are small by Theorems~\ref{theorem:coble} and~\ref{theorem:three-pencils}.
So, it is enough to check that the morphism $\operatorname{\mathsf{res}}$ is injective, which is obvious,
since $\cY_{5,1}$ is a $\P^2$-bundle over~$S$ and~$\cX^{5,1}_\tau$ is a (flat) conic bundle inside $\cY_{5,1}$.

Moreover, for $\tau \not\in \bDDtau$ the conic bundle is standard,
hence the image of the top arrow is a sublattice of index~$2$ or~$1$,
depending on whether the conic bundle has a rational section or not.
Since we also know from \cite{Beauville} or Theorem~\ref{theorem:Beauville} that for~\mbox{$\tau \not\in \bDDtau$}
the threefold $\cX_{\tau}^{5,1}$ is not rational,
we conclude that the conic bundle~\mbox{$p\colon \cX_{\tau}^{5,1}\to S$} has no rational sections,
and thus $\operatorname{\mathsf{res}}$ is an isomorphism.
\end{proof}

\begin{remark}
\label{remark:clxt-nonspecial}
Recall that by Lemma~\ref{lemma:cxtau-xt} for $\tau \ne 0, \infty$ one has an isomorphism $\cX_\tau \cong X_t$ for $t = (\tau^2 + 1)/4$.
Thus Proposition~\ref{proposition:cl-cy} and Lemma~\ref{lemma:morphism-cly-clcx} provide a description
of $\Cl(X_t)$ for all $t \not\in\bDD\cup\{\infty\}$.
\end{remark}

It remains to analyze the class groups of the special quartics $X_t$.

We can think of the map~\eqref{eq:Cl-long-composition} as of a map~\mbox{$\Cl(\cY) \to \Cl(X_t)$};
this map is $\fS_6$-equivariant, where the action of $\fS_6$ on $\cY$ is natural.
We denote the cokernel of this map by
\begin{equation*}
\VCl(X_t) := \Cl(X_t)/\Cl(\cY),
\end{equation*}
and refer to this group as the {\sf excess class group} of $X_t$.
To prove Theorem~\ref{theorem:Cl} we need to compute the latter group for $t = 1/2$, $1/6$, and $7/10$ as an $\fS_6$-representation.
For this we need a couple of observations.

\begin{lemma}
\label{lemma:Cl-standard-S4}
For a standard subgroup $\fS_4 \subset \fS_6$ we have $\rkCl(X_t)^{\fS_4} = 1$ for any $t \ne \infty$.
In particular, we have $\rk\VCl(X_t)^{\fS_4} = 0$ for any $t \ne \infty$.
\end{lemma}
\begin{proof}
We may assume that $\fS_4$ preserves the homogeneous coordinates $x_5$ and $x_6$ on $\P^5$.
Denote $\bp_i := x_1^i+\ldots+x_6^i$.
Consider the quotients $\P^5/\fS_4$ and $X_t/\fS_4$.
Then
\begin{equation*}
\P^5/\fS_4 \cong \P(1,1,1,2,3,4),
\end{equation*}
where the weighted homogeneous coordinates of weights $1,1,1,2,3$, and $4$ correspond to
the $\fS_4$-invariants $x_5$, $x_6$, $\bp_1$, $\bp_2$, $\bp_3$, and $\bp_4$, respectively.
The quotient variety $X_t/\fS_4$ is given in $\P(1,1,1,2,3,4)$ by the equations
\begin{equation*}
\bp_1 = \bp_4 - t\bp_2^2 = 0,
\end{equation*}
so that $X_t/\fS_4 \cong \P(1,1,2,3)$.
Therefore, we have $\rkCl(X_t)^{\fS_4} = \rkCl(X_t/\fS_4) = 1$, see for instance \cite[1.7.5]{Fulton}.
Since also $\rkCl(\cY)^{\fS_4} = 1$ (see Corollary~\ref{corollary:rk-1}), it follows that~\mbox{$\rk\VCl(X_t)^{\fS_4} = 0$}.
\end{proof}

\begin{remark}[{cf.~\cite[Remark~2.11]{CheltsovPrzyjalkowskiShramov-BarBur}}]
\label{remark:guilt-complex}
An argument similar to the proof of Lemma~\ref{lemma:Cl-standard-S4} was (incorrectly!)
used in the proof of~\cite[Theorem~1.20]{CheltsovShramov-A6} for the standard subgroup~\mbox{$\fA_{4,2}\cong\fS_4$} in~$\fS_6$
to deduce that~\mbox{$\rkCl(X_{1/2})^{\fA_6}=1$}.
However, the assertion is correct: it
was later obtained in~\cite[Corollary~2.10]{CheltsovPrzyjalkowskiShramov-BarBur} by a different method.
Using Theorem~\ref{theorem:Cl} we can find these ranks as well: 
indeed, one has~\mbox{$\rkCl(X_{1/2})^{\fA_6} = \rkCl(X_{1/2})^{\fA_{4,2}} = 1$} by Corollary~\ref{corollary:rk-1}.
\end{remark}

\begin{lemma}
\label{lemma:x16-s5}
For a non-standard subgroup $\fS_5 \subset \fS_6$ we have $\rkCl(X_{1/6})^{\fS_5} = 2$.
In particular, we have $\rk\VCl(X_{1/6})^{\fS_5} = 1$.
\end{lemma}
\begin{proof}
By~\cite[\S6]{CheltsovShramov} the quartic $X_{1/6}$ is $\fS_5$-equivariantly isomorphic away from codimension~2 to the blow up $\widehat{X}_{1/6}$ of ten lines in $\P^3$,
that form a so-called double-five configuration.
Therefore we have $\Cl(X_{1/6}) \cong \Cl(\widehat{X}_{1/6})$ as $\fS_5$-representations.
Furthermore, the group~$\fS_5$ acts transitively on this configuration of lines, hence~\mbox{$\rkCl(\widehat{X}_{1/6})^{\fS_5} = 2$}.
Since also $\rkCl(\cY)^{\fS_5} = 1$ (see Corollary~\ref{corollary:rk-1},
and keep in mind that according to Notation~\ref{notation:groups} the non-standard subgroup~\mbox{$\fS_5 \subset \fS_6$}
is denoted by~$\overline{\fS_{5}}$), it follows that~\mbox{$\rk\VCl(X_{1/6})^{\fS_5} = 1$}.
\end{proof}

Now we are ready to describe the excess class groups for the special quartics.

\begin{proposition}
\label{proposition:cl-x-special}
There are the following isomorphisms of $\fS_6$-representations:
\begin{equation*}
\VCl(X_{\frac12}) \otimes \Q \cong \rR(3,1^3),\quad
\VCl(X_{\frac16}) \otimes \Q \cong \rR(2,2,2),\quad
\VCl(X_{\frac7{10}}) \otimes \Q \cong \rR(1^6).
\end{equation*}
\end{proposition}
\begin{proof}
We replace the quartics $X_{1/2}$, $X_{1/6}$, and $X_{7/10}$ by their partial resolutions of singularities $\cX_1^{5,1}$, $\cX_{1/\sqrt{-3}}^{5,1}$, and $\cX_{3/\sqrt{5}}^{5,1}$, respectively.
Similarly to the proof of Lemma~\ref{lemma:morphism-cly-clcx}, we obtain isomorphisms of $\fA_5$-representations
\begin{equation}\label{eq:Cl-over-S}
\begin{aligned}
\Cl(\cX^{5,1}_\tau/S) \otimes \Q
&\cong \left(\Cl(\cY^{5,1}/S)\oplus \left(\Cl(\cX_\tau^{5,1})/\Cl(\cY^{5,1})\right)\right)\otimes\Q\cong
\\
& \cong R_1 \oplus (\VCl(X_t) \otimes \Q)\vert_{\fA_5},
\end{aligned}
\end{equation}
with the summand $R_1$ on the right generated by the canonical class.
Next we use the computation of Corollary~\ref{corollary:relaive-class-group}
to describe the left hand side of~\eqref{eq:Cl-over-S}.
Namely, by Corollary~\ref{corollary:relaive-class-group} the left hand side is isomorphic to
$R_1 \oplus \Ind_G^{\fA_5}(\upsilon)$ for a certain subgroup $G \subset \fA_5$ and its one-dimensional representation~$\upsilon$.
Canceling the $R_1$ summands, we obtain an isomorphism
\begin{equation*}
(\VCl(X_t) \otimes \Q)\vert_{\fA_5} \cong \Ind_G^{\fA_5}(\upsilon).
\end{equation*}
It only remains to use the description of the subgroup $G$ and its representation $\upsilon$ also provided by Corollary~\ref{corollary:relaive-class-group}.

In the case $t = 1/2$, so that $\tau = 1$, it gives
\begin{equation*}
(\VCl(X_{\frac{1}{2}}) \otimes \Q)\vert_{\fA_5} \cong \Ind_{\fA_{3,2}}^{\fA_5}(-\mathbf{1}) \cong R'_3 \oplus R''_3 \oplus R_4.
\end{equation*}
Therefore, by Lemma~\ref{lemma:restriction-s5-a5} we deduce that
$(\VCl(X_{1/2}) \otimes \Q)\vert_{\fS_5}$ is isomorphic either to~\mbox{$\rR(3,1^2) \oplus \rR(4,1)$} or to $\rR(3,1^2) \oplus \rR(2,1^3)$,
hence by Lemma~\ref{lemma:restriction-s6-s5} we have either~\mbox{$\VCl(X_{1/2}) \otimes \Q \cong \rR(4,1^2)$} or~$\VCl(X_{1/2}) \otimes \Q \cong \rR(3,1^3)$.
The first case is impossible by Lemma~\ref{lemma:Cl-standard-S4}, because by Pieri's rule the restriction
of the $\fS_6$-representation $\rR(4,1^2)$ to a standard subgroup $\fS_4$ contains a trivial subrepresentation, hence the required result.

Similarly, in the case $t = 1/6$, so that $\tau = 1/\sqrt{-3}$,
we have
\begin{equation*}
(\VCl(X_{\frac{1}{6}}) \otimes \Q)\vert_{\fA_5} \cong \Ind_{\fA_{4}}^{\fA_5}(\mathbf{1}) \cong R_1 \oplus R_4.
\end{equation*}
Therefore, by Lemma~\ref{lemma:restriction-s5-a5} we deduce that $(\VCl(X_{1/6}) \otimes \Q)\vert_{\fS_5}$ is isomorphic to the sum of
one of the representations $\rR(5)$ and $\rR(1^5)$, and one of the representations~\mbox{$\rR(4,1)$} and~\mbox{$\rR(2,1^3)$}.
On the other hand, $(\VCl(X_{1/6}) \otimes \Q)\vert_{\fS_5}$ should contain $\rR(5)$ by Lemma~\ref{lemma:x16-s5}, so it follows that
$(\VCl(X_{1/6}) \otimes \Q)\vert_{\fS_5}$ is either $\rR(5) \oplus \rR(4,1)$, or $\rR(5) \oplus \rR(2,1^3)$.
By Lemma~\ref{lemma:restriction-s6-s5} only the the first of them can be obtained as a restriction of a representation of~$\fS_6$
with respect to a non-standard embedding of $\fS_5$, and the corresponding representation of $\fS_6$ is $\rR(2,2,2)$.
Thus, we have $\VCl(X_{1/6}) \otimes \Q \cong \rR(2,2,2)$.

Finally, in the case $t = 7/10$, so that $\tau = 3/\sqrt{5}$,
we have
\begin{equation*}
(\VCl(X_{7/10}) \otimes \Q)\vert_{\fA_5} \cong R_1,
\end{equation*}
hence $\VCl(X_{7/10}) \otimes \Q$ is either $\rR(6)$ or~$\rR(1^6)$.
Again, the first case is impossible by Lemma~\ref{lemma:Cl-standard-S4}, hence the required result.
\end{proof}

Now we are ready to prove the main result of this section.

\begin{proof}[Proof of Theorem~\textup{\ref{theorem:Cl}}]
The description of $\Cl(\cY)\otimes\Q$ is given by Proposition~\ref{proposition:cl-cy}, and the descriptions of
$\Cl(\cX_\infty)\otimes\Q$ and $\Cl(X_t)\otimes\Q$ for $t\not\in\bDD\cup\{\infty\}$ follow from a combination of
Proposition~\ref{proposition:cl-cy} with Lemma~\ref{lemma:morphism-cly-clcx}.
The last three isomorphisms follow from Proposition~\ref{proposition:cl-x-special} in view of the definition of the excess class group.
\end{proof}

\begin{remark}
\label{remark:Segre}
To study $G$-equivariant  birational maps of the remaining $\fS_6$-invariant quartic $X_{1/4}$ to $G$-Mori fibre spaces,
one can replace $X_{1/4}$ by its projective dual, which is the Segre cubic~$Z$.
This may be simpler because $Z$ has terminal singularities.
The corresponding problem for $Z$ was partially solved in \cite[Theorem~1.3]{Avilov}.
In particular, if~$G$ is a standard subgroup~$\fA_5$ in $\fS_6$, then $\rkCl(Z)^G=1$ by \cite[Proposition~3.1]{Avilov},
and we expect that $Z$, and thus also $X_{1/4}$, is not $G$-rational.
In this case the induced action of $G$ on~$Z$ is also given by a standard embedding $\fA_5\cong G\hookrightarrow\mathrm{Aut}(Z)\cong\fS_6$,
see e.g.~\mbox{\cite[\S2.2]{HowardMillsonSnowdenVakil}}.
On the contrary, if $G$ is a non-standard subgroup $\fA_5$ in $\fS_6$, then $Z$ is known to be $G$-rational, see~\mbox{\cite[3.16]{Prokhorov-FieldsOfInv}}.
\end{remark}

\begin{remark}
One of the geometric interpretations of the non-trivial summands
of~\mbox{$\Cl(X_t)\otimes\Q$} that appear in Theorem~\ref{theorem:Cl} is as follows.
Suppose that $t\neq 1/4, \infty$, so that the singularities of $X_t$ are nodes by Theorem~\ref{theorem:Geer}.
Let $\nu\colon \widetilde{X}_t\to X_t$ be the blow up of all singular points of $X_t$, and
let $D_1,\ldots,D_r$ be the exceptional divisors of $\nu$.
Then~$\widetilde{X}_t$ is smooth, and $D_i \cong \P^1\times\P^1$.
Let $M_i^+$ and~$M_i^-$ be the rulings from two different families on $D_i$.
One can check that there is a natural
perfect pairing between the vector subspace in~\mbox{$H^4(\tilde{X}_t,\CC)$}
spanned by the one-cycles~\mbox{$M_i^+ - M_i^-$}
and the space~\mbox{$\big(\Cl(X_t)/\Pic(X_t)\big) \otimes \CC$.}
Note also that the structure of this subspace of~\mbox{$H^4(\tilde{X}_t,\CC)$}
as an $\fS_6$-representation can be independently deduced from~\cite[Proposition~1.3]{Schoen} and~\cite[Lemma~1]{Beauville}.
\end{remark}

\appendix

\section{Cremona--Richmond configuration}
\label{section:CR}

The {\sf Cremona--Richmond configuration} is the configuration $\CR$ of 15 lines with 15 triple intersection points in $\P^4$ formed by the singular locus of the Igusa quartic.
By a small abuse of terminology, we will sometimes say that the singular locus is the configuration~$\CR$ itself.
We refer the reader to \cite{Cremona1877}, \cite{Richmond}, and~\cite[\S9]{Dolgachev-Abstract} for  basic properties.

Explicitly, the configuration $\CR$ can be described  as follows.
Consider $\P^4$ as the hyperplane given by equation~\eqref{eq:hyperplane} in $\P^5$ with the usual $\fS_6$-action.
For each pairs-splitting
\begin{equation*}
\{1,\ldots,6\} = I_1 \sqcup I_2 \sqcup I_3,
\end{equation*}
where $|I_1| = |I_2| = |I_3| = 2$, let
$\mathrm{L}_{(I_1\,|\,I_2\,|\,I_3)}$ be the line in $\P^4$ given by equations
\begin{equation*}
\text{$x_i = x_j$ if $\{i,j\} = I_p$ for some $p \in \{1,2,3\}$}.
\end{equation*}
This gives 15 lines in $\P^4$; for instance, $\mathrm{L}_{(1,2\,|\,3,4\,|\,5,6)}$
is the line given by equations
\begin{equation}
\label{eq:cr-lines}
x_1 = x_2,\ x_3 = x_4,\ x_5 = x_6,
\end{equation}
and the other lines are obtained from this by the $\fS_6$-action.

Similarly, for every two-element subset $I \subset \{1,\ldots,6\}$ let $\mathrm{P}_{I}$ be the point in $\P^4$ given by equations
\begin{equation*}
\text{$x_i = x_j$ if either $i,j \in I$ or $i,j \in \bar{I}$},
\end{equation*}
where $\bar{I}$ is the complement of $I$ in $\{1,\ldots,6\}$.
This gives 15 points in $\P^4$; for instance,
\begin{equation}
\label{eq:cr-points}
\mathrm{P}_{1,2} = (2:2:-1:-1:-1:-1),
\end{equation}
and the other points are obtained from this by the $\fS_6$-action (so, this is the set $\Upsilon_{15}$ defined in~\S\ref{subsection:xt}).

It is easy to see that $\mathrm{P}_{I}$ lies on $\mathrm{L}_{(I_1\,|\,I_2\,|\,I_3)}$ if and only if $I = I_p$ for some $p \in \{1,2,3\}$, i.e.,
if~$I$ is one of the pairs in the pairs-splitting, or, equivalently, the pairs-splitting extends the pair $I$.
In particular, there are three lines through each of the points (corresponding to three pairs-splittings of $\bar{I}$),
and there are three points on each line (corresponding to three pairs in a pairs-splitting).
Moreover, the points $\mathrm{P}_{I}$ are the only intersection points of the lines~$\mathrm{L}_{(I_1\,|\,I_2\,|\,I_3)}$.
Because of this $\CR$ is often referred to as a {\sf $(15_3)$-configuration}.

In this section we discuss some properties of $\CR$.
In particular, in Theorem~\ref{theorem:CR-unique} we show that $\CR$ is determined uniquely up to a projective transformation of $\P^4$
by its combinatorial structure (under a mild non-degeneracy assumption),
and that the Igusa quartic is the only quartic whose singular locus contains $\CR$.

We start by a discussion of combinatorics of $\CR$.

\begin{lemma}\label{lemma:cr-selfdual}
The configuration $\CR$ is combinatorially self-dual: an outer automorphism of $\fS_6$ induces a bijection
between the set of points $\mathrm{P}_I$ and the set of lines $\mathrm{L}_{(I_1\,|\,I_2\,|\,I_3)}$ that preserves the incidence correspondence.
\end{lemma}
\begin{proof}
There is a natural bijection between subsets of cardinality two in the set $\{1,\ldots,6\}$, and transpositions in the group $\fS_6$.
Similarly, there is a natural bijection between pairs-splittings of the set $\{1,\ldots,6\}$, and elements of cycle type $[2,2,2]$ in $\fS_6$.
Let us denote the transposition corresponding to a subset $I\subset \{1,\ldots,6\}$ by $w(I)$,
and the element of cycle type $[2,2,2]$ corresponding to a pairs-splitting $(I_1,I_2,I_3)$ of $\{1,\ldots,6\}$ by $w(I_1,I_2,I_3)$.
The incidence relation of lines and points of $\CR$ can be reformulated in group-theoretic terms:
the line $\mathrm{L}_{(I_1\,|\,I_2\,|\,I_3)}$ is incident to the point $\mathrm{P}_{I}$ if and only if
the permutations $w(I)$ and~$w(I_1,I_2,I_3)$ commute (or, which is the same, the composition $w(I)\circ w(I_1,I_2,I_3)$ has cycle type $[2,2]$).

Choose an outer automorphism $\alpha$ of the group $\fS_6$.
The automorphism $\alpha$ interchanges transpositions with elements of cycle type $[2,2,2]$.
Thus $\alpha$ defines a map from the set of points of $\CR$ to the set of lines of $\CR$,
and a map from the set of lines of $\CR$ to the set
of points of $\CR$.
Moreover, this map preserves the incidence relation.
\end{proof}

Lemma~\ref{lemma:cr-selfdual} implies the following result that we used in the main part of the paper.

\begin{corollary}
\label{corollary:a5-action-upsilon15}
Every standard subgroup $\fA_5 \subset \fS_6$ acts transitively on the set of lines of~$\CR$,
and every non-standard subgroup $\fA_5 \subset \fS_6$ acts transitively on the set of points of~$\CR$.
\end{corollary}
\begin{proof}
The first assertion is evident from combinatorics, and the second assertion follows from
the first one in view of the bijection of Lemma~\ref{lemma:cr-selfdual}.
\end{proof}

The following description of $\CR$ is very useful.
Choose a triples-splitting
\begin{equation*}
\{ 1, \ldots, 6 \} = K_0 \sqcup K_1,
\qquad
|K_0| = |K_1| = 3.
\end{equation*}
For each bijection $g \colon K_0 \xrightarrow{\,\scriptstyle\sim\,} K_1$ let $\Gamma(g)$ be the pairs-splitting
formed by all pairs~\mbox{$\{k_0,g(k_0)\}$}, where $k_0$ runs through $K_0$ (and hence $g(k_0)$ runs through $K_1$).
The~6 lines and 9 points
\begin{equation*}
\left\{ \mathrm{L}_{\Gamma(g)} \right\}_{g \in \Iso(K_0,K_1)}
\qquad\text{and}\qquad
\left\{ \mathrm{P}_{k_0,k_1} \right\}_{(k_0,k_1) \in K_0 \times K_1}
\end{equation*}
form a subconfiguration $\CR'_{K_0,K_1} \subset \CR$ of the Cremona--Richmond configuration,
see Fig.~\ref{figure:jail}.
Because of its characteristic shape we call it a {\sf jail configuration}.
Note that $\CR'_{K_0,K_1}$ is contained in the hyperplane
\begin{equation}
\label{eq:jail-hyperplane}
H_{K_0} := \left\{ \sum_{k \in K_0} x_k = 0 \right\} = \left\{ \sum_{k \in K_1} x_k = 0 \right\} =: H_{K_1}.
\end{equation}
We call it the {\sf jail hyperplane}.
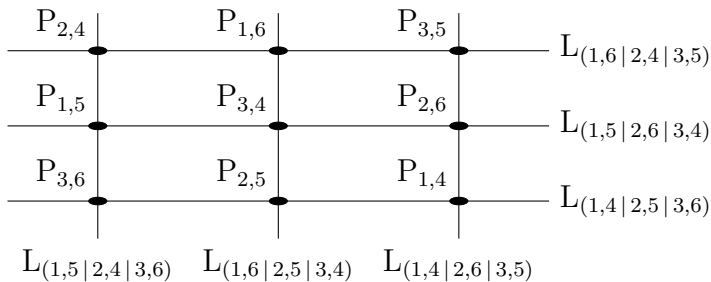
\begin{figure}[h]
\begin{tikzpicture}[xscale = 1.2, yscale = .5]
\draw ((1,0) node [below] {$\mathrm{L}_{(1,5\,|\,2,4\,|\,3,6)}$} -- (1,1) node [above left] {$\mathrm{P}_{3,6}$} -- (1,3) node [above left] {$\mathrm{P}_{1,5}$} -- (1,5) node [above left] {$\mathrm{P}_{2,4}$} -- (1,6);
\draw ((3,0) node [below] {$\mathrm{L}_{(1,6\,|\,2,5\,|\,3,4)}$} -- (3,1) node [above left] {$\mathrm{P}_{2,5}$} -- (3,3) node [above left] {$\mathrm{P}_{3,4}$} -- (3,5) node [above left] {$\mathrm{P}_{1,6}$} -- (3,6);
\draw ((5,0) node [below] {$\mathrm{L}_{(1,4\,|\,2,6\,|\,3,5)}$} -- (5,1) node [above left] {$\mathrm{P}_{1,4}$} -- (5,3) node [above left] {$\mathrm{P}_{2,6}$} -- (5,5) node [above left] {$\mathrm{P}_{3,5}$} -- (5,6);
\draw ((0,1) -- (6,1) node [right] {$\mathrm{L}_{(1,4\,|\,2,5\,|\,3,6)}$};
\draw ((0,3) -- (6,3) node [right] {$\mathrm{L}_{(1,5\,|\,2,6\,|\,3,4)}$};
\draw ((0,5) -- (6,5) node [right] {$\mathrm{L}_{(1,6\,|\,2,4\,|\,3,5)}$};
\draw[fill] (1,1) circle [radius=.1];
\draw[fill] (1,3) circle [radius=.1];
\draw[fill] (1,5) circle [radius=.1];
\draw[fill] (3,1) circle [radius=.1];
\draw[fill] (3,3) circle [radius=.1];
\draw[fill] (3,5) circle [radius=.1];
\draw[fill] (5,1) circle [radius=.1];
\draw[fill] (5,3) circle [radius=.1];
\draw[fill] (5,5) circle [radius=.1];
\end{tikzpicture}
\caption{The jail subconfiguration $\CR'_{\{1,2,3\},\{4,5,6\}}$ in the Cremona--Richmond configuration $\CR$\label{figure:jail}}
\end{figure}

The remaining 9 lines and 6 points
\begin{equation*}
\left\{ \mathrm{L}_{(k_0,k_1\,|\,K_0 \setminus k_0\,|\,K_1 \setminus k_1)} \right\}_{(k_0,k_1) \in K_0 \times K_1}
\qquad\text{and}\qquad
\left\{ \mathrm{P}_I \right\}_{I \subset K_0\ \text{or}\ I \subset K_1}
\end{equation*}
form a complete bipartite graph, see Fig.~\ref{figure:bipartite}; we call it a {\sf bipartite configuration}.
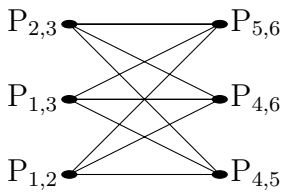
\begin{figure}[h]
\begin{tikzpicture}[xscale = 1, yscale = .5]
\draw (0,0) node [left] {$\mathrm{P}_{1,2}$} -- (2,0) node [right] {$\mathrm{P}_{4,5}$} -- (0,2) node [left] {$\mathrm{P}_{1,3}$} -- (2,2) node [right] {$\mathrm{P}_{4,6}$} -- (0,4) node [left] {$\mathrm{P}_{2,3}$} -- (2,4) node [right] {$\mathrm{P}_{5,6}$} -- (0,0) -- (2,2) -- (0,2) -- (2,4) -- (0,4) -- (2,0);
\draw[fill] (0,0) circle [radius=.1];
\draw[fill] (0,2) circle [radius=.1];
\draw[fill] (0,4) circle [radius=.1];
\draw[fill] (2,0) circle [radius=.1];
\draw[fill] (2,2) circle [radius=.1];
\draw[fill] (2,4) circle [radius=.1];
\end{tikzpicture}
\caption{The bipartite subconfiguration $\CR''_{\{1,2,3\},\{4,5,6\}}$ in the Cremona--Richmond configuration $\CR$\label{figure:bipartite}}
\end{figure}

For any decomposition
\begin{equation*}
\CR = \CR'_{K_0,K_1} \cup \CR''_{K_0,K_1}
\end{equation*}
into a jail and a bipartite subconfigurations its components interact quite weakly:
every line~\mbox{$\mathrm{L}_{(k_0,k_1\,|\,K_0 \setminus k_0\,|\,K_1 \setminus k_1)}$} from the bipartite component
passes through a single point $\mathrm{P}_{k_0,k_1}$ in the jail component.
This gives a bijection between bipartite lines and jail points (compatible with the natural bijection of both sets with $K_0 \times K_1$).

\begin{lemma}\label{lemma:cr-hyperplane}
Let $\rC$ be a configuration of $15$ lines with $15$ intersection points in $\P^4$ which is not contained in $\P^3$
and is combinatorially isomorphic to the Cremona--Richmond configuration.
If $\rC = \rC' \cup \rC''$ is a jail--bipartite decomposition then the jail component $\rC'$ spans a hyperplane,
and the bipartite component $\rC''$ spans $\P^4$.
\end{lemma}
\begin{proof}
The jail component $\rC'$ has the shape shown in Fig.~\ref{figure:jail}.
Two vertical lines do not intersect, hence they span a hyperplane $H'\subset \P^4$.
Three horizontal lines intersect each of them, hence they are contained in $H'$.
The last vertical line intersects the horizontal lines, hence it is also contained in $H'$.

The bipartite component $\rC''$ has the shape shown in
Fig.~\ref{figure:bipartite}.
Assume it is contained in a hyperplane~$H'' \subset \P^4$.
Then every line of the bipartite component is contained in $H''$.
Since every point of the jail component lies on a line of the bipartite component, it follows that the jail component  is also contained in $H''$.
Thus $\rC \subset H''$, which contradicts the assumptions of the lemma.
\end{proof}

\begin{remark}\label{remark:jb-decompositions}
The set $\{1,2,3,4,5,6\}$ has 10 distinct triples-splittings, giving rise to
10 distinct jail-bipartite decompositions of the Cremona--Richmond configuration.
The 10 hyperplanes supporting the jail components of $\CR$ appeared in Remark~\ref{remark:rho-4-2-Weil-divisor}.
\end{remark}

\begin{theorem}\label{theorem:CR-unique}
Let $\rC$ be a configuration of $15$ lines with $15$ intersection points in $\P^4$ which is not contained in $\P^3$ and is combinatorially isomorphic to the Cremona--Richmond configuration.
Then it is projectively isomorphic to the Cremona--Richmond configuration.
\end{theorem}
\begin{proof}
Choose a jail-bipartite decomposition $\rC = \rC' \cup \rC''$.
Choose five points $P_1,\ldots,P_5$ in the bipartite component $\rC''$ that are not contained in a hyperplane (this is possible by Lemma~\ref{lemma:cr-hyperplane}),
and let $H'$ be the hyperplane containing the jail component~$\rC'$.
Note that~$P_i \not\in H'$ for all $i$.
Indeed, if $P_i \in H'$ then every line of the bipartite component passing through $P_i$ would be contained in~$H'$ (since it also contains a point of the jail component),
hence the three points of~$\rC''$ that are connected to $P_i$ by lines in $\rC''$ will be also contained in~$H'$.
Applying the same argument to one of these points, we would deduce that the whole bipartite component is contained in~$H'$, hence $\rC \subset H'$, which contradicts our assumptions.

Assume that the points $P_1$, $P_3$, and $P_5$ are not connected to each other by lines in~$\rC''$;
that is, they are contained in one part of the bipartite component, and $P_2$, $P_4$ are contained in the other.
Since the points $P_i$ do not lie on a hyperplane, they can be taken to points
\begin{equation}
\label{eq:Pi}
\begin{array}{c}
P_1 = (1:0:0:0:0),\quad
P_3 = (0:0:1:0:0),\quad
P_5 = (0:0:0:0:1),\\
P_2 = (0:1:0:0:0),\quad
P_4 = (0:0:0:1:0),
\end{array}
\end{equation}
of $\P^4$ by a projective transformation.
Since the hyperplane $H'$ does not pass through the points $P_i$, it can be simultaneously taken to the hyperplane
defined by the equation
\begin{equation*}
x_1 - x_2 + x_3 - x_4 + x_5 = 0.
\end{equation*}
Now for each odd $i$ and even $j$ consider the line passing through $P_i$ and $P_j$.
By assumption it belongs to the bipartite component $\rC''$.
The intersection points of these lines with $H'$ are the following six points
\begin{equation}
\label{eq:Pij}
\begin{array}{c}
P_{12} = (1:1:0:0:0),\quad
P_{32} = (0:1:1:0:0),\quad
P_{52} = (0:1:0:0:1),\\
P_{14} = (1:0:0:1:0),\quad
P_{34} = (0:0:1:1:0),\quad
P_{54} = (0:0:0:1:1).
\end{array}
\end{equation}
It follows that $P_{ij}$ are points of the jail component~$\rC'$.
Consequently, the following six lines belong to the jail component $\rC'$:
\begin{equation*}
\begin{array}{c}
\langle P_{12},P_{34} \rangle = \{ x_1 - x_2 = x_3 - x_4 = x_5 = 0 \},\quad
\langle P_{12},P_{54} \rangle = \{ x_1 - x_2 = x_5 - x_4 = x_3 = 0 \},\\
\langle P_{32},P_{14} \rangle = \{ x_3 - x_2 = x_1 - x_4 = x_5 = 0 \},\quad
\langle P_{32},P_{54} \rangle = \{ x_3 - x_2 = x_5 - x_4 = x_1 = 0 \},\\
\langle P_{52},P_{14} \rangle = \{ x_5 - x_2 = x_1 - x_4 = x_3 = 0 \},\quad
\langle P_{52},P_{34} \rangle = \{ x_5 - x_2 = x_3 - x_4 = x_1 = 0 \},\\
\end{array}
\end{equation*}
and their three extra intersection points
\begin{equation}
\label{eq:Pijkl}
P_{1234} = (1:1:1:1:0),\quad
P_{1245} = (1:1:0:1:1),\quad
P_{2345} = (0:1:1:1:1)
\end{equation}
also belong to $\rC'$.
Finally, the last point $P_0$ of the bipartite component is the point
\begin{equation}
\label{eq:P0}
P_0 = \langle P_1, P_{2345} \rangle  \cap \langle P_3, P_{1245} \rangle \cap \langle P_5, P_{1234} \rangle  = (1:1:1:1:1).
\end{equation}
This proves that such configuration is unique up to a projective transformation.
The explicit transformation from $\P^4$ to $\P^5$ that takes the points~\eqref{eq:Pi}, \eqref{eq:Pij}, \eqref{eq:Pijkl}, and~\eqref{eq:P0}
to the points~$\mathrm{P}_{i,j}$ that were defined in~\eqref{eq:cr-points} is given by the matrix
\begin{equation*}
\begin{pmatrix}
\hphantom{-}1 & \hphantom{-}1 & -2 & \hphantom{-}1 & -2 \\
-2 & \hphantom{-}1 & \hphantom{-}1 & \hphantom{-}1 & -2 \\
-2 & \hphantom{-}1 & -2 & \hphantom{-}1 & \hphantom{-}1 \\
\hphantom{-}1 & -2 & \hphantom{-}1 &  \hphantom{-}1 & \hphantom{-}1 \\
\hphantom{-}1 & -2 & \hphantom{-}1 & -2 & \hphantom{-}1 \\
\hphantom{-}1 &  \hphantom{-}1 & \hphantom{-}1 & -2 & \hphantom{-}1 \\
\end{pmatrix};
\end{equation*}
in particular, the point $P_5$ is mapped to the point~$\mathrm{P}_{1,2}$ in~\eqref{eq:cr-points}.
This completes the proof of Theorem~\ref{theorem:CR-unique}.
\end{proof}

\begin{remark}
Let $\rC$ be a configuration combinatorially isomorphic to $\CR$. Then one can always project $\rC$ isomorphically to $\P^3$.
In particular, the assumption of Theorem~\ref{theorem:CR-unique} requiring that the configuration is not contained
in $\P^3$ is necessary.
\end{remark}

\begin{corollary}\label{corollary:Igusa-unique}
Let $\rC$ be a configuration of $15$ lines with $15$ intersection points in $\P^4$ which is not contained in~$\P^3$ and is combinatorially isomorphic to the Cremona--Richmond configuration.
Suppose that $X$ is a quartic threefold that contains~$\rC$ in its singular locus.
Then it is projectively isomorphic to the Igusa quartic.
\end{corollary}
\begin{proof}
By Theorem~\ref{theorem:CR-unique} it is enough to show that the Igusa quartic $X$ is the unique quartic singular along~$\rC$.
Suppose that $X'$ is another quartic with this property.
Since $X$ is irreducible, the intersection $Z=X\cap X'$ is two-dimensional, and $\deg Z = 16$.
Let $\rC'$ be one of the jail subconfigurations of~$\rC$.
Then $\rC'$ is contained in a unique two-dimensional smooth quadric~$T$;
this quadric is swept out by lines that meet three of the lines in $\rC'$.
The lines of $\rC'$ are singular both on $X$ and $X'$, so we conclude that $T$ is contained in~$Z$.
It remains to notice that $\rC$ contains $10$ jail subconfigurations, all of them giving rise to different two-dimensional quadrics contained in $Z$.
The degree of the union of these quadrics is $20$; this is greater than $\deg Z$, which gives a contradiction.
\end{proof}


\newcommand{\etalchar}[1]{$^{#1}$}

\end{document}